\documentclass[11pt,reqno]{amsart}
\usepackage{hyperref}
\hypersetup{hypertexnames=false}
\usepackage{mathrsfs}
\usepackage{amsfonts}
\usepackage{amssymb}
\usepackage{amsthm}
\usepackage{amsmath}
\usepackage{tikz}
\usepackage{geometry}
\usepackage{tabu}
\usepackage{float}
\usepackage{color}
\usepackage{cite}
\usepackage{mathtools, xcolor}
\usepackage{titletoc}
\usepackage{xcolor}
\allowdisplaybreaks

\newtheorem{theorem}{Theorem}[section]
\newtheorem{proposition}[theorem]{Proposition}
\newtheorem{lemma}[theorem]{Lemma}

\newtheorem{remark}[theorem]{Remark}

\makeatletter \@addtoreset{equation}{section} \makeatother

\def\p{\partial}

\def\tilde{\widetilde}

\newcommand{\beq}{\begin{equation}}
\newcommand{\eeq}{\end{equation}}

\newcommand{\Rmnum}[1]{\expandafter\@slowromancap\romannumeral #1@}

\makeatletter
\renewcommand{\subjclass}[2][2020]{\def\@subjclass{#2}}

\makeatother
\begin{document}

\title[2D compressible MHD system]
{Global well-posedness of  2D non-resistive compressible MHD system}
\author[Yi Zhu]{Yi Zhu}

\address{Department of Mathematics, East China University of Science and Technology, Shanghai 200237,  P.R. China}

\email{zhuyim@ecust.edu.cn}

\date{}
\subjclass[2020]{35Q35, 35A01, 35A02, 76W05}
\keywords{ Background magnetic field; Compressible fluids; Non-resistive MHD system; Global classical solutions}

\begin{abstract}
This paper investigates the non-resistive compressible magnetohydrodynamic (MHD) equations in $\mathbb{R}^2$. We establish the global existence and stability of classical solutions for initial data sufficiently close to a constant equilibrium state. A distinguishing feature of our result is that global stability is derived solely from pure $H^s$ energy estimates and an intrinsic $L^2$ time-decay mechanism, thereby bypassing the traditional requirement for the initial data of $L^1$ integrability or negative-order Sobolev norm regularity.
To achieve this goal, we first introduce a specific quantity motivated by the effective viscous flux, which intrinsically couples the density and magnetic field perturbations.
Secondly, to overcome the critical time-decay obstacle arising from the absence of negative-order regularity, we develop a novel pseudo-negative-derivative technique. Moreover, we regard the wildest nonlinear term as a whole and bypass the need to obtain time-decay estimate for individual components.
These approaches enable us to close the higher-order energy estimate entirely within standard Sobolev spaces.
\end{abstract}

\maketitle

\titlecontents{section}[1.5em]{\vspace{0.5ex}}{\contentslabel{1.5em}}{}{\titlerule*[0.8pc]{.}\contentspage}
\titlecontents{subsection}[3em]{\vspace{0.5ex}}{\contentslabel{1.5em}}{}{\titlerule*[0.8pc]{.}\contentspage}
\tableofcontents

\section{Introduction}

\subsection{Background and motivation}
This paper focuses on the two-dimensional (2D) viscous and non-resistive compressible magnetohydrodynamic (MHD) system in $\mathbb{R}^2$\textcolor{blue}{:}
\begin{equation}\label{mhd0}
	\begin{cases}
		\tilde{\rho}_t +\nabla\cdot(\tilde{\rho}u)=0,  \quad\quad  (t, x) \in \mathbb{R}^+ \times \mathbb{R}^2,\\
		\tilde{\rho}u_t + \tilde{\rho}u\cdot\nabla u - \mu \Delta u-\lambda \nabla\nabla\cdot u + \nabla P(\tilde{\rho}) = B\cdot\nabla B - \frac{1}{2} \nabla |B|^2,\\
		B_t+u\cdot\nabla B +B\nabla\cdot u=B\cdot\nabla u,\\
		\nabla\cdot B=0,\\
		\tilde{\rho}(0,x)=\tilde{\rho}_0(x), \qquad u(0,x)=u_0(x), \qquad B(0,x)=B_0(x),
	\end{cases}
\end{equation}
where $\tilde{\rho} = \tilde{\rho} (t,x) \in \mathbb{R}^{+}$ represents the fluid density,  $u = u(t,x)$ and $B =B(t,x)$ denote the velocity and the magnetic field, respectively. $P=P(\tilde{\rho}) \in \mathbb{R}^{+}$ is the scalar pressure\textcolor{blue}{,} which is a strictly increasing function of $\tilde{\rho}$. Here, $\lambda$ and $\mu$ are the viscosity coefficients satisfying
$$ \mu >0, \qquad \lambda +\mu >0.$$

The main objectives of this article can be elaborated from two perspectives. First, from a physical perspective, the magnetohydrodynamic equations constitute a profoundly important class of models. The MHD systems reflect the fundamental physical laws governing the motion of electrically conducting fluids, such as plasmas, liquid metals, and electrolytes.
The velocity field obeys the Navier-Stokes equations with the Lorentz force generated by the magnetic field, while the magnetic field satisfies Maxwell's equations of electromagnetism. The MHD equations have played key roles in the study of numerous phenomena in geophysics, astrophysics, cosmology, and engineering (see, e.g., \cite{MaBe, Tao, Alf, Bis, Davi, Pri}).
Our first goal is to understand the well-posedness and stability of perturbations near a background magnetic field. This study is partially motivated by the remarkable stabilizing phenomenon observed in physical experiments on electrically conducting fluids. These experiments have revealed that a background magnetic field can effectively stabilize MHD flows (see, e.g., \cite{Alex, Bur, Davi0, Davi1, Davi, Gall, Gall2}).
 Although similar properties have been extensively studied in the incompressible case, many aspects remain unexplored in compressible MHD systems.
 We aim to fully understand the underlying mechanisms and rigorously establish the corresponding stabilizing phenomena as mathematically rigorous results for the compressible MHD equations.

On the other hand, the most essential and important objective of this article is to mathematically address the following question:
 {\bf How can we weaken the requirement on the initial data space and the $L^1$ time-decay integrability while ensuring the global well-posedness for the concerned compressible MHD system? }

To better contextualize this core objective, we first briefly review related previous works.
Focusing on the Cauchy problem for the viscous and non-resistive compressible MHD system \eqref{mhd0}, assuming that the magnetic field is perturbed near a non-trivial equilibrium state and the density is perturbed around  $\bar \rho_0 = 1$, Wu and Wu \cite{WuWu} were the first to present a systematic approach to the small-data global existence and stability problem.
The authors in \cite{WuWu} exploited the extra stabilizing effects by converting the governing system into a system of wave equations and employing extensive Fourier analysis. Due to the limitations of their method, the initial data in \cite{WuWu} requires the following smallness:
\begin{equation}\label{wuwu1}
  \| \langle\nabla\rangle^M (n_0,u_0,\nabla \psi_0)\|_{L^2_{xy}} +  \| \langle\nabla\rangle^5 (n_0,u_0,\nabla \psi_0)\|_{L^1_{xy}} \lesssim \delta,
\end{equation}
which then ensure the corresponding decay estimates\textcolor{blue}{:}
\begin{equation}\label{wuwu2}
  \|n(t)\|_{L_{xy}^{\infty}} \lesssim \delta t^{-\frac{1}{2}}; \quad \|u(t)\|_{L_{xy}^{\infty}} \lesssim \delta t^{-1};\quad \|\nabla\psi(t)\|_{L_{xy}^{\infty}} \lesssim \delta t^{-\frac{1}{2}}.
\end{equation}
Naturally, to ensure the global existence of classical solutions, one expects all variables to exhibit better dissipation, as seen in \eqref{wuwu2}.
It is also natural that a stronger space for the initial data yields better dissipation, as required in \eqref{wuwu1}.
In short, the known results indicate that the $L^1(\mathbb{R}^2)$ norm requirement on the initial data ensures the $t^{-1}$ time-decay rate for the velocity field, which ultimately guarantees the global well-posedness of the system under consideration.

However, it remains unknown whether a global classical solution exists for the concerned compressible MHD system when the initial data are merely small in the traditional Sobolev  space $H^s(\mathbb{R}^2)$ (for some $s > 0$),  without any requirements on negative-order Sobolev norms or $L^1(\mathbb{R}^2)$ space requirement. This paper provides a positive answer to this question.
Thus, the motivation behind this article is clear.  {\bf Our goal is to determine the largest possible initial data space\textcolor{blue}{---}particularly by relaxing the assumptions on the low-frequency component of the initial data\textcolor{blue}{---}that admits global classical solutions. Equivalently, we seek the slowest admissible time-decay rate of the variables in system \eqref{mhd0} that still ensures global existence and stability.}

Through a simple energy analysis, we can see that if we merely assume the smallness of the $H^s(\mathbb{R}^2)$ (for some $s > 0$) norms, the corresponding time-decay rate reduces to  $t^{-\frac{1}{2}}$. In this scenario, obtaining global solutions becomes highly challenging.

\subsection{Formulation and main result}
Before stating our main result, we briefly review some related works.
Although this paper focuses on the compressible magnetohydrodynamic equations, it is worth noting that in recent years, the study of incompressible magnetohydrodynamic equations has attracted the attention of numerous scholars and yielded many significant results.
When the density $\tilde{\rho}$ is constant, \eqref{mhd0} reduces to the viscous and non-resistive incompressible MHD system, which has been the subject of numerous investigations (see, e.g., \cite{in1, in2, in3, CW1, DZ, in5, LXZ, in8, in9, in11, feffer, feffer2, chemin}).
The compressible MHD equations with both viscous dissipation and magnetic diffusion have also been the focus of many recent studies (see, e.g., \cite{com4, hong, com1, com2, HXJ, wuguochun}).

In contrast, relatively few results are currently available for the compressible viscous MHD equations without magnetic diffusion.
In the absence of magnetic diffusivity, the well-posedness problem (even for small initial data) as well as the stability problem near a background magnetic field becomes extremely difficult.
Hu and Lin \cite{hu} established the well-posedness for the 2D compressible MHD equations with a special class of initial data near a background magnetic field.
As mentioned previously, the work of Wu and Wu \cite{WuWu} was the first to investigate the global classical solutions of the concerned compressible MHD system on $\mathbb{R}^2$. They presented a novel and systematic approach to the well-posedness and stability problem for equations \eqref{mhd0}. They also introduced a diagonalization process to understand the spectral structure and large-time behavior.
Wu and Zhu \cite{wuzhu} investigated this system on the bounded domain $\mathbb{T}^2$ and solved the problem using pure energy estimates, which helps reduce the complexity present in other approaches.
Recently, Dong, Wu, and Zhai \cite{dong} proved the global existence of strong solutions to the $2\frac{1}{2}$-D compressible non-resistive MHD equations with small initial data.
The compressible and incompressible non-resistive MHD equations in a bounded domain were studied in \cite{tanwang}.
Jiang and Jiang \cite{JJ}  derived the stability/instability criteria for the stratified compressible magnetic Rayleigh--Taylor problem in Lagrangian coordinates for the three-dimensional case.
Jiang and Zhang \cite{jiangzhang} studied the non-resistive limit and the magnetic boundary layer of the 1D compressible MHD equations.
For more related works on the compressible MHD equations, we refer to \cite{wuzhai, liuzhang, lisun, lisun2, zhong, com3, HuangXinYan} and the references therein.
Needless to say, the references listed above represent only a very small fraction of the vast literature on this subject.

Without loss of generality, we assume that $\mu = 1$, $\lambda= 0$, and $P'(1) = 1$. It is clear that a special solution of \eqref{mhd0} is given by the zero velocity field, the constant density $1$, and the background magnetic field $B_0=e_2$, where $e_2=(0,1)$.
The perturbation $(\rho, u, b)$  around this equilibrium, with $\rho = \tilde{\rho}-1$ and $b= B - e_2$\textcolor{blue}{,} then obeys
\begin{equation}\label{mhd1}
	\begin{cases}
		\rho_t +\nabla\cdot u = - \nabla\cdot(\rho u),  \qquad\qquad \qquad\qquad \qquad \qquad (t,x) \in \mathbb{R}^+ \times \mathbb{R}^2, \\
		u_t - \Delta u  - \left(\begin{array}{c}
			\nabla^{\bot} \cdot b \\
			0 \\
		\end{array}\right) +  \frac{1}{\rho  + 1}\nabla P = - u\cdot \nabla u - \frac{\rho}{\rho + 1}  \Delta u \\
		\qquad \qquad  +   \frac{1}{\rho + 1} b\cdot\nabla b- \frac{1}{2(\rho+1)}\nabla |b|^2 - \frac{\rho}{\rho + 1}\left(
		\begin{array}{c}
			\nabla^{\bot} \cdot b \\
			0 \\
		\end{array}
		\right), \\
		b_t - \nabla^{\bot}u_1=- u\cdot\nabla b + b\cdot\nabla u - b\nabla\cdot u,\\
		\nabla\cdot b=0,
	\end{cases}
\end{equation}
where $\nabla^{\bot}=(\partial_2,-\partial_1)$. In comparison with the original MHD system, \eqref{mhd1} contains two extra terms
$
-\frac1{\rho+1} \left(\begin{array}{c}
	\nabla^{\bot} \cdot b \\
	0 \\
\end{array}\right)
$  and $-\nabla^{\bot}u_1$ on the left-hand sides of the equations for $u$ and $b$, respectively. The first term is further decomposed into the dominant linear term $ - \left(\begin{array}{c}
	\nabla^{\bot} \cdot b \\
	0 \\
\end{array}\right)$ and the small nonlinear term $- \frac{\rho}{\rho + 1}\left(
\begin{array}{c}
\nabla^{\bot} \cdot b \\
0 \\	\end{array}
\right)$ on the right-hand side of the velocity equation. These two terms, resulting from the expansion near the background magnetic field, help enhance the dissipation and stabilization in system \eqref{mhd1}.

Due to the lack of diffusion or damping mechanism in the equations for the magnetic field and density, the well-posedness and stability problem considered here appears to be extremely challenging. Standard direct approaches are insufficient to solve this problem. This paper presents several key observations and novel ideas to exploit the enhanced dissipation induced by the background magnetic field and the special coupling structures within the MHD system \eqref{mhd1}. Our main result can be stated as follows.

\begin{theorem}\label{thm1}
	Assume that the initial data $(\rho_0, u_0, b_0) \in H^3(\mathbb{R}^2)$ satisfy
	$$	\nabla \cdot b_0 = 0.
	$$
	Then there exists a sufficiently small constant $\epsilon>0$ such that, if
	\begin{equation*}
		\|\rho_0\|_{H^3} + \|u_0\|_{H^3} + \|b_0\|_{H^3} \leq \epsilon,
	\end{equation*}
	system \eqref{mhd1} admits a unique global classical solution $(\rho, u, b)$ satisfying the following uniform bound for all $t>0$:
	\begin{align*}
		\|u(t)\|_{H^3}^2 + \|b(t)\|_{H^3}^2 + \|\rho(t)\|_{H^3}^2 \leq C \epsilon^2.
    \end{align*}
Furthermore, this solution satisfies the following dissipation estimate:
\begin{equation*}
  \begin{split}
    \int_0^\infty \left(\|\nabla u(\tau)\|_{H^3}^2 + \|\partial_2 \rho(\tau)\|_{H^2}^2 + \|\partial_2 b(\tau)\|_{H^2}^2\right) \,d\tau\leq C \epsilon^2.
  \end{split}
\end{equation*}
Here, $C >0$ denotes a generic constant that may change from line to line but is independent of the solution and time.
\end{theorem}

\begin{remark}
A distinguishing feature of our result is that global stability is established solely through pure $H^s$ energy estimates and an intrinsic $L^2$ time-decay mechanism. Consequently, our global stability result does not rely on the traditional $t^{-1}$ decay rate (see \eqref{wuwu2}) required in \cite{WuWu}.
\end{remark}

\begin{remark}
Given that our assumption on the initial data excludes negative-order Sobolev norms and relies solely on the $L^2$ time-decay mechanism, closing the energy estimates becomes challenging. We must overcome the difficulties caused by the failure of the critical Sobolev embedding theorem in $\mathbb{R}^2$; for instance,
\[
\|b_1\|_{L^\infty} \not\lesssim \|b_1\|_{\dot H^1} = \|\partial_2 b\|_{L^2}.
\]
To address this, we introduce an auxiliary quantity and treat it as a single entity. It then satisfies the following estimate\textcolor{blue}{:}
\begin{equation*}
\int_0^\infty \int_{\mathbb{R}^2} |b_1|^2 |\partial_1^3 (\rho, b_2)|^2 \,dx\,d\tau \leq C \epsilon^2.
\end{equation*}
\end{remark}

\begin{remark}
Furthermore, motivated by the effective viscous flux from compressible fluid dynamics, we introduce the quantity
\begin{equation}\nonumber
\Gamma \triangleq b_2 + P(1+\rho) - P(1) + \tfrac{1}{2}|b_2|^2,
\end{equation}
which intrinsically couples the density and magnetic field perturbations. We can then establish that
\begin{equation}\nonumber
  \int_0^\infty \|\nabla \Gamma(\tau)\|_{L^2}^2\; d\tau \leq C \epsilon^2.
\end{equation}
\end{remark}

\vskip .1in

\subsection{Enhanced dissipative structure and construction of energy functional}\label{subsection-dissipative-structure}
In this subsection, we outline the key ideas and structural observations underlying the proof of Theorem \ref{thm1}. The analysis is nontrivial due to the absence of explicit diffusion or damping mechanisms in the equations for the density perturbation $\rho$ and the magnetic field perturbation $b$. To establish global stability, one must carefully exploit enhanced dissipation mechanisms to counteract the potential growth of these undamped components. Our approach relies on the following five key observations.
Each observation addresses a core mathematical difficulty, which we will elaborate on alongside the corresponding observation.

$\bullet$ \textbf{Directional dissipation of $\partial_2 b$ induced by the background field $\mathbf{B}_0 = \mathbf{e}_2$.}
The absence of magnetic diffusion renders the well-posedness problem exceptionally challenging.
However, perturbations around the constant background magnetic field generate extra dissipation in the $x_2$-direction, mirroring the stabilizing effect observed in physical experiments. Mathematically, linearizing \eqref{mhd1} around the equilibrium yields:
\[
\begin{cases}
    \partial_t \rho = -\nabla\cdot u, \\
    \partial_t u = \Delta u - \nabla \rho + \begin{pmatrix} \nabla^\perp \cdot b \\ 0 \end{pmatrix}, \\
    \partial_t b = \nabla^\perp u_1,
\end{cases}
\]
where $\nabla \rho$ approximates $\nabla P(\rho+1)$ since $P'(1)=1$. Differentiating with respect to $t$ and eliminating variables leads to a system of damped wave equations\textcolor{blue}{:}
\begin{align*}
    \partial_{tt} u_1 - \Delta \partial_t u_1 - \Delta u_1 &= \partial_1 \nabla\cdot u, \\
    \partial_{tt} u_2 - \Delta \partial_t u_2 &= \partial_2 \nabla\cdot u, \\
    \partial_{tt} \rho - \Delta \partial_t \rho - \Delta \rho &= -\partial_1 \nabla^\perp \cdot b, \\
    \partial_{tt} b - \Delta \partial_t b - \Delta b &= -\partial_1 \nabla^\perp \rho.
\end{align*}
Furthermore, the equation for $b$ can be reduced to a fourth-order partially damped wave equation\textcolor{blue}{:}
\[
    (\partial_{tt} - \Delta \partial_t - \Delta)^2 b = \partial_1^2 \Delta b.
\]
The operator identity $\Delta^2 - \partial_1^2\Delta = \partial_2^2\Delta$ explicitly reveals the regularization mechanism acting along the $x_2$-direction.

$\bullet$ \textbf{Directional dissipation of $\partial_2 \rho$.}
Due to the absence of any dissipation term in the density equation, obtaining dissipation estimate for the density is difficult.
Therefore, we seek to uncover hidden properties of the density promoted by certain favorable quantities.
A direct coupling between the linearized equations for $\rho$ and $u_2$,
\begin{equation*}
\begin{cases}
    \partial_t \rho = -\partial_2 u_2 - \partial_1 u_1, \\
    \partial_t u_2 = -\partial_2 \rho + \Delta u_2,
\end{cases}
\end{equation*}
exhibits a partially damped wave structure.
This provides dissipation specifically for $\partial_2\rho$. Although this dissipation is anisotropic, unlike the full dissipation in the compressible Navier--Stokes equations, it is nevertheless sufficient for our stability analysis.

$\bullet$ \textbf{Improved behavior of $\nabla\cdot u$ and $\partial_1 u_1$.}
Despite the compressibility of the fluid, the velocity divergence $\nabla\cdot u$ enjoys better control than the vorticity $\nabla\times u$. From the continuity equation $\nabla\cdot u = -(\partial_t + u\cdot\nabla)\rho - \rho\nabla\cdot u$, we can express $\nabla\cdot u$ in terms of time derivative terms and higher-order nonlinear terms. This allows us to close the estimates for terms such as (see Lemma \ref{lemma-divu-r2}):
\[
\int_0^t \int_{\mathbb{R}^2} \nabla\cdot u  \big(|\partial_1^3 \rho|^2 + |\partial_1^3 b_2|^2\big) \,dx\,d\tau.
\]
Moreover, the perturbation around $\mathbf{B}_0=\mathbf{e}_2$ ensures that the derivative in the $x_2$-direction exhibits stronger dissipative properties than the derivative in the $x_1$-direction.
Although $\partial_1 u_1 \neq -\partial_2 u_2$ in the compressible regime, the induction equation for $b_2$ permits the substitution:
\[
\partial_1 u_1 = \partial_t b_2 + u\cdot\nabla b_2 - b\cdot\nabla u_2 + b_2\nabla\cdot u,
\]
which skillfully transfers $\partial_1 u_1$ into controllable terms (see Lemma \ref{lemma-p1u1-r2}).
However, this observation can only handle lower-order derivatives. To control higher-order terms such as $\partial_1^3 u_1$, we introduce a refined structural quantity via the effective viscous flux in the following subsection.

$\bullet$ \textbf{Effective viscous flux adapted to the perturbed MHD system.}
For the compressible MHD equations, the standard effective viscous flux is defined by\textcolor{blue}{:}
$$ F \triangleq (2\mu + \lambda) \nabla \cdot u - \big(P(1+\rho) - P(1)\big) -\frac{1}{2}|B|^2.$$
Since the magnetic field is now perturbed near the equilibrium $e_2$ and the velocity $u$ behaves sufficiently well due to viscosity, it is natural to investigate the following scalar quantity:
\begin{equation}\label{defgamma}
\Gamma \triangleq b_2 + P(1+\rho) - P(1) + \tfrac{1}{2}|b_2|^2,
\end{equation}
which intrinsically couples the density and magnetic field perturbations. For analytical convenience, we introduce the auxiliary variable:
\begin{equation}\label{defomega}
\Omega \triangleq \nabla^\perp \cdot b - \partial_1 P - \tfrac{1}{2}\partial_1|b|^2 + b\cdot\nabla b_1 = \partial_2 b_1(1+b_2) - \partial_1 \Gamma.
\end{equation}
The momentum equation for $u_1$ then becomes\textcolor{blue}{:}
\begin{equation} \label{equ1}
\partial_t u_1 + u\cdot\nabla u_1 - \Delta u_1 - \Omega = -\frac{\rho}{\rho+1}\big(\Delta u_1 + \Omega \big).
\end{equation}
Applying the divergence operator to the momentum equation \eqref{mhd1} yields the following important elliptic equation, which is crucial for estimating the energies $\mathcal{E}_4$ and $\mathcal{E}_5$:
\begin{equation}\label{eq-div-rho+b2}
\begin{split}
 \Delta \Gamma
 =& -\nabla \cdot u_t + \Delta \nabla \cdot u - \nabla \cdot (u \cdot \nabla u) + \nabla \cdot \big( b\cdot \nabla b - \tfrac{1}{2} \nabla |b_1|^2\big) \\
&+ \nabla \cdot \Big[\frac{\rho}{\rho + 1}\Big\{\nabla P - \Delta u - b \cdot \nabla b + \tfrac{1}{2} \nabla |b|^2 - \left(
		\begin{array}{c}
			\nabla^\perp \cdot b \\
			0 \\
		\end{array}
		\right) \Big\} \Big] \\
=& -\nabla \cdot u_t + \Delta \nabla \cdot u + \mathcal{H}.
\end{split}
\end{equation}
Here\textcolor{blue}{,}
\begin{equation}\label{H}
\begin{split}
\mathcal{H} &\triangleq - \nabla \cdot (u \cdot \nabla u) + \nabla \cdot \big( b\cdot \nabla b - \tfrac{1}{2} \nabla |b_1|^2\big) \\
&\quad + \nabla \cdot \Big[\frac{\rho}{\rho + 1}\left(
		\begin{array}{c}
			-\Omega - \Delta u_1 \\
			\partial_2 P - \Delta u_2 - b \cdot \nabla b_2 + \tfrac{1}{2} \partial_2 |b|^2 \\
		\end{array}
		\right) \Big].
\end{split}
\end{equation}
To reveal the underlying dissipative mechanism, we can calculate the evolution of $\Omega$ by considering the evolution of each term as follows:
\begin{equation}\nonumber
\begin{split}
&(\nabla^{\bot} \cdot b)_t + u \cdot \nabla (\nabla^{\bot} \cdot b )- \Delta u_1 = - \partial_2 u\cdot \nabla b_1 + \partial_1 u\cdot \nabla b_2 +  \nabla^{\bot}\cdot \big(  b \cdot \nabla u - b \nabla \cdot u \big),\\
&\big(\frac{1}{2}\partial_1 |b|^2 \big)_t + \frac{1}{2} u \cdot \nabla (\partial_1 |b|^2 )= \; \frac{1}{2} \partial_1 u \cdot \nabla (|b|^2 )+ \partial_1 \big( b \cdot \nabla^{\bot} u_1 + b \cdot ( b \cdot \nabla u - b \nabla \cdot u ) \big), \\
&(b\cdot \nabla b_1)_t + u \cdot \nabla (b \cdot \nabla b_1) =  \; b \cdot \nabla (\partial_2 u_1 )+ \nabla^{\bot} u_1 \cdot \nabla b_1 - b \cdot \nabla (u \cdot \nabla b_1 )\\
& \qquad\qquad\qquad + b \cdot \nabla (b \cdot \nabla u_1 - b_1 \nabla \cdot u) + (b \cdot \nabla u - b \nabla \cdot u) \cdot \nabla b_1, \\
&(\partial_1 P)_t + u\cdot \nabla (\partial_1 P) + \partial_1 \nabla \cdot u = - \partial_1 u \cdot \nabla P - \partial_1 \big(P'(\tilde \rho) \tilde \rho - 1 \big) \nabla \cdot u.
\end{split}
\end{equation}
Without loss of generality, we have assumed that $P'(1)=1$. Combining the time derivatives of each component above and using \eqref{mhd1}, we obtain:
\begin{equation}\label{eqomega}
\partial_t \Omega + u\cdot\nabla \Omega - 2\partial_1^2 u_1 = \partial_2^2 u_1 + \partial_1\partial_2 u_2 + Q,
\end{equation}
where $Q$ collects the nonlinear interactions (the explicit form is given in \eqref{Q} below):
\begin{equation}\label{Q}
\begin{split}
Q \triangleq \; &\partial_1 u\cdot \nabla b_2 - \frac{1}{2} \partial_1 u \cdot \nabla (|b|^2 )+ \partial_1 u \cdot \nabla P + \nabla^{\bot} \cdot (b \cdot \nabla u)\\
& - \partial_2 u \cdot \nabla b_1 - \partial_1 (b \cdot \nabla^{\bot} u_1)   + b \cdot \nabla (\partial_2 u_1) + \nabla^{\bot} u_1 \cdot \nabla b_1\\
&  + b \cdot \nabla(b \cdot \nabla u_1)- \partial_1 \big( b \cdot (b \cdot \nabla u) \big)- \nabla^{\bot} \cdot (b \nabla \cdot u) + \partial_1 (|b|^2 \nabla \cdot u)  \\
& + \partial_1  \Big\{  \big(P'(\tilde \rho) \tilde \rho - 1 \big) \nabla \cdot u \Big\}  - \nabla \cdot u (b \cdot \nabla b_1) - b \cdot \nabla (b_1 \nabla \cdot u).
\end{split}
\end{equation}
Linearizing \eqref{eqomega} and \eqref{equ1} together then yields the coupled system\textcolor{blue}{:}
\[
\begin{cases}
    \partial_t u_1 = \Delta u_1 + \Omega,\\
    \partial_t \Omega = 2\partial_1^2 u_1 + \partial_2^2 u_1 + \partial_1\partial_2 u_2,
\end{cases}
\]
which exhibits a wave-like structure that provides essential regularization for $\partial_1 u_1$ and $\Omega$, enabling the closure of higher-order energy estimates.

$\bullet$ \textbf{Critical time-decay obstacle and the pseudo-negative-derivative technique.}
When dealing with the nonlinear term $\langle u\cdot\nabla b, b\rangle_{H^3}$, one may encounter the critical integral (see Lemma \ref{lemma-p1u1-r2})\textcolor{blue}{:}
\[
\int_0^t \int_{\mathbb{R}^2} |b_1|^2 |\partial_1^3 b_2|^2 \,dx\,d\tau.
\]
Since our assumptions on the initial data exclude negative-order Sobolev norms, standard $L^2$-based decay estimates are insufficient to bound $\|b_1\|_{L_t^2 L_x^\infty}$.
Indeed, the critical Sobolev embedding theorem fails in $\mathbb{R}^2$, for example,
\[
\|b_1\|_{L^\infty} \not\lesssim \|b_1\|_{\dot H^1} = \|\partial_2 b\|_{L^2}.
\]
This then prevents direct control of the above integral by the natural energy dissipation. To overcome this difficulty, we treat the problematic term as a unified structure and bound it via a pseudo-negative-derivative technique.
This leads to the introduction of the auxiliary energy\textcolor{blue}{:}
\[
\int_0^t \int_{\mathbb{R}^2} \big(|u_1|^2 + |b_1|^2\big) |\partial_1^3 (\rho, b_2)|^2 \,dx\,d\tau.
\]
The core idea can be illustrated by the linearized relation $\partial_t u_1 - \Delta u_1 = \partial_2 b_1 - \partial_1(\rho + b_2)$. By integrating the $x_2$-derivative structure, we recover\textcolor{blue}{:}
\begin{equation}\nonumber
\int_{\mathbb{R}^2} \underbrace{\int_{-\infty}^{x_2} \partial_2 b_1(x_1,y_2)\, b_1(x_1,y_2) \,dy_2}_{\text{pseudo-negative-derivative term}} |\partial_1^3 b_2|^2 \,dx = \cdots
\end{equation}

\vskip 0.1in
Thus far, we have introduced the five key observations obtained while studying the structure of concerned compressible MHD system.
The above observations are systematically incorporated into the construction of a composite energy functional. To capture the anisotropic dissipation, we define:
\begin{equation}\label{energy-set}
\begin{split}
\mathcal{E}_0(t) \triangleq& \sup_{0 \leq \tau \leq t} \big(\|u(\tau)\|_{H^3}^2 + \|b(\tau)\|_{H^3}^2 + \|\rho(\tau)\|_{H^3}^2\big) + \int_0^t \|\nabla u(\tau)\|_{H^3}^2 \,d\tau, \\
\mathcal{E}_1(t) \triangleq& \int_0^t \|\partial_2 \rho(\tau)\|_{H^2}^2 \,d\tau, \qquad
\mathcal{E}_2(t) \triangleq \int_0^t \|\partial_2 b(\tau)\|_{H^2}^2 \,d\tau, \\
\mathcal{E}_3(t) \triangleq& \int_0^t \big(\|\Omega(\tau)\|_{H^2}^2 + \|\nabla \Gamma(\tau)\|_{L^2}^2\big) \,d\tau, \\
\mathcal{E}_4(t) \triangleq& \int_0^t \int_{\mathbb{R}^2} \big(|u_1|^2 + |b_1|^2\big) |\partial_1^3 (\rho, b_2)|^2 \,dx\,d\tau, \\
\mathcal{E}_5(t) \triangleq& \int_0^t \int_{\mathbb{R}^2} \Gamma^2 |\partial_1^3 (\rho, b_2)|^2 \,dx\,d\tau, \\
\mathcal{E}_6(t) \triangleq& \int_0^t \big\|(\partial_t \rho, \partial_t b, \partial_t u, \partial_t \Omega, \partial_t \Gamma)\big\|_{L^2}^2 \,d\tau
\end{split}
\end{equation}
Here\textcolor{blue}{,} $\Omega$ and $\Gamma$ are given by \eqref{defomega} and \eqref{defgamma}\textcolor{blue}{,} respectively. Additionally, we define the total energy as\textcolor{blue}{:}
\begin{equation}\nonumber
\mathcal{E}_{\mathrm{total}}(t) \triangleq \sum_{i=0}^6 \mathcal{E}_i(t).
\end{equation}

\vskip .3in
\section{Preliminaries}

\subsection{Notations}
Throughout the paper, we work in $\mathbb{R}^2$ and denote the spatial gradient by $\nabla = (\partial_1, \partial_2)$. For a Banach space $X$ on $\mathbb{R}^2$ and $1 \leq p \leq \infty$, the mixed space-time norm is defined by
\begin{equation}\nonumber
\| f \|_{L_t^p (X)} \triangleq \big\| \| f(\tau, \cdot) \|_{X} \big\|_{L^p([0,t])}.
\end{equation}
Given two Banach spaces $X$ and $Y$, the norm for the intersection $X \cap Y$ is given by $\| f \|_{X \cap Y} \triangleq \| f \|_X + \| f \|_Y$. For vector-valued or paired functions, we adopt the shorthand $\| (f, g) \|_X \triangleq \| f \|_X + \| g \|_X$. For a multi-index $\alpha = (\alpha_1, \alpha_2) \in \mathbb{N}^2$ with $|\alpha| = \alpha_1 + \alpha_2$, we write $\partial^\alpha = \partial_1^{\alpha_1} \partial_2^{\alpha_2}$. The symbol $\nabla^k$ denotes the collection of all spatial derivatives for order $k$. For an integer $m \geq 0$, the standard and homogeneous Sobolev norms are defined as
\begin{equation}\nonumber
\| f \|_{H^m}^2 \triangleq \sum_{k=0}^m \| \nabla^k f \|_{L^2}^2, \qquad \| f \|_{\dot{H}^m}^2 \triangleq \| \nabla^m f \|_{L^2}^2.
\end{equation}
Throughout the paper, the notation $A \lesssim B$ signifies that $A \leq C B$ for some generic constant $C > 0$ which may change from line to line but is independent of the solution and time.

\subsection{Basic propositions}
This subsection presents two propositions to be used in the proof of the lemmas in the subsequent sections. The following proposition provides a powerful tool to control the triple products in terms of anisotropic upper bounds.

\begin{proposition}\label{prop-anisotropic-est}
The following inequalities hold when the right-hand sides are all bounded,
\begin{equation}\nonumber
\begin{split}
\left(\int_{\mathbb{R}^2} |f g|^2 \;dx\right)^\frac{1}{2} \lesssim \;&\|f\|_{L^2}^{\frac{1}{2}} \|\partial_1 f\|_{L^2}^{\frac{1}{2}} \|g\|_{L^2}^{\frac{1}{2}} \|\partial_2 g\|_{L^2}^{\frac{1}{2}}, \\
\|f\|_{L^\infty(\mathbb{R}^2)} \lesssim \; & \|f\|_{L^2}^{\frac{1}{4}} \|\partial_1 f\|_{L^2}^{\frac{1}{4}}\|\p_2 f\|_{L^2}^{\frac{1}{4}} \|\partial_1 \p_2 f\|_{L^2}^{\frac{1}{4}} \\
\lesssim \;& \|f\|_{H^1}^\frac{1}{2} \|\p_2 f\|_{H^1}^\frac{1}{2}.
\end{split}
\end{equation}
\end{proposition}

The proof of Proposition \ref{prop-anisotropic-est} relies on the following one-dimensional interpolation inequality, for $f \in H^1(\mathbb R)$ there holds
\begin{equation}\nonumber
\|f\|_{L^\infty (\mathbb{R})} \leq \sqrt{2} \|f\|_{L^2(\mathbb{R})}^\frac{1}{2} \|f'\|_{L^2(\mathbb{R})}^\frac{1}{2}.
\end{equation}
A detailed proof of this proposition can be found in \cite{CW1, wuzhu}.

\begin{proposition}\label{prop-nonlinear-func}
Let $f \in C^\infty(\mathbb{R})$ be a smooth function, assume that $v \in H^2(\mathbb{R}^2)$ with $\|v\|_{H^2} < 1$, then the following inequalities hold:
\begin{equation}\label{prop11}
\begin{split}
\|f(v)\|_{L^\infty} &\lesssim |f(0)| + \|v\|_{L^\infty}, \quad \text{if } f(0) \neq 0,\\
\|f(v)\|_{L^{p}} &\lesssim \|v\|_{L^{p}}, \quad \forall\, 1 \leq p \leq \infty, \quad \text{if } f(0) = 0,\\
\|f(v)\|_{L^{p}} &\lesssim \|v^2\|_{L^{p}}, \quad\forall\, 1 \leq p \leq \infty, \quad \text{if } f(0) = f'(0) = 0,\\
\|\nabla^k f(v)\|_{L^2} &\lesssim \|\nabla^k v\|_{L^2}, \quad k=1,2.
\end{split}
\end{equation}
Similarly, let $g \in C^\infty(\mathbb{R}^2)$ be a smooth function, assume $v, w \in H^2(\mathbb{R}^2)$ with
$$\|v\|_{H^2}, \|w\|_{H^2} < 1,$$
then there holds,
\begin{equation*}
\begin{split}
\|g(v, w)\|_{L^\infty} &\lesssim |g(0,0)| + \|v\|_{L^\infty} + \|w\|_{L^\infty}, \quad \text{if } g(0,0) \neq 0,\\
\|g(v, w)\|_{L^{p}} &\lesssim \|v\|_{L^{p}} + \|w\|_{L^{p}}, \quad\forall\, 1 \leq p \leq \infty, \quad \text{if } g(0,0) = 0,\\
\|g(v, w)\|_{L^{p}} &\lesssim \|v^2\|_{L^{p}} + \|w^2\|_{L^{p}}, \;\;\forall\, 1 \leq p \leq \infty, \;\;\text{if } g(0,0) = \partial_1 g(0,0) = \partial_2 g(0,0) = 0,\\
\|\nabla^k g(v, w)\|_{L^2} &\lesssim \|\nabla^k v\|_{L^2} + \|\nabla^k w\|_{L^2}, \quad k=1,2.
\end{split}
\end{equation*}
\end{proposition}

\begin{proof}
We restrict our attention to the demonstration of \eqref{prop11}, as the estimates for the two-variable function follow by applying identical arguments to the respective partial derivatives. Since $\|v\|_{L^\infty} \lesssim \|v\|_{H^2}$ and $f$ possesses sufficient smoothness, the compositions $f(v(x))$, $f'(v(x))$, and $f''(v(x))$ remain uniformly bounded for all $x \in \mathbb{R}^2$. Invoking the mean value theorem, we deduce
$$f(v) = f(0) + f'(r_1(v)) v \quad \text{and} \quad f(v) = f(0) + f'(0)v + \frac{f''(r_2(v))}{2!}v^2$$
for some intermediate values satisfying $|r_1(v)|, |r_2(v)| \lesssim \|v\|_{L^\infty}$. Consequently, we obtain the following estimates:
\begin{itemize}
\item If $f(0) = 0$ and $f'(0) = 0$:
\begin{equation}\nonumber
\|f(v)\|_{L^{p}} \lesssim \; \|v^2\|_{L^{p}}, \quad \forall \; 1 \leq p \leq \infty.
\end{equation}
\item If $f(0) = 0$:
\begin{equation}\nonumber
\|f(v)\|_{L^{p}} \lesssim \; \|v\|_{L^{p}}, \quad \forall \; 1 \leq p \leq \infty.
\end{equation}
\item If $f(0) \neq 0$:
\begin{equation}\nonumber
\|f(v)\|_{L^\infty} \lesssim \; |f(0)| + \|v\|_{L^\infty}.
\end{equation}
\end{itemize}

Turning our attention to the gradient estimate in \eqref{prop11}, we only provide the detailed analysis for the case $k = 2$, since the remaining cases can be handled analogously. An application of H\"{o}lder's inequality and the Sobolev embedding theorem readily yields
\begin{equation}\nonumber
\begin{split}
\|\nabla^2 f(v)\|_{L^2} =& \; \|\nabla (f'(v) \nabla v)\|_{L^2} \\
\lesssim& \; \| f''(v) |\nabla v|^2 \|_{L^2} +\| f'(v) \nabla^2 v\|_{L^2}\\
\lesssim& \; \|f''(v)\|_{L^\infty}\|\nabla v\|_{L^2} \|\nabla^2 v\|_{L^2} + \|f'(v)\|_{L^\infty} \|\nabla^2 v\|_{L^2} \\
\lesssim& \; \|\nabla^2 v\|_{L^2}.
\end{split}
\end{equation}
This finalizes the proof.
\end{proof}

\section{Estimates for the nonlinear terms involving $\p_1 u_1$, $\p_2 u_2$, and $\nabla \cdot u$}\label{lemma-pre}

Before embarking on the estimate for energies $\mathcal{E}_0(t)$ to $\mathcal{E}_6(t)$ defined in \eqref{energy-set}, we shall first focus on some intricate nonlinear terms. These nonlinear terms will appear frequently throughout the subsequent analytical procedures. Therefore, the lemmas established in this section are indispensable for subsequent estimates. Throughout this section, we consistently assume that
\begin{equation}\label{ansatz}
\mathcal{E}_0(t) \leq 1 \quad \text{and} \quad \|\rho\|_{L^\infty} \leq \frac{1}{2}, \qquad \text{for all } t \in [0, T].
\end{equation}
Since the proof of Theorem \ref{thm1} relies on a bootstrap argument, condition \eqref{ansatz} serves as our standard \textit{a priori} assumption.

\subsection{Lemmas for lower-order derivatives}
The following lemmas provide estimates for the nonlinear terms involving $\p_1 u_1$ and $(\rho, b_2)$, excluding any higher-order terms in $(u_1, \rho, b_2)$.

\begin{lemma}\label{lemma-p2u2}
Let $f(r, s)$ be a smooth function in the case $f(0,0)=\p_r f(0,0)=\p_s f(0,0)=0$; i.e., its Taylor expansion around the origin contains only quadratic and higher-order terms in $(r,s)$. Then, for any $r,s\in\{b_2,\rho\}$, we have the following estimate:
\begin{equation}\nonumber
\begin{split}
\Big|\int_{\mathbb{R}^2} \p_2 u_2 \, f(r,s)\,dx \Big| \lesssim \|u_2\|_{L^2}^{\frac12} \|\p_1 u_2\|_{L^2}^{\frac12} \big\|(r,s)\big\|_{L^2}^{\frac12} \big\|(\p_2 r,\p_2 s)\big\|_{L^2}^{\frac32}.
\end{split}
\end{equation}
\end{lemma}

\begin{proof}
Performing integration by parts and utilizing the chain rule, we deduce
\begin{equation}\nonumber
\begin{split}
\int_{\mathbb{R}^2} \p_2 u_2 \, f(r,s)\,dx = -\int_{\mathbb{R}^2} u_2 \bigl[\p_r f(r,s)\,\p_2 r + \p_s f(r,s)\,\p_2 s\bigr]\,dx.
\end{split}
\end{equation}
As a direct consequence, the application of H\"{o}lder's inequality and Proposition \ref{prop-anisotropic-est} reveals that
\begin{equation}\nonumber
\begin{split}
\Big|\int_{\mathbb{R}^2} \p_2 u_2 \, f(r,s)\,dx \Big| &\lesssim \|u_2\,\p_r f(r,s)\|_{L^2}\|\p_2 r\|_{L^2} + \|u_2\,\p_s f(r,s)\|_{L^2}\|\p_2 s\|_{L^2} \\
&\lesssim \|u_2\|_{L^2}^{\frac12} \|\p_1 u_2\|_{L^2}^{\frac12} \|\p_r f(r,s)\|_{L^2}^{\frac12} \|\p_2\p_r f(r,s)\|_{L^2}^{\frac12} \|\p_2 r\|_{L^2} \\
&\quad + \|u_2\|_{L^2}^{\frac12} \|\p_1 u_2\|_{L^2}^{\frac12} \|\p_s f(r,s)\|_{L^2}^{\frac12} \|\p_2\p_s f(r,s)\|_{L^2}^{\frac12} \|\p_2 s\|_{L^2}.
\end{split}
\end{equation}
It is worth noting that
\begin{equation}\nonumber
\begin{split}
\p_2\p_r f(r,s) &= \p_r^2 f(r,s)\,\p_2 r + \p_r\p_s f(r,s)\,\p_2 s,
\end{split}
\end{equation}
and in a similar vein,
\begin{equation}\nonumber
\begin{split}
\p_2\p_s f(r,s) &= \p_s^2 f(r,s)\,\p_2 s + \p_s\p_r f(r,s)\,\p_2 r.
\end{split}
\end{equation}
Substituting these identities back into the previous inequality, the desired estimate then follows directly from Proposition~\ref{prop-nonlinear-func}, which thereby completes the demonstration of this lemma.
\end{proof}

\begin{lemma}\label{lemma-p1u1}
Let $f(r, s)$ be a smooth function satisfying $f(0, 0) = \partial_{r} f(0, 0) = \partial_{s}f(0, 0) = 0$ (which indicates that the Taylor expansion of $f$ around $(0,0)$ contains only quadratic and higher-order terms). Specifically, there exists intermediate points $q_1(r,s), q_2(r,s)$ such that $f(r,s) = \frac{1}{2!} \sum_{i,j \in \{r,s\}} \partial_i \partial_j f(q_1(r, s), q_2(r, s)) \, ij.$ Then we have the following estimates:
\begin{equation}\label{prop-p1u1-b2b2}
\begin{split}
& \Big|\int_{\mathbb{R}^2} \partial_1 u_1 f(b_2, b_2)\;dx + \frac{d}{dt}\int_{\mathbb{R}^2} b_2 f(b_2,b_2)\;dx \\
& \quad - \sum_{i,j \in \{r,s\}} \frac{d}{dt}\int_{\mathbb{R}^2} \frac{\partial_i \partial_j f(q_1(b_2, b_2), q_2(b_2, b_2))}{3} b_2^3\;dx \Big| \\
\lesssim& \;\|(u, b)\|_{H^2} \big(\|\nabla u\|_{H^1}^2 + \|\partial_2 b\|_{H^2}^2 + \|\partial_t b\|_{L^2}^2\big),
\end{split}
\end{equation}
and
\begin{equation}\label{prop-p1u1-rhorho}
\begin{split}
&\Big| \int_{\mathbb{R}^2} \partial_1 u_1 f(\rho, \rho)\;dx+ \frac{d}{dt}\int_{\mathbb{R}^2} b_2 f(\rho, \rho)\;dx + \frac{d}{dt}\int_{\mathbb{R}^2} b_2 \tilde f(\rho, b_2)\;dx\\
& \quad - \frac{d}{dt}\int_{\mathbb{R}^2} \frac{\partial_s^2 \tilde f(\tilde q_1(\rho, b_2), \tilde q_2(\rho, b_2))}{3} b_2^3\;dx - \frac{1}{2}\frac{d}{dt}\int_{\mathbb{R}^2} \partial_r \partial_s \tilde f(\tilde q_1(\rho, b_2), \tilde q_2(\rho, b_2)) b_2^2 \rho\;dx \\
&\quad + \frac{1}{2}\frac{d}{dt}\int_{\mathbb{R}^2} b_2 \tilde {\!\tilde {f}}(b_2,b_2)\;dx - \sum_{i,j \in \{r,s\}} \frac{d}{dt}\int_{\mathbb{R}^2} \frac{\partial_i \partial_j \tilde {\!\tilde {f}}(\tilde {\!\tilde {q_1}}(b_2, b_2), \tilde {\!\tilde {q_2}}(b_2, b_2))}{3!} b_2^3\;dx\Big| \\
\lesssim& \; \|(\rho, u, b)\|_{H^2}\big(\|(\partial_2\rho, \nabla u, \partial_2 b)\|_{H^2}^2 + \|(\partial_t \rho, \partial_t b_2)\|_{L^2}^2\big).
\end{split}
\end{equation}
Here, $\tilde f(r,s) \triangleq -(rs+ r^2 s) \sum_{i,j \in \{r,s\}} \partial_i \partial_j f(r, s)$ and $\;\tilde {\!\tilde {f}}(r,s) \triangleq -\partial_r \partial_s \tilde f(q_1(r,s), q_2(r,s)) s^2$.
Assume further that $\partial_r^2 f(0,0) = 0$, then we arrive at
\begin{equation}\label{prop-p1u1-rhob2}
\begin{split}
&\Big|\int_{\mathbb{R}^2} \partial_1 u_1 f(\rho, b_2)\;dx + \frac{d}{dt}\int_{\mathbb{R}^2} b_2 f(\rho, b_2)\;dx - \frac{d}{dt}\int_{\mathbb{R}^2} \frac{\partial_s^2 f(q_1(\rho, b_2), q_2(\rho, b_2))}{3} b_2^3\;dx \\
&\quad - \frac{1}{2}\frac{d}{dt}\int_{\mathbb{R}^2} \partial_r \partial_s f(q_1(\rho, b_2), q_2(\rho, b_2)) b_2^2 \rho\;dx + \frac{1}{2} \frac{d}{dt}\int_{\mathbb{R}^2} b_2 \tilde {\!\tilde {f}}(b_2,b_2)\;dx \\
& \quad - \sum_{i,j \in \{r,s\}} \frac{d}{dt}\int_{\mathbb{R}^2} \frac{\partial_i \partial_j \tilde {\!\tilde {f}}(\tilde {\!\tilde {q_1}}(b_2, b_2), \tilde {\!\tilde {q_2}}(b_2, b_2))}{3!} b_2^3\;dx \Big| \\
\lesssim& \; \|(\rho, u, b)\|_{H^2}\big(\|(\partial_2\rho, \nabla u, \partial_2 b)\|_{H^2}^2 + \|(\partial_t \rho, \partial_t b_2)\|_{L^2}^2\big).
\end{split}
\end{equation}
In which we still let ${\tilde {\!\tilde {f}}}(r,s) \triangleq -\partial_r \partial_s f(q_1(r,s), q_2(r,s)) s^2$.
\end{lemma}

\begin{proof}
Recalling the equation for $b_2$,
\begin{equation}\nonumber
-\partial_1 u_1 = \partial_t b_2 + u\cdot \nabla b_2 - b \cdot \nabla u_2 + b_2 \nabla \cdot u,
\end{equation}
we can infer that for $r, s \in \{\rho, b_2\}$ there holds
\begin{equation}\nonumber
\begin{split}
\int_{\mathbb{R}^2} \partial_1 u_1 f(r, s)\;dx =& -\int_{\mathbb{R}^2} \big(\partial_t b_2 + u\cdot \nabla b_2 - b \cdot \nabla u_2 + b_2 \nabla \cdot u\big) f(r,s) \; dx \\
=& -\frac{d}{dt}\int_{\mathbb{R}^2} b_2 f(r,s)\;dx\\
& +\int_{\mathbb{R}^2} \big(b_2 \partial_t f(r,s)- u \cdot \nabla b_2 f(r,s)\big)\;dx\\
& -\int_{\mathbb{R}^2} \big( - b \cdot \nabla u_2 + b_2 \nabla \cdot u \big) f(r,s) \; dx \\
\triangleq& \; \mathcal{R}_{1} + \mathcal{R}_{2} + \mathcal{R}_{3}.
\end{split}
\end{equation}
Terms involving the total time derivative $\frac{d}{dt}$, such as $\mathcal{R}_1$, will be moved to the left-hand side of the estimates. We can now proceed to estimate $\mathcal{R}_2$ and $\mathcal{R}_3$ directly.

\paragraph{\bf Estimate of $\mathcal{R}_{3}$:}
Since $f(0, 0) = \partial_{r} f(0, 0) = \partial_{s}f(0, 0) = 0$, an immediate consequence of Proposition \ref{prop-nonlinear-func} is that
\begin{equation}\nonumber
\begin{split}
\mathcal{R}_{3} \lesssim \|\nabla u\|_{L^2} \|b\|_{L^2} \|f(r,s)\|_{L^\infty}\lesssim \|\nabla u\|_{L^2} \|b\|_{L^2} \|(r,s)\|_{L^\infty}^2.
\end{split}
\end{equation}
Furthermore, employing Proposition \ref{prop-anisotropic-est} allows us to refine this bound to
\begin{equation}\label{R3}
\begin{split}
\mathcal{R}_{3} \lesssim& \;\|\nabla u\|_{L^2} \|b\|_{L^2} \|(r,s)\|_{H^1}\|\partial_2(r,s)\|_{H^1}.
\end{split}
\end{equation}

\paragraph{{\bf Estimate of $\mathcal{R}_2$:}}
The treatment of the second term $\mathcal{R}_2$
 requires a more delicate analysis, which we will present case by case.
\subparagraph*{\bf Case $r = b_2, s = b_2$:}
The Taylor expansion of $f$ reads
$$f(b_2, b_2) = \sum_{i,j \in \{r,s\}} \frac{\partial_i \partial_j f\big(q_1(b_2, b_2), q_2(b_2, b_2)\big)}{2!} b_2^2,$$
where $|q_1(b_2, b_2)|, |q_2(b_2, b_2)| \lesssim \|b_2\|_{L^\infty}$. In light of this expansion, we rewrite $\mathcal{R}_2$ as
\begin{equation}\nonumber
\begin{split}
\mathcal{R}_2 =& \sum_{i,j \in \{r,s\}} \int_{\mathbb{R}^2} \frac{\partial_i \partial_j f(q_1, q_2)}{2!} b_2 \partial_t b_2^2\;dx \\
&+ \sum_{i,j \in \{r,s\}} \int_{\mathbb{R}^2} \frac{(\partial_i \partial_j \partial_r f)(\partial_r q_1 + \partial_s q_1) + (\partial_i \partial_j \partial_s f)(\partial_r q_2 + \partial_s q_2)}{2!} b_2^3 \partial_t b_2 \;dx \\
& + \sum_{i,j \in \{r,s\}} \int_{\mathbb{R}^2} \frac{\partial_i \partial_j f(q_1, q_2)}{3!} (\nabla \cdot u) b_2^3 \;dx \\
&+ \sum_{i,j \in \{r,s\}} \int_{\mathbb{R}^2} (u \cdot \nabla b_2) \frac{\partial_i \partial_j \partial_r f (\partial_r q_1 + \partial_s q_1) + \partial_i \partial_j \partial_s f (\partial_r q_2 + \partial_s q_2)}{3!} b_2^3 \;dx. \\
\end{split}
\end{equation}
Executing integration by parts with respect to the temporal derivative, it yields
\begin{equation}\label{R2}
\begin{split}
\mathcal{R}_2 =& \sum_{i,j \in \{r,s\}} \frac{d}{dt}\int_{\mathbb{R}^2} \frac{\partial_i \partial_j f(q_1, q_2)}{3} b_2^3\;dx \\
&+ \frac{1}{6}\sum_{i,j \in \{r,s\}} \int_{\mathbb{R}^2} \Big[(\partial_i \partial_j \partial_r f)(\partial_r q_1 + \partial_s q_1) + (\partial_i \partial_j \partial_s f)(\partial_r q_2 + \partial_s q_2)\Big] b_2^3 \partial_t b_2 \;dx \\
& + \sum_{i,j \in \{r,s\}} \int_{\mathbb{R}^2} \frac{\partial_i \partial_j f(q_1, q_2)}{3!} (\nabla \cdot u) b_2^3 \;dx \\
&+ \sum_{i,j \in \{r,s\}} \int_{\mathbb{R}^2} (u \cdot \nabla b_2) \frac{\partial_i \partial_j \partial_r f (\partial_r q_1 + \partial_s q_1) + \partial_i \partial_j \partial_s f (\partial_r q_2 + \partial_s q_2)}{3!} b_2^3 \;dx.
\end{split}
\end{equation}
Regarding the second term on the right-hand side, we can bound it as follows:
\begin{equation}\nonumber
\begin{split}
&\big(\big\|(\partial_i \partial_j \partial_r f)(\partial_r q_1 + \partial_s q_1)\big\|_{L^\infty} + \big\|(\partial_i \partial_j \partial_s f)(\partial_r q_2 + \partial_s q_2)\big\|_{L^\infty}\big)\|\partial_t b_2\|_{L^2} \|b_2\|_{L^2} \|b_2\|_{L^\infty}^2 \\
\lesssim& \; \|\partial_t b_2\|_{L^2} \|b_2\|_{L^2} \|b_2\|_{H^1}\|\partial_2 b_2\|_{H^1},
\end{split}
\end{equation}
where we have utilized the anisotropic bounds from Proposition \ref{prop-anisotropic-est}.
Moving on to the third term on the right-hand side of \eqref{R2}, a repeated application of Proposition \ref{prop-anisotropic-est} ensures that it can be easily controlled by
\begin{equation}\nonumber
\begin{split}
 \|\nabla \cdot u\|_{L^2} \|b_2\|_{L^2} \|b_2\|_{L^\infty}^2
\lesssim \|\nabla \cdot u\|_{L^2} \|b_2\|_{L^2} \|b_2\|_{H^1}\|\partial_2 b_2\|_{H^1}.
\end{split}
\end{equation}
Meanwhile, concerning the last term on the right-hand side of \eqref{R2}, following a parallel reasoning via H\"older's inequality and Proposition \ref{prop-anisotropic-est}, we obtain
\begin{equation}\nonumber
\begin{split}
& \|u \cdot \nabla b_2\|_{L^2}\|b_2\|_{L^2} \|b_2^2\|_{L^\infty} \\
& \lesssim \|u\|_{L^2}^{{\frac{1}{2}}} \|\partial_1 u\|_{L^2}^{{\frac{1}{2}}} \|\nabla b_2\|_{L^2}^{{\frac{1}{2}}} \|\partial_2 \nabla b\|_{L^2}^{{\frac{1}{2}}} \|b_2\|_{L^2}\|b_2\|_{H^1}\|\partial_2 b_2\|_{H^1} \\
& \lesssim \big(\|u\|_{L^2} + \|b_2\|_{H^1}\big)\big(\|\partial_1 u\|_{L^2}^2 + \|\partial_2 b\|_{H^1}^2\big).
\end{split}
\end{equation}
Synthesizing these individual bounds, and operating under the \textit{a priori} assumption \eqref{ansatz}, we finally conclude that
\begin{equation}\label{case1}
\begin{split}
&\Big|\mathcal{R}_2 - {\sum_{i,j \in \{r,s\}}} \frac{d}{dt}\int_{\mathbb{R}^2} \frac{\partial_i \partial_j f(q_1, q_2)}{3} b_2^3\;dx \Big| \\
\lesssim& \;\big(\|b\|_{H^2} + \|u\|_{H^2}\big)\big(\|\nabla u\|_{H^1}^2 + \|\partial_2 b\|_{H^2}^2 + \|\partial_t b_2\|_{L^2}^2\big).
\end{split}
\end{equation}

\subparagraph*{\bf Case $r = \rho, s = b_2$:}

Since $f(0,0) = \partial_r f(0,0) = \partial_s f(0,0) = \partial_r^2f(0,0) = 0$, we have
\begin{equation}\nonumber
\begin{split}
f(r,s) =& \frac{\partial_s^2 f(q_1, q_2) s^2 + 2 \partial_r \partial_s f(q_1, q_2) rs +\partial_r^2 f(q_1, q_2) r^2 }{2!} \\
=& \frac{\partial_s^2 f(q_1, q_2) s^2 + 2 \partial_r \partial_s f(q_1, q_2) rs + [\partial_r^3 f(\bar q_1, \bar q_2)q_1 +\partial_r^2 \partial_s f(\bar q_1, \bar q_2)q_2]  r^2} {2!} \\
\end{split}
\end{equation}
where $|q_1(r, s)|, |q_2(r, s)|,|\bar q_1(r, s)|, |\bar q_2(r, s)| \lesssim \max\{|r|, |s|\}$.
Accordingly, $\mathcal{R}_2$ can be reformulated as
\begin{equation}\label{R2b2rho}
\begin{split}
\mathcal{R}_2 =& \int_{\mathbb{R}^2} \frac{\partial_s^2 f(q_1, q_2)}{2!} \big(b_2 \partial_t b_2^2 - (u \cdot \nabla b_2) b_2^2\big)\;dx \\
&+ \int_{\mathbb{R}^2} \frac{\partial_r \partial_s^2 f (\partial_r q_1 \partial_t\rho+\partial_s q_1 \partial_t b_2) + \partial_s^3 f(\partial_r q_2 \partial_t \rho + \partial_s q_2 \partial_t b_2)}{2!} b_2^3 \;dx \\
& + \int_{\mathbb{R}^2} \partial_r \partial_s f(q_1, q_2) \big(b_2 { \partial_t(b_2 \rho)} - (u \cdot \nabla b_2) (b_2 \rho)\big)\;dx \\
&+ \int_{\mathbb{R}^2} \big(\partial_r^2 \partial_s f (\partial_r q_1 \partial_t\rho+\partial_s q_1 \partial_t b_2) + \partial_r \partial_s^2 f(\partial_r q_2 \partial_t \rho + \partial_s q_2 \partial_t b_2) \big) b_2^2 \rho \;dx \\
& + \int_{\mathbb{R}^2} b_2 [\partial_r^3 f(\bar q_1, \bar q_2)q_1 +\partial_r^2 \partial_s f(\bar q_1, \bar q_2)q_2] \rho \partial_t\rho\;dx \\
&+ \int_{\mathbb{R}^2} \frac{b_2 [\partial_r^4f(\bar q_1 , \bar q_2)(\partial_r \bar q_1 \partial_t \rho + \partial_s \bar q_1  \partial_t b_2) q_1 + \partial_r^3 \partial_s f(\bar q_1 , \bar q_2)(\partial_r \bar q_2 \partial_t \rho + \partial_s \bar q_2  \partial_t b_2) q_1] \rho^2}{2!}\;dx \\
&+ \int_{\mathbb{R}^2} \frac{b_2 [\partial_r^3 \partial_s f(\bar q_1 , \bar q_2)(\partial_r \bar q_1 \partial_t \rho + \partial_s \bar q_1  \partial_t b_2) q_2 + \partial_r^2 \partial_s^2 f(\bar q_1 , \bar q_2)(\partial_r \bar q_2 \partial_t \rho + \partial_s \bar q_2  \partial_t b_2) q_2] \rho^2}{2!}\;dx \\
&+ \int_{\mathbb{R}^2} \frac{b_2 [\partial_r^3f(\bar q_1 , \bar q_2)(\partial_r q_1 \partial_t \rho + \partial_s  q_1  \partial_t b_2) + \partial_r^2 \partial_s f(\bar q_1 , \bar q_2)(\partial_r  q_2 \partial_t \rho + \partial_s q_2  \partial_t b_2)] \rho^2}{2!}\;dx \\
&-\int_{\mathbb{R}^2} \frac{u \cdot \nabla b_2 [\partial_r^3 f(\bar q_1, \bar q_2)q_1 +\partial_r^2 \partial_s f(\bar q_1, \bar q_2)q_2]  \rho^2}{2!} \; dx \\
\triangleq& \;\mathcal{R}_{2,1} + \mathcal{R}_{2,2} + \mathcal{R}_{2,3} + \mathcal{R}_{2,4} + \mathcal{R}_{2,5} + \mathcal{R}_{2,6} + \mathcal{R}_{2,7} + \mathcal{R}_{2,8} + \mathcal{R}_{2,9}.
\end{split}
\end{equation}
 For $\mathcal{R}_{2,5}$, $\mathcal{R}_{2,6}$, $\mathcal{R}_{2,7}$, $\mathcal{R}_{2,8}$ and $\mathcal{R}_{2,9}$, applying H\"{o}lder's inequality and the Sobolev embedding theorem, together with the \textit{a priori} assumption \eqref{ansatz}, directly yields
\begin{equation}\nonumber
\begin{split}
  &\mathcal{R}_{2,5} + \mathcal{R}_{2,6} + \mathcal{R}_{2,7} + \mathcal{R}_{2,8} + \mathcal{R}_{2,9} \\
  \lesssim& \|b_2\|_{L^\infty} \|(q_1, q_2)\|_{L^2} \|\rho\|_{L^\infty} \|(\partial_t \rho, \partial_t b_2) \|_{L^2} + \|b_2\|_{L^2} \|(\partial_t \rho, \partial_t b_2)\|_{L^2} \|\rho\|_{L^\infty}^2 \\
  &+ \|u\|_{L_{x_1}^\infty L_{x_2}^2} \|\nabla b_2\|_{L_{x_1}^2 L_{x_2}^\infty} \|(q_1, q_2)\|_{L^2} \|\rho\|_{L^\infty}^2\\
  \lesssim& \|(\rho, b_2, u)\|_{H^2}^2 (\|(\partial_t \rho, \partial_t b_2)\|_{L^2}^2 +  \|(\p_2\rho, \p_2 b, \nabla u)\|_{H^1}^2).
\end{split}
\end{equation}

For $\mathcal{R}_{2,1}$, $\mathcal{R}_{2,2}$ and $\mathcal{R}_{2,4}$, by employing the same strategy as in \eqref{R2}, we obtain
\begin{equation}\nonumber
\begin{split}
&\Big|\mathcal{R}_{2,1} + \mathcal{R}_{2,2} + \mathcal{R}_{2,4} - \frac{d}{dt}\int_{\mathbb{R}^2} \frac{\partial_s^2 f(q_1, q_2)}{3} b_2^3\;dx \Big| \\
\lesssim& \; \|(\rho, u, b)\|_{H^2}\big(\|(\partial_2\rho, \nabla u, \partial_2 b)\|_{H^2}^2 + \|(\partial_t \rho, \partial_t b_2)\|_{L^2}^2\big).
\end{split}
\end{equation}
To proceed, we decompose $\mathcal{R}_{2,3}$ by shifting the temporal derivative:
\begin{equation}\nonumber
\begin{split}
\mathcal{R}_{2,3} =& \int_{\mathbb{R}^2} \partial_r \partial_s f(q_1, q_2) \Big( \frac{1}{2} \partial_t (b_2^2 \rho) + \frac{1}{2} b_2^2 \partial_t \rho - (u \cdot \nabla b_2) (b_2 \rho)\Big)\;dx \\
=& \;\frac{1}{2}\frac{d}{dt}\int_{\mathbb{R}^2} \partial_r \partial_s f(q_1, q_2) b_2^2 \rho\;dx \\
& - \frac{1}{2}\int_{\mathbb{R}^2} \big(\partial_r^2 \partial_s f (\partial_r q_1 \partial_t \rho + \partial_s q_1 \partial_t b_2) + \partial_r \partial_s^2 f (\partial_r q_2 \partial_t \rho + \partial_s q_2 \partial_t b_2)\big) b_2^2 \rho\;dx \\
&+ \frac{1}{2}\int_{\mathbb{R}^2} \partial_r \partial_s f(q_1, q_2) \big(b_2^2 \partial_t \rho - (u \cdot \nabla |b_2|^2) \rho\big)\;dx \\
\triangleq& \;\mathcal{R}_{2,3,1} + \mathcal{R}_{2,3,2} + {\mathcal{R}_{2,3,3}}.
\end{split}
\end{equation}
Following the analytical framework applied to the second term in \eqref{R2}, we can bound $\mathcal{R}_{2,3,2}$ by
\begin{equation}\nonumber
\begin{split}
\mathcal{R}_{2,3,2} \lesssim \|(\rho, b)\|_{H^2} \big(\|(\partial_2\rho, \partial_2b)\|_{H^2}^2 + \|(\partial_t \rho, \partial_t b_2)\|_{L^2}^2\big).
\end{split}
\end{equation}
As for $\mathcal{R}_{2,3,3}$, substituting the equation of $\rho$ and performing integration by parts, we get
\begin{equation}\nonumber
\begin{split}
\mathcal{R}_{2,3,3} =& \;\frac{1}{2}\int_{\mathbb{R}^2} \partial_r \partial_s f(q_1, q_2) \Big(b_2^2 (-\nabla \cdot u - \rho \nabla \cdot u - u \cdot \nabla \rho) - (u \cdot \nabla |b_2|^2) \rho\Big)\;dx \\
=& - \frac{1}{2}\int_{\mathbb{R}^2} \partial_r \partial_s f(q_1, q_2) \p_1 u_1 b_2^2 \;dx \\
& +\frac{1}{2}\int_{\mathbb{R}^2} \p_2 (\partial_r \partial_s f(q_1, q_2)b_2^2) u_2 \;dx + \frac{1}{2}\int_{\mathbb{R}^2} u \cdot \nabla \big(\partial_r \partial_s f(q_1, q_2)\big) |b_2|^2 \rho\;dx \\
\triangleq& \; \mathcal{R}_{2,3,3,1} + \mathcal{R}_{2,3,3,2}.
\end{split}
\end{equation}
In regard to $\mathcal{R}_{2,3,3,1}$, appealing to \eqref{prop-p1u1-b2b2} there is
\begin{equation}\nonumber
\begin{split}
& \Big|\mathcal{R}_{2,3,3,1} + \frac{1}{2}\frac{d}{dt}\int_{\mathbb{R}^2} b_2 \;\tilde {\!\tilde {f}}(b_2,b_2)\;dx \\
& \quad - \sum_{i,j \in \{r,s\}} \frac{d}{dt}\int_{\mathbb{R}^2} \frac{\partial_i \partial_j \tilde {\!\tilde {f}}\big(\;\tilde {\!\tilde {q_1}}(b_2, b_2), \tilde {\!\tilde {q_2}}(b_2, b_2)\big)}{3!} b_2^3\;dx \Big| \\
\lesssim& \;\|(u, b)\|_{H^2} \big(\|\nabla u\|_{H^1}^2 + \|\partial_2 b\|_{H^2}^2 + \|\partial_t b\|_{L^2}^2\big),
\end{split}
\end{equation}
where $\tilde {\!\tilde {f}}(r,s) \triangleq -\partial_r \partial_s f\big(q_1(r,s), q_2(r,s)\big) s^2$.
On the other hand, for $\mathcal{R}_{2,3,3,2}$, in a similar vein as the estimate of the fourth term in \eqref{R2}, we obtain
\begin{equation}\nonumber
\begin{split}
\mathcal{R}_{2,3,3,2} \lesssim \|(\rho, u, b)\|_{H^2} \|(\partial_2\rho, \nabla u, \partial_2b)\|_{H^2}^2.
\end{split}
\end{equation}
Aggregating these derived inequalities, we finally get
\begin{equation}\label{case2}
\begin{split}
&\Big|\mathcal{R}_{2}- \frac{d}{dt}\int_{\mathbb{R}^2} \frac{\partial_s^2 f(q_1, q_2)}{3} b_2^3\;dx - \frac{1}{2}\frac{d}{dt}\int_{\mathbb{R}^2} \partial_r \partial_s f(q_1, q_2) b_2^2 \rho\;dx \\
&\quad + \frac{1}{2}\frac{d}{dt}\int_{\mathbb{R}^2} b_2 \tilde {\!\tilde {f}}(b_2,b_2)\;dx - \sum_{i,j \in \{r,s\}} \frac{d}{dt}\int_{\mathbb{R}^2} \frac{\partial_i \partial_j \tilde {\!\tilde {f}}(\;\tilde {\!\tilde {q_1}}(b_2, b_2), \tilde {\!\tilde {q_2}}(b_2, b_2))}{3!} b_2^3\;dx \Big| \\
\lesssim& \, \|(\rho, u, b)\|_{H^2}\big(\|(\partial_2\rho, \nabla u, \partial_2 b)\|_{H^2}^2 + \|(\partial_t \rho, \partial_t b_2)\|_{L^2}^2\big).
\end{split}
\end{equation}

\subparagraph*{{\bf Case $r = \rho, s = \rho$:}}
Following an analogous procedure like the case $r = b_2, s = b_2$, we rewrite $\mathcal{R}_2$ as follows
\begin{equation}\label{R2rhorho}
\begin{split}
\mathcal{R}_2 =& {\sum_{i,j \in \{r,s\}}} \int_{\mathbb{R}^2} \frac{\partial_i \partial_j f(q_1(\rho, \rho), q_2(\rho, \rho))}{2!} \big(b_2 \partial_t \rho^2 - (u \cdot \nabla b_2) \rho^2\big)\;dx \\
&+{\sum_{i,j \in \{r,s\}}} \int_{\mathbb{R}^2} \frac{(\partial_i \partial_j \partial_r f)(\partial_r q_1 \partial_t \rho + \partial_s q_1 \partial_t \rho) + (\partial_i \partial_j \partial_s f)(\partial_r q_2 \partial_t \rho + \partial_s q_2 \partial_t \rho)}{2!} b_2 \rho^2 \;dx \\
\triangleq& \; \mathcal{R}_{2,1} + \mathcal{R}_{2,2}.
\end{split}
\end{equation}
For $\mathcal{R}_{2,1}$, exploiting once again the equation of $\rho$, we have
\begin{equation}\nonumber
\begin{split}
\mathcal{R}_{2,1} =& - {\sum_{i,j \in \{r,s\}}} \int_{\mathbb{R}^2} \frac{\partial_i \partial_j f(q_1, q_2)}{2!} \Big[2 b_2 \big(\nabla \cdot u + \nabla \cdot (\rho u)\big) \rho + (u \cdot \nabla b_2) \rho^2\Big]\;dx \\
= &-2{\sum_{i,j \in \{r,s\}}} \int_{\mathbb{R}^2} \partial_1 u_1 (b_2 \rho + b_2 \rho^2) \frac{\partial_i \partial_j f(q_1, q_2)}{2!}\;dx \\
& -2{\sum_{i,j \in \{r,s\}}} \int_{\mathbb{R}^2} \partial_2 u_2 (b_2 \rho + b_2 \rho^2) \frac{\partial_i \partial_j f(q_1, q_2)}{2!} \;dx\\
& - {\sum_{i,j \in \{r,s\}}} \int_{\mathbb{R}^2} \frac{\partial_i \partial_j f(q_1, q_2)}{2!} \big(b_2 u \cdot \nabla |\rho|^2+ (u \cdot \nabla b_2) \rho^2\big)\;dx \\
\triangleq& \; \mathcal{R}_{2,1,1} + \mathcal{R}_{2,1,2} + \mathcal{R}_{2,1,3}.
\end{split}
\end{equation}
By virtue of \eqref{prop-p1u1-rhob2} (notice here in fact we have proved \eqref{prop-p1u1-rhob2} already since Case $r = \rho, s = b_2$ for $\mathcal{R}_{2}$ has been proved before), we have
\begin{equation}\nonumber
\begin{split}
&\Big| \mathcal{R}_{2,1,1} + \frac{d}{dt}\int_{\mathbb{R}^2} b_2 \tilde f(\rho, b_2)\;dx - \frac{d}{dt}\int_{\mathbb{R}^2} \frac{\partial_s^2 \tilde f(\tilde q_1, \tilde q_2)}{3} b_2^3\;dx\\
&\quad - \frac{1}{2}\frac{d}{dt}\int_{\mathbb{R}^2} \partial_r \partial_s \tilde f(\tilde q_1, \tilde q_2) b_2^2 \rho\;dx + \frac{1}{2} \frac{d}{dt}\int_{\mathbb{R}^2} b_2 \tilde {\!\tilde {f}}(b_2,b_2)\;dx \\
& \quad - \sum_{i,j \in \{r,s\}} \frac{d}{dt}\int_{\mathbb{R}^2} \frac{\partial_i \partial_j \tilde {\!\tilde {f}}(\; \tilde {\!\tilde {q_1}}(b_2, b_2), \tilde {\!\tilde {q_2}}(b_2, b_2))}{3!} b_2^3\;dx\Big| \\
\lesssim& \; \|(\rho, u, b)\|_{H^2}\big(\|(\partial_2\rho, \nabla u, \partial_2 b)\|_{H^2}^2 + \|(\partial_t \rho, \partial_t b_2)\|_{L^2}^2\big).
\end{split}
\end{equation}
Here, the auxiliary functions are defined as
$$\tilde f(r,s)\triangleq -(rs+ r^2 s) \sum_{i,j \in \{r,s\}} \partial_i \partial_j f(r, s)$$
and
$$\tilde {\!\tilde {f}}(r,s) \triangleq -\partial_r \partial_s \tilde f(q_1(r,s), q_2(r,s)) s^2.$$
One can readily verify that $\tilde f(0, 0) = \partial_r \tilde f(0,0) = \partial_s \tilde f(0,0) = \partial_r^2 \tilde f(0,0) = 0.$
Consequently, invoking Lemma \ref{lemma-p2u2}, we can obtain
\begin{equation}\nonumber
\begin{split}
\mathcal{R}_{2,1,2} \lesssim& \;\|u_2\|_{L^2}^{{\frac{1}{2}}} \|\partial_1 u_2\|_{L^2}^{{\frac{1}{2}}} \|(b_2, \rho)\|_{L^2}^{{\frac{1}{2}}} {\|(\partial_2 b_2, \partial_2 \rho)\|_{L^2}^{\frac{3}{2}}}.
\end{split}
\end{equation}
Furthermore, addressing $\mathcal{R}_{2,1,3}$, integration by parts gives
\begin{equation}\nonumber
\begin{split}
\mathcal{R}_{2,1,3} =& {\sum_{i,j \in \{r,s\}}} \int_{\mathbb{R}^2} \frac{\partial_i \partial_j f(q_1, q_2)}{2!} (\nabla \cdot u) b_2 |\rho|^2\;dx \\
& + {\sum_{i,j \in \{r,s\}}} \int_{\mathbb{R}^2} u \cdot \nabla \Big(\frac{\partial_i \partial_j f(q_1, q_2)}{2!}\Big) b_2 |\rho|^2\;dx.
\end{split}
\end{equation}
Replicating the same process as was done for the third and fourth terms in \eqref{R2}, we shall derive
\begin{equation}\nonumber
\begin{split}
\mathcal{R}_{2,1,3} \lesssim& \; \|(\rho, u, b)\|_{H^2}\|(\nabla u, \partial_2\rho, \partial_2 b)\|_{H^2}^2.
\end{split}
\end{equation}

It remains to handle $\mathcal{R}_{2,2}$. Drawing upon the techniques used for the second term in \eqref{R2}, we have
\begin{equation}\nonumber
\begin{split}
\mathcal{R}_{2,2} \lesssim \|(\rho, u, b)\|_{H^2}\big(\|(\partial_2\rho, \nabla u, \partial_2 b)\|_{H^2}^2 + \|(\p_t \rho, \p_t b_2)\|_{L^2}^2\big).
\end{split}
\end{equation}
Ultimately, consolidating all the pieces yields
\begin{equation}\label{case3}
\begin{split}
&\Big| \mathcal{R}_{2} + \frac{d}{dt}\int_{\mathbb{R}^2} b_2 \tilde f(\rho, b_2)\;dx - \frac{d}{dt}\int_{\mathbb{R}^2} \frac{\partial_s^2 \tilde f(\tilde q_1, \tilde q_2)}{3} b_2^3\;dx\\
&\quad - \frac{1}{2}\frac{d}{dt}\int_{\mathbb{R}^2} \partial_r \partial_s \tilde f(\tilde q_1, \tilde q_2) b_2^2 \rho\;dx + \frac{1}{2} \frac{d}{dt}\int_{\mathbb{R}^2} b_2 \tilde {\!\tilde {f}}(b_2,b_2)\;dx \\
& \quad - \sum_{i,j \in \{r,s\}} \frac{d}{dt}\int_{\mathbb{R}^2} \frac{\partial_i \partial_j \tilde {\!\tilde {f}}(\;\tilde {\!\tilde {q_1}}(b_2, b_2), \tilde {\!\tilde {q_2}}(b_2, b_2))}{3!} b_2^3\;dx\Big| \\
\lesssim& \; \|(\rho, u, b)\|_{H^2}\big(\|(\partial_2\rho, \nabla u, \partial_2 b)\|_{H^2}^2 + \|(\partial_t \rho, \partial_t b_2)\|_{L^2}^2\big).
\end{split}
\end{equation}

By amalgamating all the three cases analyzed above, namely \eqref{case1}, \eqref{case2} and \eqref{case3}, we then obtain the comprehensive estimate for $\mathcal{R}_{2}$. Together with the estimate of $\mathcal{R}_{3}$ established in \eqref{R3}, we then finish the proof of this lemma.
\end{proof}

\subsection{Lemmas for higher-order derivatives}
The bounds established in Lemmas \ref{lemma-p2u2} and \ref{lemma-p1u1} are restricted to lower-order derivatives.
When addressing higher-order scenarios (e.g., terms containing $\p_1^3 u_1$, $\p_1^3 \rho$, $\p_1^3 b_2$), the aforementioned methodology encounters the obstacle of derivative loss. Through the ingenious introduction of the auxiliary quantities $\Omega$ and $\Gamma$, we are able to successfully retrieve the missing regularity in the following lemmas.

\begin{lemma}\label{lemma-p13u1}
Let $f \in C^\infty(\mathbb{R}^2)$. For any choice of $r, s \in \{\rho, b_2\}$ and integers $0 \leq a \leq 1$, $0 \leq l, m \leq 3$ satisfying
\begin{itemize}
\item $l + m \leq 4$ if $a = 0$,
\item $l + m \leq 3$ if $a = 1$,
\end{itemize}
the following estimate holds:
\begin{equation}\nonumber
\begin{split}
&\Big|\int_{\mathbb{R}^2} \partial_1^{3+a} u_1 \; \partial_1^l r \;\partial_1^m s \; f(\rho, b_2) \, dx - \frac{1}{2} \frac{d}{dt}\int_{\mathbb{R}^2} \partial_1^{1+a} \Omega \; \partial_1^l r \; \partial_1^m s \; f(\rho, b_2) \, dx \Big| \\
\lesssim&\;\|(\rho, u, b)\|_{H^3} \big(\|(\partial_2 \rho, \partial_2 b, \Omega)\|_{H^2}^2 + \|\nabla u\|_{H^3}^2 + \|(\p_t \rho, \p_t b_2)\|_{L^2}^2\big).
\end{split}
\end{equation}
\begin{proof}
Invoking the evolution equation of $\Omega$ in \eqref{eqomega},
\begin{equation}\nonumber
\Omega_t + u \cdot \nabla \Omega - 2 \partial_1^2 u_1 = \partial_2^2 u_1 + \partial_1\partial_2 u_2 + Q,
\end{equation}
we proceed to reformulate the target integral in the following manner
\begin{equation}\nonumber
\begin{split}
&\int_{\mathbb{R}^2} \partial_1^{3+a} u_1 \; \partial_1^l r \; \partial_1^m s \; f(\rho, b_2) \, dx \\
=& \;\frac{1}{2}\int_{\mathbb{R}^2}\partial_1^{1+a}\big(\Omega_t + u \cdot \nabla \Omega - \partial_2^2 u_1 - \partial_1\partial_2 u_2 - Q\big) \; \partial_1^l r \; \partial_1^m s\; f(\rho, b_2) \, dx \\
=& \;\frac{1}{2}\int_{\mathbb{R}^2}\partial_1^{1+a} \Omega_t \; \partial_1^l r \; \partial_1^m s \; f(\rho, b_2) \, dx \\
& + \frac{1}{2}\int_{\mathbb{R}^2}\partial_1^{1+a}\big( - \partial_2^2 u_1 - \partial_1\partial_2 u_2\big) \; \partial_1^l r \;\partial_1^m s \; f(\rho, b_2) \, dx\\
& + \frac{1}{2}\int_{\mathbb{R}^2}\partial_1^{1+a}\big( u \cdot \nabla \Omega - Q\big) \; \partial_1^l r \; \partial_1^m s \; f(\rho, b_2) \, dx \\
\triangleq& \; \mathcal{P}_1 + \mathcal{P}_2 + \mathcal{P}_3.
\end{split}
\end{equation}

\paragraph{\bf Estimate of $\mathcal{P}_1$:}
Executing integration by parts with respect to the temporal variable, we decompose $\mathcal{P}_1$ into
\begin{equation}\nonumber
\begin{split}
\mathcal{P}_1=& \;\frac{1}{2} \frac{d}{dt}\int_{\mathbb{R}^2}\partial_1^{1+a} \Omega \; \partial_1^l r \; \partial_1^m s \; f(\rho, b_2) \, dx \\
& - \frac{1}{2} \int_{\mathbb{R}^2}\partial_1^{1+a} \Omega \; \partial_1^l r_t \; \partial_1^m s \; f(\rho, b_2) \, dx \\
& - \frac{1}{2} \int_{\mathbb{R}^2}\partial_1^{1+a} \Omega \; \partial_1^l r \; \partial_1^m s_t \; f(\rho, b_2) \, dx \\
& - \frac{1}{2} \int_{\mathbb{R}^2}\partial_1^{1+a} \Omega \; \partial_1^l r \; \partial_1^m s \; \big((\partial_r f) \rho_t + (\partial_s f) \partial_tb_2\big) \, dx \\
\triangleq& \; \mathcal{P}_{1,1} + \mathcal{P}_{1,2} + \mathcal{P}_{1,3} + \mathcal{P}_{1,4}.
\end{split}
\end{equation}
\textbf{Estimate of $\mathcal{P}_{1,2}$ and $\mathcal{P}_{1,3}$:}
We restrict our analysis to the specific configuration where $r = \rho, s = b_2$. The remaining configurations can be treated in an identical fashion.
\begin{equation}\nonumber
\begin{split}
\mathcal{P}_{1,2} + \mathcal{P}_{1,3} = & - \frac{1}{2} \int_{\mathbb{R}^2}\partial_1^{1+a} \Omega \; \partial_1^l \rho_t \; \partial_1^m b_2 \; f(\rho, b_2) \, dx \\
& - \frac{1}{2} \int_{\mathbb{R}^2}\partial_1^{1+a} \Omega \; \partial_1^l \rho \; \partial_1^m \partial_t b_2 \; f(\rho, b_2) \, dx.
\end{split}
\end{equation}
Plugging in the governing equations of $\rho$ and $b_2$ yields
\begin{equation}\label{transform-derivative-loss}
\begin{split}
\mathcal{P}_{1,2} + \mathcal{P}_{1,3} =& - \frac{1}{2} \int_{\mathbb{R}^2}\partial_1^{1+a} \Omega \; \partial_1^l \big( - \nabla \cdot u - \nabla \cdot (\rho u)\big) \; \partial_1^m b_2 \; f(\rho, b_2) \, dx \\
& - \frac{1}{2} \int_{\mathbb{R}^2}\partial_1^{1+a} \Omega \; \partial_1^l \rho \; \partial_1^m \big(-\partial_1 u_1 - u \cdot \nabla b_2 + b \cdot \nabla u_2 - b_2 \nabla \cdot u\big) \; f(\rho, b_2) \, dx.
\end{split}
\end{equation}
The most intricate scenario arises when $a = 1$ with $l = 3$ or $m = 3$. To circumvent the potential derivative loss, we redistribute the highest horizontal derivative via integration by parts as demonstrated below
\begin{equation}\nonumber
\begin{split}
& - \frac{1}{2} \int_{\mathbb{R}^2}\partial_1^{2} \Omega \; u \cdot \nabla \partial_1^3 \rho \; b_2 \; f(\rho, b_2) \, dx \\
=&\; \frac{1}{2} \int_{\mathbb{R}^2}u \cdot \nabla\partial_1^{2} \Omega \; b_2 \; f(\rho, b_2) \; \partial_1^3 \rho \, dx \\
& + \frac{1}{2} \int_{\mathbb{R}^2}\partial_1^{2} \Omega \; u \cdot \nabla \big( b_2 \; f(\rho, b_2)\big) \; \partial_1^3 \rho \, dx \\
& + \frac{1}{2} \int_{\mathbb{R}^2}\partial_1^{2} \Omega \, (\nabla \cdot u) \partial_1^3 \rho \; b_2 \, f(\rho, b_2) \, dx,
\end{split}
\end{equation}
and in a parallel manner for the $b_2$ component:
\begin{equation}\nonumber
\begin{split}
& - \frac{1}{2} \int_{\mathbb{R}^2}\partial_1^{2} \Omega \, u \cdot \nabla \partial_1^3 b_2 \, \rho \,f(\rho, b_2) \, dx \\
=& \;\frac{1}{2} \int_{\mathbb{R}^2}u \cdot \nabla \partial_1^{2} \Omega \, \rho \, f(\rho, b_2) \,\partial_1^3 b_2 \, dx \\
& + \frac{1}{2} \int_{\mathbb{R}^2}\partial_1^{2} \Omega \, u \cdot \nabla \big( \rho \, f(\rho, b_2)\big) \, \partial_1^3 b_2 \, dx \\
& + \frac{1}{2} \int_{\mathbb{R}^2}\partial_1^{2} \Omega \, (\nabla \cdot u) \,\partial_1^3 b_2 \, \rho \, f(\rho, b_2) \, dx.
\end{split}
\end{equation}
Consequently, by H\"{o}lder's inequality and  Sobolev embedding theorem there is
\begin{equation}\nonumber
\begin{split}
& \mathcal{P}_{1,2} + \mathcal{P}_{1,3} + \mathfrak{P} \\
\lesssim& \;\|\Omega\|_{H^2}\big(\|\nabla u\|_{H^3} + \|u\|_{H^3}^{\frac{1}{2}} \|\partial_1 u\|_{H^3}^{\frac{1}{2}} \|\rho\|_{H^2}^{\frac{1}{2}} \|\partial_2 \rho \|_{H^2}^{\frac{1}{2}} + \|u\|_{H^3}^{\frac{1}{2}} \|\partial_1 u\|_{H^3}^{\frac{1}{2}} \|b\|_{H^2}^{\frac{1}{2}} \|\partial_2 b \|_{H^2}^{\frac{1}{2}}\big) \\
& \cdot \|(\rho, u, b)\|_{H^3} \\
\lesssim& \; \|(\rho, u, b)\|_{H^3} \big(\|(\partial_2 \rho, \partial_2 b, \Omega)\|_{H^2}^2 + \|\nabla u\|_{H^3}^2\big).
\end{split}
\end{equation}
Here,
\begin{equation}\label{frakR}
\mathfrak{P} = \begin{cases} \frac{1}{2} \int_{\mathbb{R}^2}u \cdot \nabla\partial_1^{2} \Omega \,\partial_1^m b_2 \, f(\rho, b_2) \, \partial_1^3 \rho \, dx, & l = 3,\\ \frac{1}{2} \int_{\mathbb{R}^2}u \cdot \nabla \partial_1^{2} \Omega \, \partial_1^l \rho \, f(\rho, b_2) \, \partial_1^3 b_2 \, dx, & m = 3. \end{cases}
\end{equation}
The residual term $\mathfrak{P}$ will be elegantly canceled by the contribution from $\mathcal{P}_{3,3}$ in \eqref{cancel-r}. The remaining cases follow by analogous reasoning, hence, we omit the tedious details for the sake of brevity.

\paragraph{\bf Estimate of $\mathcal{P}_{1,4}$:}
Appealing to H\"{o}lder's inequality, Sobolev embedding, and the \textit{a priori} bound \eqref{ansatz}, we readily deduce
\begin{equation}\nonumber
\begin{split}
\mathcal{P}_{1,4} \lesssim& \; \|\Omega\|_{H^2}\|(\rho, b)\|_{H^3}\|(\partial_t \rho, \partial_t b_2)\|_{L^2}.
\end{split}
\end{equation}

\paragraph{\bf Estimate of $\mathcal{P}_2$:}
Performing integration by parts in $x_2$ to transfer the vertical derivative, we recast $\mathcal{P}_2$ as
\begin{equation}\nonumber
\begin{split}
\mathcal{P}_2 =& \; \frac{1}{2}\int_{\mathbb{R}^2}\partial_1^{1+a}\big(\partial_2 u_1 + \partial_1 u_2\big) \, \partial_2 \big(\partial_1^l r \, \partial_1^m s \, f(\rho, b_2)\big) \, dx.
\end{split}
\end{equation}
When $l = 3$ or $m = 3$, we again employ integration by parts to redistribute the highest order derivative. This manipulation leads to
\begin{equation}\nonumber
\begin{split}
\mathcal{P}_2 \lesssim& \;\|\nabla u\|_{H^3} \|(\partial_2 \rho, \partial_2 b)\|_{H^2} \|(\rho, b)\|_{H^3}.
\end{split}
\end{equation}

\paragraph{\bf Estimate of $\mathcal{P}_3$:}
Recalling the definition, we have
\begin{equation}\nonumber
\begin{split}
\mathcal{P}_3 = \frac{1}{2}\int_{\mathbb{R}^2}\partial_1^{1+a}\big( u \cdot \nabla \Omega - Q\big) \, \partial_1^l r \, \partial_1^m s \, f(\rho, b_2) \, dx.
\end{split}
\end{equation}
We then divide the proof into two cases.

\textbf{Case $\mathbf{a = 0}$:}
\begin{equation}\nonumber
\begin{split}
\mathcal{P}_3 =& \; \frac{1}{2}\int_{\mathbb{R}^2}\partial_1 u \cdot \nabla \Omega \,\partial_1^l r \, \partial_1^m s \, f(\rho, b_2) \, dx \\
& + \frac{1}{2}\int_{\mathbb{R}^2}u \cdot \nabla \partial_1 \Omega \, \partial_1^l r \, \partial_1^m s \, f(\rho, b_2) \, dx \\
& - \frac{1}{2}\int_{\mathbb{R}^2}\partial_1 Q \, \partial_1^l r \,\partial_1^m s \, f(\rho, b_2) \, dx \\
\triangleq & \; \mathcal{P}_{3,1} + \mathcal{P}_{3,2} + \mathcal{P}_{3,3}.
\end{split}
\end{equation}
It is straightforward to verify that for $\mathcal{P}_{3,1}$,
\begin{equation}\nonumber
\begin{split}
\mathcal{P}_{3,1} \lesssim&\; \|\partial_1 u\|_{L^\infty} \|\nabla \Omega\|_{L^4} \|(\rho, b_2)\|_{H^{2, 4}}\|(\rho, b_2)\|_{H^3} \|f(\rho, b_2)\|_{L^\infty} \\
\lesssim& \;\|(\rho, b_2)\|_{H^3}^2\|\partial_1 u\|_{H^2}\|\nabla \Omega\|_{H^1}.
\end{split}
\end{equation}
For $\mathcal{P}_{3,2}$, exploiting anisotropic type interpolation inequalities we derive
\begin{equation}\nonumber
\begin{split}
\mathcal{P}_{3,2} \lesssim& \;\|u\|_{L^\infty}\|\nabla^2 \Omega\|_{L^2} \Big(\|(\rho, b_2)\|_{H^{1, \infty}}\|(\rho, b_2)\|_{H^3} \\
& +\sum_{k_1 = 0}^2 \sum_{k_2 = 0}^2\|\partial_1^{k_1}(\rho, b_2)\|_{L_{x_1}^\infty L_{x_2}^2}\|\partial_1^{k_2}(\rho, b_2)\|_{L_{x_1}^2 L_{x_2}^\infty}\Big) \|f(\rho, b_2)\|_{L^\infty} \\
\lesssim& \; \|u\|_{L^2}^{\frac{1}{2}} \|\nabla^2 u\|_{L^2}^{\frac{1}{2}} \|\nabla^2 \Omega\|_{L^2} \|(\rho, b_2)\|_{H^2}^{\frac{1}{2}} \|\partial_2(\rho, b_2)\|_{H^2}^{\frac{1}{2}} \|(\rho, b_2)\|_{H^3}\\
& + \|u\|_{L^2}^{\frac{1}{2}} \|\nabla^2 u\|_{L^2}^{\frac{1}{2}} \|\nabla^2 \Omega\|_{L^2}\|(\rho, b_2)\|_{H^2}\|\partial_1(\rho, b_2)\|_{H^2}^{\frac{1}{2}}\|\partial_2(\rho, b_2)\|_{H^2}^{\frac{1}{2}} \\
\lesssim& \; \|(\rho, b_2, u)\|_{H^3}^2\|\nabla^2 u\|_{L^2}^{\frac{1}{2}} \|\nabla^2 \Omega\|_{L^2}\|\partial_2(\rho, b_2)\|_{H^2}^{\frac{1}{2}} .
\end{split}
\end{equation}
Turning our attention to $\mathcal{P}_{3,3}$, invoking the definition of $Q$ in \eqref{Q} and proceeding in a similar vein as above, we obtain
\begin{equation}\nonumber
\begin{split}
\mathcal{P}_{3,3} \lesssim& \; \|(\rho, b_2)\|_{H^3}^{\frac{3}{2}} \|\partial_1 Q \|_{L^2}\|\partial_2(\rho, b_2)\|_{H^2}^{\frac{1}{2}} \\
\lesssim& \;\|(\rho, b_2)\|_{H^3}^{\frac{3}{2}}\|(\rho, b_2)\|_{H^3}^{\frac{1}{2}}\|\partial_2 (\rho, b_2)\|_{H^2}^{\frac{1}{2}}\|\nabla u\|_{H^3}\|\partial_2(\rho, b_2)\|_{H^2}^{\frac{1}{2}} \\
\lesssim& \; \|(\rho, b_2)\|_{H^3}^2 \|\partial_2 (\rho, b_2)\|_{H^2}\|\nabla u\|_{H^3}.
\end{split}
\end{equation}
Aggregating these contributions ultimately furnishes
\begin{equation}\nonumber
\mathcal{P}_3 \lesssim \|(\rho, u, b_2)\|_{H^3}^2 \big(\| (\partial_2\rho, \partial_2 b_2)\|_{H^2}^2 + \|\nabla u\|_{H^3}^2 + \|\nabla \Omega\|_{H^1}^2\big).
\end{equation}

\textbf{Case $\mathbf{a = 1}$:}
This scenario presents additional complexities due to the potential derivative loss from $u \cdot \nabla \partial_1^2 \Omega$. Explicit cancellation is therefore imperative.
\begin{equation}\label{cancel-r}
\begin{split}
\mathcal{P}_3 =&\; \frac{1}{2}\int_{\mathbb{R}^2}\partial_1^2 u \cdot \nabla \Omega \, \partial_1^l r \, \partial_1^m s \, f(\rho, b_2) \, dx \\
& + \int_{\mathbb{R}^2}\partial_1 u \cdot \nabla \partial_1 \Omega \, \partial_1^l r \, \partial_1^m s \, f(\rho, b_2) \, dx \\
& + \frac{1}{2}\int_{\mathbb{R}^2}u \cdot \nabla \partial_1^2 \Omega \, \partial_1^l r \, \partial_1^m s \, f(\rho, b_2) \, dx \\
& - \frac{1}{2}\int_{\mathbb{R}^2}\partial_1^2 Q \, \partial_1^l r \, \partial_1^m s \, f(\rho, b_2) \, dx \\
\triangleq & \; \mathcal{P}_{3,1} + \mathcal{P}_{3,2} + \mathcal{P}_{3,3} + \mathcal{P}_{3,4}.
\end{split}
\end{equation}
The estimates for $\mathcal{P}_{3,1}$, $\mathcal{P}_{3,2}$, and $\mathcal{P}_{3,4}$ proceed identically to the $a=0$ case, yielding
\begin{equation}\nonumber
\begin{split}
\mathcal{P}_{3,1} + \mathcal{P}_{3,2} + \mathcal{P}_{3,4} \lesssim \|(\rho, u, b_2)\|_{H^3}^2 \big(\|\partial_2 (\rho, b_2)\|_{H^2}^2 + \|\nabla u\|_{H^3}^2 + \|\nabla \Omega\|_{H^1}^2\big).
\end{split}
\end{equation}
We now focus on the critical term $\mathcal{P}_{3,3}$. If $l < 3$ and $m < 3$, integration by parts transfers one $\partial_1$ derivative away from $\Omega$,
\begin{equation}\nonumber
\begin{split}
\mathcal{P}_{3,3} =& - \frac{1}{2}\int_{\mathbb{R}^2}\partial_1 u \cdot \nabla \partial_1 \Omega \, \partial_1^l r \, \partial_1^m s \, f(\rho, b_2) \, dx \\
& - \frac{1}{2}\int_{\mathbb{R}^2}u \cdot \nabla \partial_1 \Omega \, \partial_1^{l+1} r \, \partial_1^m s \, f(\rho, b_2) \, dx \\
& - \frac{1}{2}\int_{\mathbb{R}^2}u \cdot \nabla \partial_1 \Omega \, \partial_1^l r \, \partial_1^{m+1} s \, f(\rho, b_2) \, dx \\
& - \frac{1}{2}\int_{\mathbb{R}^2}u \cdot \nabla \partial_1 \Omega \, \partial_1^l r \, \partial_1^{m} s \, \big(\p_\rho f(\rho, b_2)\p_1\rho + \p_{b_2}f(\rho, b_2) \p_1 b_2\big) \, dx \\
\lesssim&\; \|(\rho, u, b_2)\|_{H^3}^2 \big(\|\partial_2 (\rho, b_2)\|_{H^2}^2 + \|\nabla u\|_{H^3}^2 + \|\nabla \Omega\|_{H^1}^2\big).
\end{split}
\end{equation}
When $l = 3$ or $m = 3$, the highest-order term $\mathcal{P}_{3,3}$ in \eqref{cancel-r} will be exactly canceled by $\mathfrak{P}$ (see \eqref{frakR}), which originated from $\mathcal{P}_{1,2} + \mathcal{P}_{1,3}$. This delicate cancellation mechanism completes the energy estimate of $\mathcal{P}_3$ and triumphantly overcomes the derivative loss problem. Synthesizing the above estimates and invoking the \textit{a priori} assumption \eqref{ansatz}, we conclude the proof of this lemma.
\end{proof}
\end{lemma}

\begin{lemma}\label{lemma-divu-rho-r2}
Let $r \in \{\partial_1^3 \rho, \partial_1^3 b_2\}$. Then the following inequality holds,
\begin{equation}\nonumber
\begin{split}
\Big|\int_{\mathbb{R}^2}(\nabla \cdot u) r^2 \rho \;dx + \frac{1}{2}\frac{d}{dt} \int_{\mathbb{R}^2}\rho^2 r^2 \;dx \Big| \lesssim \|(\rho, b_2)\|_{H^3}^2\|\partial_2 \rho\|_{H^1} \|\nabla u\|_{H^3}.
\end{split}
\end{equation}
\end{lemma}

\begin{proof}
Recalling the continuity equation for $\rho$,
$$\nabla \cdot u = - (\partial_t + u \cdot \nabla)\rho - \rho \nabla \cdot u,$$
we decompose the target integral as follows:
\begin{equation}\nonumber
\begin{split}
\int_{\mathbb{R}^2}(\nabla \cdot u) r^2 \rho \;dx =& - \frac{1}{2}\frac{d}{dt} \int_{\mathbb{R}^2}\rho^2 r^2 \;dx \\
& + \frac{1}{2}\int_{\mathbb{R}^2}\rho^2 (\partial_t + u \cdot \nabla) r^2 \;dx \\
& - \frac{1}{2}\int_{\mathbb{R}^2}\rho^2 (\nabla \cdot u) r^2\;dx \\
\triangleq& \; \mathcal{Q}_1 + \mathcal{Q}_2 + \mathcal{Q}_3.
\end{split}
\end{equation}

\paragraph{\bf Estimate of $\mathcal{Q}_2$:}
By virtue of H\"{o}lder's inequality and Proposition \ref{prop-anisotropic-est} we establish that
\begin{equation}\nonumber
\begin{split}
\mathcal{Q}_2 \lesssim& \;\|\rho\|_{L^\infty}^2 \|(\partial_t + u\cdot \nabla)r^2\|_{L^1} \\
\lesssim& \; \|\rho\|_{H^1}\|\partial_2 \rho\|_{H^1} \|(\partial_t + u\cdot \nabla)r^2\|_{L^1}.
\end{split}
\end{equation}
Utilizing the equations of $\rho$ and $b_2$, we can compute the material derivative of $r^2$  explicitly as
$$(\partial_t + u \cdot \nabla)|\partial_1^3 \rho|^2 = - 2\partial_1^3 \rho \partial_1^3 (\nabla \cdot u) - 2\partial_1^3\rho \partial_1^3(\rho \nabla \cdot u)- 2[\p_1^3, u\cdot \nabla]\rho \p_1^3\rho,$$
and
$$(\partial_t + u \cdot \nabla)|\partial_1^3 b_2|^2 = - 2\partial_1^3 b_2 \partial_1^3 (\partial_1 u_1) + 2\partial_1^3b_2 \partial_1^3(b\cdot \nabla u_2 - b_2 \nabla \cdot u)-2[\p_1^3, u\cdot \nabla]b_2 \p_1^3b_2.$$
Subsequently, applying H\"{o}lder's inequality and the Sobolev embedding theorem furnishes
\begin{equation}\nonumber
\begin{split}
\mathcal{Q}_2 \lesssim& \;\|\rho\|_{H^1}\|\partial_2 \rho\|_{H^1} \|r\|_{H^3} \|\nabla u\|_{H^3} \\
\lesssim& \; \|(\rho, b_2)\|_{H^3}^2\|\partial_2 \rho\|_{H^1} \|\nabla u\|_{H^3}.
\end{split}
\end{equation}

\paragraph{\bf Estimate of $\mathcal{Q}_3$:}
Following a parallel reasoning, we deduce
\begin{equation}\nonumber
\begin{split}
\mathcal{Q}_3 \lesssim& \;\|\rho\|_{L^\infty}^2 \|\nabla \cdot u\|_{L^\infty} \|r\|_{L^2}^2 \\
\lesssim& \;\|\rho\|_{H^1}\|\partial_2\rho\|_{H^1}\|\nabla \cdot u\|_{H^2} \|r\|_{L^2}^2 \\
\lesssim& \; \|(\rho, b_2)\|_{H^3}^2\|\partial_2\rho\|_{H^1}\|\nabla u\|_{H^3}.
\end{split}
\end{equation}
Amalgamating all the above estimates finalizes the proof of this lemma.
\end{proof}

At this juncture, we recall the effective viscous flux $\Gamma$ defined in \eqref{defgamma}. It will play a pivotal role in the following lemma.

\begin{lemma}\label{lemma-divu-b2-r2}
Let $r \in \{\partial_1^3 \rho, \partial_1^3 b_2\}$. Then the following control holds,
\begin{equation}\nonumber
\begin{split}
&\Big|\int_{\mathbb{R}^2}(\nabla \cdot u) r^2 b_2\;dx - \frac{1}{2}\frac{d}{dt} \int_{\mathbb{R}^2}\rho^2 r^2 \;dx \Big| \\
\lesssim& \int_{\mathbb{R}^2}\Gamma^2 r^2\;dx +\|(\rho, b_2)\|_{H^3}^2\big( \|(\partial_2 \rho, \partial_2 b_2)\|_{H^1}^2 + \|\nabla u\|_{H^3}^2\big).
\end{split}
\end{equation}
\end{lemma}

\begin{proof}
In light of the definition of $\Gamma$ in \eqref{defgamma}, we can formulate the algebraic identity as follows,
$$ b_2 = \Gamma - \rho -\big( P(1+\rho) - P(1) - \rho\big) - \frac{1}{2}|b_2|^2. $$
Substituting this identity into the target integral and applying Young's inequality, we arrive at
\begin{equation}\nonumber
\begin{split}
\int_{\mathbb{R}^2}(\nabla \cdot u) r^2 b_2\;dx + \int_{\mathbb{R}^2}(\nabla \cdot u) r^2 \rho\;dx \leq& \int_{\mathbb{R}^2}|\nabla \cdot u|^2 r^2\;dx + \int_{\mathbb{R}^2} \Gamma^2 r^2\;dx \\
& + \int_{\mathbb{R}^2}\big[\big( P(1+\rho) - P(1) - \rho\big)^2 + |b_2|^4\big] r^2\;dx.
\end{split}
\end{equation}
Observing that $P(1+\rho) - P(1) - \rho = \mathcal{O}(\rho^2)$ under the \textit{a priori} assumption \eqref{ansatz}, and invoking Lemma \ref{lemma-divu-rho-r2} together with Proposition \ref{prop-anisotropic-est} to bound the nonlinear remainders, we readily obtain the desired estimate.
\end{proof}

\begin{lemma}\label{lemma-p1u1-rho-r2}
Let $r \in \{\partial_1^3 \rho, \partial_1^3 b_2\}$. Then the following two inequalities hold.
\begin{equation}\nonumber
\begin{split}
&\Big|\int_{\mathbb{R}^2}\partial_1 u_1 \, r^2 \rho \;dx +\frac{d}{dt}\int_{\mathbb{R}^2}b_2 \, r^2 \rho \;dx + \frac{1}{2}\frac{d}{dt} \int_{\mathbb{R}^2}\rho^2 r^2 \;dx\Big|\\
\lesssim& \int_{\mathbb{R}^2}\Gamma^2 r^2\;dx + \|(b_2, \rho)\|_{H^3}^2\|\partial_2(b_2, \rho)\|_{H^1} \| \nabla u\|_{H^3},
\end{split}
\end{equation}
and
\begin{equation}\nonumber
\begin{split}
&\Big| \int_{\mathbb{R}^2}\partial_1 u_1 \, r^2 b_2 \;dx + \frac{1}{2}\frac{d}{dt}\int_{\mathbb{R}^2}|b_2|^2 r^2 \;dx \Big| \\
\lesssim& \|(b_2, \rho)\|_{H^3}^2\|\partial_2(b_2, \rho)\|_{H^1} \| \nabla u\|_{H^3}.
\end{split}
\end{equation}
\end{lemma}

\begin{proof}
From the evolution equation of $b_2$, we have
\begin{equation}\nonumber
- \partial_1 u_1 = \partial_t b_2 + u \cdot \nabla b_2 - b \cdot \nabla u_2 + b_2 \nabla \cdot u.
\end{equation}
Noting the relation $b \cdot \nabla u_2 - b_2 \nabla \cdot u = b_1 \partial_1 u_2 - b_2 \partial_1 u_1$, we infer that for any $s \in \{\rho, b_2\}$ there holds
\begin{equation}\nonumber
\begin{split}
\int_{\mathbb{R}^2}\partial_1 u_1 \, r^2 s \;dx =& - \int_{\mathbb{R}^2}\big[\partial_t b_2 + u \cdot \nabla b_2 - (b_1 \partial_1 u_2 - b_2 \partial_1 u_1)\big] r^2 s \;dx \\
=& - \int_{\mathbb{R}^2}(\partial_t b_2 + u \cdot \nabla b_2) r^2 s \;dx \\
& + \int_{\mathbb{R}^2}(b_1 \partial_1 u_2 - b_2 \partial_1 u_1) r^2 s \;dx \\
\triangleq& \; \mathcal{K}_1 + \mathcal{K}_2.
\end{split}
\end{equation}

\paragraph{\bf Estimate of $\mathcal{K}_1$:}
Transferring the material derivative, we decompose $\mathcal{K}_1$ into
\begin{equation}\nonumber
\begin{split}
\mathcal{K}_1 =& - \frac{d}{dt}\int_{\mathbb{R}^2}b_2 r^2 s \;dx \\
&+ \int_{\mathbb{R}^2}b_2 \, (\partial_t + u \cdot \nabla) r^2 \, s \;dx \\
&+ \int_{\mathbb{R}^2}b_2 \, r^2 (\partial_t + u \cdot \nabla)s \;dx \\
&+ \int_{\mathbb{R}^2}(\nabla \cdot u) b_2 r^2 s \;dx \\
\triangleq& \; \mathcal{K}_{1,1} + \mathcal{K}_{1,2} + \mathcal{K}_{1,3} + \mathcal{K}_{1,4}.
\end{split}
\end{equation}

\textbf{Estimate of $\mathcal{K}_{1,2}$:}

$\bullet$ \textbf{When $r = \partial_1^3\rho$:}
Utilizing the continuity equation,
$$(\partial_t + u \cdot \nabla)|\partial_1^3\rho|^2 = - 2 \partial_1^3\rho \, \partial_1^3\nabla \cdot u - 2 \partial_1^3\rho \, \partial_1^3( \rho \nabla \cdot u)- 2[\partial_1^3, u\cdot \nabla]\rho \, \partial_1^3\rho,$$
an application of H\"{o}lder's inequality, Proposition \ref{prop-anisotropic-est}, and the \textit{a priori} bound \eqref{ansatz} ensures that
\begin{equation}\nonumber
\begin{split}
\mathcal{K}_{1,2} \lesssim& \; \|b_2\|_{L^\infty}\|s\|_{L^\infty}\big(\|\partial_1^3 \rho \|_{L^2} \|\partial_1^3 \nabla \cdot u\|_{L^2} + \|\partial_1^3 \rho\|_{L^2}\|\rho\|_{H^3}\|\nabla u\|_{H^3} \big)\\
\lesssim& \; \|(b_2, s)\|_{H^1}\|\partial_2(b_2, s)\|_{H^1} \|\rho\|_{H^3}\| \nabla u\|_{H^3} \\
\lesssim& \; \|(b_2, \rho)\|_{H^3}^2\|\partial_2(b_2, \rho)\|_{H^1} \| \nabla u\|_{H^3}.
\end{split}
\end{equation}

$\bullet$ \textbf{When $r = \partial_1^3 b_2$:}
From the equation of $b_2$ we extract
\begin{equation}\nonumber
\begin{split}
(\partial_t + u \cdot \nabla)|\partial_1^3 b_2|^2 =& - 2 \partial_1^3b_2 \, \partial_1^4 u_1 + 2\partial_1^3 b_2 \, \partial_1^3(b \cdot \nabla u_2)\\
&- 2\partial_1^3 b_2 \, \partial_1^3(b_2 \nabla \cdot u)- 2[\partial_1^3, u\cdot \nabla]b_2 \, \partial_1^3b_2.
\end{split}
\end{equation}
Proceeding in a similar vein as the $r = \partial_1^3\rho$ case, we obtain
\begin{equation}\nonumber
\begin{split}
\mathcal{K}_{1,2} \lesssim \|(b_2, \rho)\|_{H^3}^2\|\partial_2(b_2, \rho)\|_{H^1} \| \nabla u\|_{H^3}.
\end{split}
\end{equation}

\textbf{Estimate of $\mathcal{K}_{1,3}$:}

$\bullet$ \textbf{When $s = \rho$:}
Substituting $(\partial_t + u \cdot \nabla) \rho = - \nabla \cdot u - \rho \nabla \cdot u$, we can split $\mathcal{K}_{1,3}$ into the following components
\begin{equation}\nonumber
\begin{split}
\mathcal{K}_{1,3} =& - \int_{\mathbb{R}^2}b_2 r^2 \nabla \cdot u \;dx - \int_{\mathbb{R}^2}b_2 r^2 \rho \nabla \cdot u\;dx \\
\triangleq& \; \mathcal{K}_{1,3,1} + \mathcal{K}_{1,3,2}.
\end{split}
\end{equation}
Applying Lemma \ref{lemma-divu-b2-r2} to $\mathcal{K}_{1,3,1}$ yields
\begin{equation}\nonumber
\begin{split}
&\Big|\mathcal{K}_{1,3,1} + \frac{1}{2}\frac{d}{dt} \int_{\mathbb{R}^2}\rho^2 r^2 \;dx \Big| \\
\lesssim& \int_{\mathbb{R}^2}\Gamma^2 r^2\;dx +\|(\rho, b_2)\|_{H^3}^2\big( \|\partial_2 (\rho, b_2)\|_{H^1}^2 + \|\nabla u\|_{H^3}^2\big).
\end{split}
\end{equation}
For $\mathcal{K}_{1,3,2}$, H\"{o}lder's inequality and Proposition \ref{prop-anisotropic-est} lead to
\begin{equation}\nonumber
\begin{split}
\mathcal{K}_{1,3,2} \lesssim& \;\|\nabla \cdot u\|_{L^\infty}\|b_2\|_{L^\infty}\|r\|_{L^2}^2 \|\rho\|_{L^\infty} \\
\lesssim& \;\|(b_2, \rho)\|_{H^1}\|\partial_2(b_2, \rho)\|_{H^1}\|\nabla u\|_{H^2} \|r\|_{L^2}^2 \\
\lesssim& \; \|(b_2, \rho)\|_{H^3}^3 \|\partial_2(b_2, \rho)\|_{H^1}\|\nabla u\|_{H^2}.
\end{split}
\end{equation}

$\bullet$ \textbf{When $s = b_2$:}
We rewrite $\mathcal{K}_{1,3}$ into the following parts,
\begin{equation}\nonumber
\begin{split}
\mathcal{K}_{1,3} =& \frac{1}{2}\int_{\mathbb{R}^2} r^2(\partial_t + u \cdot \nabla )|b_2|^2\; dx \\
=& \frac{1}{2}\frac{d}{dt}\int_{\mathbb{R}^2} r^2 |b_2|^2\; dx - \frac{1}{2}\int_{\mathbb{R}^2} (\partial_t + u \cdot \nabla ) r^2 \, |b_2|^2\; dx - \frac{1}{2}\int_{\mathbb{R}^2} r^2 (\nabla \cdot u) |b_2|^2\; dx\\
\triangleq& \; \mathcal{K}_{1,3,1} + \mathcal{K}_{1,3,2} + \mathcal{K}_{1,3,3}.
\end{split}
\end{equation}
Replicating the argument utilized for $\mathcal{K}_{1,2}$, we bound
\begin{equation}\nonumber
\begin{split}
\mathcal{K}_{1,3,2} \lesssim \|(b_2, \rho)\|_{H^3}^2\|\partial_2(\rho, b_2) \|_{H^1} \| \nabla u\|_{H^3}.
\end{split}
\end{equation}
For $\mathcal{K}_{1,3,3}$, a standard anisotropic type estimate yields
\begin{equation}\nonumber
\begin{split}
\mathcal{K}_{1,3,3} \lesssim& \; \|r\|_{L^2}^2 \|\nabla \cdot u\|_{L^\infty} \|b_2\|_{L^\infty}^2 \\
\lesssim& \;\|r\|_{L^2}^2 \|\nabla \cdot u\|_{H^2} \|b_2\|_{H^1}\|\partial_2 b_2\|_{H^1} \\
\lesssim& \; \|(\rho, b_2)\|_{H^3}^3 \|\nabla \cdot u\|_{H^2} \|\partial_2 b_2\|_{H^1}.
\end{split}
\end{equation}

\textbf{Estimate of $\mathcal{K}_{1,4}$:}
Following the same logic as above, it is evident that
\begin{equation}\nonumber
\begin{split}
\mathcal{K}_{1,4} \lesssim& \; \|\nabla \cdot u\|_{L^\infty}\|b_2\|_{L^\infty}\|r\|_{L^2}^2 \|s\|_{L^\infty} \\
\lesssim& \;\|(b_2, s)\|_{H^1}\|\partial_2(b_2, s)\|_{H^1}\|\nabla u\|_{H^2} \|r\|_{L^2}^2 \\
\lesssim& \;\|(b_2, \rho)\|_{H^3}^3 \|\partial_2(b_2, \rho)\|_{H^1}\|\nabla u\|_{H^2}.
\end{split}
\end{equation}

\paragraph{\bf Estimate of $\mathcal{K}_2$:}
Since $b_1 \partial_1 u_2 - b_2 \partial_1 u_1$ involves only $\partial_1-$derivative of $u$, we can readily derive
\begin{equation}\nonumber
\begin{split}
\mathcal{K}_2 \lesssim& \; \|(b, \rho)\|_{H^3}^3 \|\partial_2(b, \rho)\|_{H^1}\|\nabla u\|_{H^2}.
\end{split}
\end{equation}
Consolidating all the estimates for both cases above and invoking the \textit{a priori} assumption \eqref{ansatz}, we bring the proof of this lemma to a close.
\end{proof}

\begin{lemma}\label{lemma-divu-r2}
The following estimates hold when the right-hand sides are bounded.
\begin{equation}\nonumber
\begin{split}
&\Big|\int_{\mathbb{R}^2}\nabla \cdot u \, |\partial_1^3\rho|^2\;dx + \frac{d}{dt} \int_{\mathbb{R}^2}\rho \, |\partial_1^3\rho|^2 \;dx + \frac{d}{dt}\int_{\mathbb{R}^2}\partial_1^2 \Omega \, \partial_1^3 \rho \, \rho\;dx\\
&- 4\frac{d}{dt} \int_{\mathbb{R}^2}\rho^2 |\partial_1^3\rho|^2 \;dx -6 \frac{d}{dt}\int_{\mathbb{R}^2}b_2 |\partial_1^3\rho|^2 \rho \;dx \Big| \\
\lesssim& \int_{\mathbb{R}^2}\Gamma^2 |\partial_1^3\rho|^2\;dx + \|(\rho, u, b)\|_{H^3} \big(\|(\partial_2 \rho, \partial_2 b, \Omega)\|_{H^2}^2 + \|\nabla u\|_{H^3}^2+\|(\partial_t \rho, \partial_t b_2)\|_{L^2}^2\big) ,
\end{split}
\end{equation}
and
\begin{equation}\nonumber
\begin{split}
&\Big|\int_{\mathbb{R}^2}\nabla \cdot u \, |\partial_1^3b_2|^2\;dx + \frac{d}{dt} \int_{\mathbb{R}^2}\rho |\partial_1^3b_2|^2 \;dx + \frac{d}{dt}\int_{\mathbb{R}^2}\partial_1^2 \Omega \, \partial_1^3 b_2 \, \rho \;dx \\
& -8 \frac{d}{dt}\int_{\mathbb{R}^2}b_2 |\partial_1^3b_2|^2 \rho \;dx -4 \frac{d}{dt}\int_{\mathbb{R}^2} |\partial_1^3 b_2|^2 |b_2|^2\; dx\Big|\\
\lesssim& \int_{\mathbb{R}^2}\Gamma^2 |\partial_1^3b_2|^2\;dx+ \|(\rho, u, b)\|_{H^3} \big(\|(\partial_2 \rho, \partial_2 b, \Omega)\|_{H^2}^2 + \|\nabla u\|_{H^3}^2+\|(\partial_t \rho, \partial_t b_2)\|_{L^2}^2\big).
\end{split}
\end{equation}
\end{lemma}

\begin{proof}
Invoking the continuity equation, we extract the identity $\nabla \cdot u = - (\partial_t + u \cdot \nabla)\rho - \rho \nabla \cdot u$. Substituting this relation into the target integral furnishes, for $r \in \{\p_1^3\rho, \p_1^3 b_2\}$, the following decomposition:
\begin{equation}\nonumber
\begin{split}
\int_{\mathbb{R}^2}\nabla \cdot u \, r^2 \;dx =& - \frac{d}{dt} \int_{\mathbb{R}^2}\rho r^2 \;dx \\
& + \int_{\mathbb{R}^2}\rho (\partial_t + u \cdot \nabla) r^2 \;dx\\
\triangleq& \; \mathcal{M}_1 + \mathcal{M}_2.
\end{split}
\end{equation}
It is worth emphasizing that the term $\mathcal{M}_1$ is seamlessly absorbed into the time derivative part on the left-hand side of the lemma's statement. Consequently, our primary focus shifts to bounding $\mathcal{M}_2$. ~ \\
~ \\
\textbf{Case $r = \partial_1^3\rho$:}
Recalling the material derivative expansion
$$(\partial_t + u \cdot \nabla)|\partial_1^3 \rho|^2 = - 2\partial_1^3 \rho \partial_1^3 \nabla \cdot u - 2\partial_1^3\rho \partial_1^3(\rho \nabla \cdot u) - 2[\partial_1^3, u \cdot \nabla]\rho \partial_1^3 \rho,$$
and meticulously expanding the commutator while collecting terms of similar structure, we recast $\mathcal{M}_2$ as follows
\begin{equation}\nonumber
\begin{split}
\mathcal{M}_2 =& -2\int_{\mathbb{R}^2} \rho \partial_1^3 \rho \partial_1^4 u_1 \;dx \\
& -2\int_{\mathbb{R}^2} \rho \partial_1^3 \rho \partial_1^3 \partial_2 u_2 \;dx -6 \int_{\mathbb{R}^2}\rho \partial_1 u_2 \partial_2 \partial_1^2 \rho \partial_1^3\rho\;dx\\
& -2\int_{\mathbb{R}^2} \rho \partial_1^3\rho \partial_1^3 \rho \nabla \cdot u\; dx \\
& -2\sum_{k = 0}^2 \mathcal{C}_3^k\int_{\mathbb{R}^2} \rho \partial_1^3\rho \partial_1^k \rho \partial_1^{3-k} \nabla \cdot u\; dx -2 \sum_{k = 2}^3 \mathcal{C}_3^k \int_{\mathbb{R}^2}\rho \partial_1^k u \cdot \nabla \partial_1^{3-k} \rho \partial_1^3\rho\;dx \\
&-6 \int_{\mathbb{R}^2}\rho \partial_1 u_1 \partial_1^3 \rho \partial_1^3\rho\;dx \\
\triangleq& \; \mathcal{M}_{2,1} + \mathcal{M}_{2,2} + \mathcal{M}_{2,3} + \mathcal{M}_{2,4}+\mathcal{M}_{2,5}.
\end{split}
\end{equation}
Turning our attention to $\mathcal{M}_{2,1}$, Lemma \ref{lemma-p13u1} directly supplies the bound
\begin{equation}\nonumber
\begin{split}
& \Big|\mathcal{M}_{2,1} + \frac{d}{dt}\int_{\mathbb{R}^2}\partial_1^2 \Omega \, \partial_1^3 \rho \, \rho\;dx \Big| \\
\lesssim& \|(\rho, u, b)\|_{H^3} \big(\|(\partial_2 \rho, \partial_2 b, \Omega)\|_{H^2}^2 + \|\nabla u\|_{H^3}^2 + \|(\partial_t \rho, \partial_t b_2)\|_{L^2}^2\big).
\end{split}
\end{equation}
For $\mathcal{M}_{2,2}$, executing integration by parts in $x_2$ yields
\begin{equation}\nonumber
\begin{split}
\mathcal{M}_{2,2} =& 2\int_{\mathbb{R}^2} \partial_2 \rho \partial_1^3 \rho \partial_1^3 u_2 \;dx - 2\int_{\mathbb{R}^2} \rho \partial_1^2 \partial_2 \rho \partial_1^4 u_2 \;dx \\
&- 2\int_{\mathbb{R}^2} \partial_1 \rho \partial_1^2 \partial_2 \rho \partial_1^3 u_2 \;dx -6 \int_{\mathbb{R}^2}\rho \partial_1 u_2 \partial_2 \partial_1^2 \rho \partial_1^3\rho\;dx.
\end{split}
\end{equation}
As a direct consequence, H\"{o}lder's inequality and the Sobolev embedding theorem ensure that
\begin{equation}\nonumber
\begin{split}
\mathcal{M}_{2,2} \lesssim& \|\rho\|_{H^3}\|\partial_2\rho\|_{H^2}\|\partial_1u_2\|_{H^3}.
\end{split}
\end{equation}
Moving on to $\mathcal{M}_{2,3}$, Lemma \ref{lemma-divu-rho-r2} provides the estimate
\begin{equation}\nonumber
\begin{split}
\Big| \mathcal{M}_{2,3} -\frac{d}{dt} \int_{\mathbb{R}^2}\rho^2 |\partial_1^3\rho|^2 \;dx \Big| \lesssim\|(\rho, b_2)\|_{H^3}^2\|\partial_2 \rho\|_{H^1} \|\nabla u\|_{H^3}.
\end{split}
\end{equation}
Furthermore, appealing to H\"{o}lder's inequality and Proposition \ref{prop-anisotropic-est}, we readily deduce for $\mathcal{M}_{2,4}$ that
\begin{equation}\nonumber
\begin{split}
\mathcal{M}_{2,4} \lesssim \|\rho\|_{H^3}^2 \|\partial_2\rho\|_{H^2} \|\nabla u\|_{H^3}.
\end{split}
\end{equation}
Finally, regarding $\mathcal{M}_{2,5}$, Lemma \ref{lemma-p1u1-rho-r2} furnishes
\begin{equation}\nonumber
\begin{split}
&\Big|\mathcal{M}_{2,5} -6 \frac{d}{dt}\int_{\mathbb{R}^2}b_2 |\partial_1^3\rho|^2 \rho \;dx -3\frac{d}{dt} \int_{\mathbb{R}^2}\rho^2 |\partial_1^3\rho|^2 \;dx\Big|\\
\lesssim& \int_{\mathbb{R}^2}\Gamma^2 |\partial_1^3\rho|^2\;dx + \|(b_2, \rho)\|_{H^3}^2\|\partial_2(b_2, \rho)\|_{H^1} \| \nabla u\|_{H^3}.
\end{split}
\end{equation}

\textbf{Case $r = \partial_1^3 b_2$:}
In a parallel manner, we evaluate the material derivative as
$$(\partial_t + u \cdot \nabla)|\partial_1^3 b_2|^2 = - 2\partial_1^3 b_2 \partial_1^4 u_1 + 2\partial_1^3b_2 \partial_1^3(b\cdot \nabla u_2 - b_2 \nabla \cdot u) - 2[\partial_1^3, u\cdot \nabla] b_2 \partial_1^3 b_2.$$
Expanding and meticulously grouping terms leads to the decomposition
\begin{equation}\label{M2}
\begin{split}
\mathcal{M}_{2} =& -2\int_{\mathbb{R}^2} \rho \partial_1^3 b_2 \partial_1^4u_1\;dx \\
&+2\int_{\mathbb{R}^2} \rho \partial_1^3 b_2 \partial_1^3(b_1 \partial_1 u_2) \;dx \\
& - 2\int_{\mathbb{R}^2} \rho \partial_1^3 b_2 \partial_1^3 b_2 \partial_1 u_1 \;dx \\
& -2\sum_{k = 0}^2 \mathcal{C}_3^k\int_{\mathbb{R}^2} \rho \partial_1^3 b_2 \partial_1^k b_2 \partial_1^{4-k} u_1\;dx -2\sum_{k = 2}^3 \mathcal{C}_3^k\int_{\mathbb{R}^2} \rho \partial_1^k u \cdot \nabla \partial_1^{3-k} b_2 \partial_1^3 b_2\;dx\\
& -6 \int_{\mathbb{R}^2} \rho \partial_1 u_1 \partial_1^3 b_2 \partial_1^3 b_2\;dx\\
& -6 \int_{\mathbb{R}^2} \rho \partial_1 u_2 \partial_1^2\partial_2 b_2 \partial_1^3 b_2\;dx\\
& \triangleq \mathcal{M}_{2,1} + \mathcal{M}_{2,2} + \mathcal{M}_{2,3} + \mathcal{M}_{2,4}+ \mathcal{M}_{2,5}+ \mathcal{M}_{2,6}.
\end{split}
\end{equation}
By virtue of Lemma \ref{lemma-p13u1}, the term $\mathcal{M}_{2,1}$ is controlled by
\begin{equation}\nonumber
\begin{split}
& \Big| \mathcal{M}_{2,1} + \frac{d}{dt}\int_{\mathbb{R}^2}\partial_1^2 \Omega \, \partial_1^3 b_2 \, \rho \;dx\Big| \\
\lesssim& \|(\rho, u, b)\|_{H^3} \big(\|(\partial_2 \rho, \partial_2 b, \Omega)\|_{H^2}^2 + \|\nabla u\|_{H^3}^2 + \|(\partial_t\rho, \partial_t b_2)\|_{L^2}^2\big).
\end{split}
\end{equation}
For the combined contribution of $\mathcal{M}_{2,3}+ \mathcal{M}_{2,5}$, Lemma \ref{lemma-p1u1-rho-r2} implies
\begin{equation}\nonumber
\begin{split}
&\Big|\mathcal{M}_{2,3}+ \mathcal{M}_{2,5} -8 \frac{d}{dt}\int_{\mathbb{R}^2}b_2 |\partial_1^3b_2|^2 \rho \;dx -4 \frac{d}{dt}\int_{\mathbb{R}^2} |\partial_1^3 b_2|^2 |b_2|^2\; dx\Big| \\
\lesssim& \int_{\mathbb{R}^2} \Gamma^2 |\partial_1^3 b_2|^2\;dx + \|(b_2, \rho)\|_{H^3}^2\|\partial_2(b_2, \rho)\|_{H^1} \| \nabla u\|_{H^3}.
\end{split}
\end{equation}
As for the residual terms in \eqref{M2}, a standard application of H\"{o}lder's inequality and the Sobolev embedding theorem reveals that
\begin{equation}\nonumber
\begin{split}
\mathcal{M}_{2,2} + \mathcal{M}_{2,4} + \mathcal{M}_{2,6}\lesssim& \|\rho\|_{H^2}\|b_2\|_{H^2}\|\partial_2 b_2\|_{H^2}\|\partial_1 u_1\|_{H^3}.
\end{split}
\end{equation}
Synthesizing the above bounds for both cases and utilizing the \textit{a priori} assumption \eqref{ansatz}, we bring the proof of this lemma to a close.
\end{proof}

Building upon the preceding foundational estimates, we now establish the most crucial lemma of this section.

\begin{lemma}\label{lemma-p1u1-r2}
The following estimates hold when the right-hand sides are bounded.
\begin{equation}\nonumber
\begin{split}
& \Big|\int_{\mathbb{R}^2}\partial_1 u_1 |\partial_1^3\rho|^2 \;dx + \frac{d}{dt}\int_{\mathbb{R}^2}b_2 |\partial_1^3\rho|^2 \;dx +\frac{d}{dt}\int_{\mathbb{R}^2}\partial_1^2 \Omega \, \partial_1^3 \rho \, b_2 \;dx \\
& +\frac{1}{2}\frac{d}{dt} \int_{\mathbb{R}^2}\rho^2 |\partial_1^3\rho|^2 \;dx-\frac{7}{2}\frac{d}{dt} \int_{\mathbb{R}^2}b_2^2 |\partial_1^3\rho|^2 \;dx \Big|\\
\lesssim& \int_{\mathbb{R}^2} (|b_1|^2 + \Gamma^2) |\partial_1^3(\rho, b_2)|^2\;dx + \|(\rho, u, b)\|_{H^3} \big(\|(\partial_2 \rho, \partial_2 b, \Omega)\|_{H^2}^2 + \|\nabla u\|_{H^3}^2\big), \\
\end{split}
\end{equation}
and
\begin{equation}\nonumber
\begin{split}
&\Big|\int_{\mathbb{R}^2}\partial_1 u_1 |\partial_1^3b_2|^2 \;dx + \frac{d}{dt}\int_{\mathbb{R}^2}b_2 |\partial_1^3b_2|^2 \;dx + \frac{d}{dt}\int_{\mathbb{R}^2}\partial_1^2 \Omega \, \partial_1^3 b_2 \, b_2 (1 + b_2) \;dx\\
& -\frac{1}{2} \frac{d}{dt}\int_{\mathbb{R}^2}\rho^2 |\partial_1^3b_2|^2\;dx -\frac{9}{2} \frac{d}{dt}\int_{\mathbb{R}^2}|b_2|^2 |\partial_1^3b_2|^2 \;dx\;dx \Big|\\
\lesssim& \int_{\mathbb{R}^2} (|b_1|^2 + \Gamma^2) |\partial_1^3(\rho, b_2)|^2\;dx + \|(\rho, u, b)\|_{H^3} \big(\|(\partial_2 \rho, \partial_2 b, \Omega)\|_{H^2}^2 + \|\nabla u\|_{H^3}^2\big).
\end{split}
\end{equation}
\end{lemma}

\begin{proof}
The proof follows the analytical framework established in Lemma \ref{lemma-p1u1-rho-r2}. Let $r \in \{\p_1^3 \rho, \p_1^3 b_2\}$. Exploiting the equation of $b_2$, we once again decompose the target integral in the following manner
\begin{equation}\nonumber
\begin{split}
\int_{\mathbb{R}^2}\partial_1 u_1 r^2\;dx =& - \int_{\mathbb{R}^2}\big[\partial_t b_2 + u \cdot \nabla b_2 - b \cdot \nabla u_2 + b_2 \nabla \cdot u\big] r^2 \;dx \\
=& - \int_{\mathbb{R}^2}(\partial_t b_2 + u \cdot \nabla b_2) r^2 \;dx \\
& + \int_{\mathbb{R}^2}(b_1 \partial_1 u_2) r^2 \;dx \\
& - \int_{\mathbb{R}^2}(b_2 \partial_1 u_1) r^2 \;dx \\
\triangleq& \; \mathcal{S}_1 + \mathcal{S}_2 + \mathcal{S}_3.
\end{split}
\end{equation}

\paragraph{\bf Estimate of $\mathcal{S}_1$:}
It's very natural that we rewrite $\mathcal{S}_1$ as
\begin{equation}\nonumber
\begin{split}
\mathcal{S}_1 =& - \frac{d}{dt}\int_{\mathbb{R}^2}b_2 r^2 \;dx \\
&+ \int_{\mathbb{R}^2}b_2 (\partial_t + u \cdot \nabla) r^2 \;dx \\
&+ \int_{\mathbb{R}^2}(\nabla \cdot u) b_2 r^2 \;dx \\
\triangleq& \; \mathcal{S}_{1,1} + \mathcal{S}_{1,2} + \mathcal{S}_{1,3}.
\end{split}
\end{equation}

\textbf{Estimate of $\mathcal{S}_{1,2}$:}

$\bullet$ \textbf{Case $r = \partial_1^3\rho$:}
Recalling the material derivative expansion
$$(\partial_t + u \cdot \nabla)|\partial_1^3\rho|^2 = - 2 \partial_1^3\rho \partial_1^3\nabla \cdot u - 2 \partial_1^3\rho \partial_1^3( \rho \nabla \cdot u)- 2[\partial_1^3, u\cdot \nabla]\rho \partial_1^3\rho,$$
we group the resulting terms systematically as
\begin{equation}\nonumber
\begin{split}
\mathcal{S}_{1,2} =& \underbrace{-2 \int_{\mathbb{R}^2}b_2 \partial_1^3\rho \partial_1^4 u_1 \;dx}_{\mathcal{S}_{1,2,1}} \underbrace{-2 \int_{\mathbb{R}^2}b_2 \partial_1^3\rho \partial_1^3 \partial_2 u_2 \;dx}_{\mathcal{S}_{1,2,2}} \\
&\underbrace{-2 \int_{\mathbb{R}^2}b_2 \partial_1^3\rho \partial_1^3\rho \nabla \cdot u \;dx}_{\mathcal{S}_{1,2,3}} \underbrace{-6 \int_{\mathbb{R}^2}b_2 \partial_1 u_1 \partial_1^3\rho \partial_1^3 \rho \;dx}_{\mathcal{S}_{1,2,4}}\\
& \underbrace{-2\sum_{k=0}^2 \mathcal{C}_3^k\int_{\mathbb{R}^2} b_2 \p_1^3 \rho \p_1^k \rho \p_1^{3-k} \nabla \cdot u \;dx-6 \int_{\mathbb{R}^2}b_2 \partial_1 u_2 \partial_1^2 \partial_2 \rho \partial_1^3 \rho \;dx}_{\mathcal{S}_{1,2,5}}\\
&\underbrace{- 2\sum_{k = 2}^3 \mathcal{C}_3^k \int_{\mathbb{R}^2} b_2\partial_1^k u \cdot \nabla \partial_1^{3-k}\rho \partial_1^3\rho \;dx}_{\mathcal{S}_{1,2,6}}.
\end{split}
\end{equation}
Invoking Lemma \ref{lemma-p13u1} for $\mathcal{S}_{1,2,1}$ yields
\begin{equation}\nonumber
\begin{split}
&\Big|\mathcal{S}_{1,2,1} + \frac{d}{dt}\int_{\mathbb{R}^2}\partial_1^2 \Omega \, \partial_1^3 \rho \, b_2 \;dx \Big| \\
\lesssim& \|(\rho, u, b)\|_{H^3} \big(\|(\partial_2 \rho, \partial_2 b, \Omega)\|_{H^2}^2 + \|\nabla u\|_{H^3}^2+ \|(\partial_t \rho, \partial_t b_2)\|_{L^2}^2\big).
\end{split}
\end{equation}
For $\mathcal{S}_{1,2,2}$, integration by parts readily gives
\begin{equation}\nonumber
\begin{split}
\mathcal{S}_{1,2,2} \lesssim& \|(b_2, \rho)\|_{H^3} \|\partial_2(b_2, \rho)\|_{H^2} \|\partial_1 u_2\|_{H^3}.
\end{split}
\end{equation}
Regarding $\mathcal{S}_{1,2,3}$, Lemma \ref{lemma-divu-b2-r2} provides the control
\begin{equation}\nonumber
\begin{split}
&\Big|\mathcal{S}_{1,2,3} +\frac{d}{dt} \int_{\mathbb{R}^2}\rho^2 r^2 \;dx \Big|\\
\lesssim& \int_{\mathbb{R}^2}\Gamma^2 r^2\;dx + \|(\rho, b_2)\|_{H^3}^2\big(\|(\partial_2 \rho, \partial_2b_2)\|_{H^1}^2 + \|\nabla u\|_{H^3}^2\big).
\end{split}
\end{equation}
Furthermore, Lemma \ref{lemma-p1u1-rho-r2} implies for $\mathcal{S}_{1,2,4}$ that
\begin{equation}\nonumber
\begin{split}
\Big|\mathcal{S}_{1,2,4}-3 \frac{d}{dt}\int_{\mathbb{R}^2}|b_2|^2 r^2 \;dx \Big|\lesssim \|(b_2, \rho)\|_{H^3}^2\|\partial_2(b_2, \rho)\|_{H^1} \| \nabla u\|_{H^3}.
\end{split}
\end{equation}
The remaining terms $\mathcal{S}_{1,2,5}$ and $\mathcal{S}_{1,2,6}$ can be bounded via direct computation, which furnishes
\begin{equation}\nonumber
\begin{split}
\mathcal{S}_{1,2,5} + \mathcal{S}_{1,2,6} \lesssim& \|(b_2, \rho)\|_{H^3}^2 \|\partial_2(b_2, \rho)\|_{H^2} \|\nabla u\|_{H^3}.
\end{split}
\end{equation}

$\bullet$ \textbf{Case $r = \partial_1^3 b_2$:}
We compute the material derivative once more as
\begin{equation}\nonumber
\begin{split}
(\partial_t + u \cdot \nabla)|\partial_1^3 b_2|^2 =& - 2 \partial_1^3b_2 \partial_1^4 u_1 + 2\partial_1^3 b_2 \partial_1^3(b \cdot \nabla u_2)\\
&- 2\partial_1^3 b_2 \partial_1^3(b_2 \nabla \cdot u) - 2[\partial_1^3, u \cdot \nabla]b_2 \partial_1^3 b_2\\
=& - 2 \partial_1^3b_2 \partial_1^4 u_1 + 2\partial_1^3 b_2 \partial_1^3(b_1 \partial_1 u_2 - b_2 \partial_1 u_1) - 2[\partial_1^3, u \cdot \nabla]b_2 \partial_1^3 b_2.
\end{split}
\end{equation}
Expanding and meticulously collecting terms will lead to the decomposition
\begin{equation}\nonumber
\begin{split}
\mathcal{S}_{1,2} =& -2 \int_{\mathbb{R}^2} b_2 \partial_1^3 b_2 \partial_1^4 u_1 \;dx - 2 \int_{\mathbb{R}^2} b_2 \partial_1^3b_2 b_2 \partial_1^4 u_1\;dx\\
&- 2 \int_{\mathbb{R}^2} b_2 \partial_1^3b_2 \partial_1^3 b_2 \partial_1 u_1\;dx- 6\int_{\mathbb{R}^2} b_2 \partial_1 u_1 \partial_1^3 b_2 \partial_1^3 b_2\;dx\\
&- 2\sum_{k = 1}^2 \mathcal{C}_3^k\int_{\mathbb{R}^2} b_2 \partial_1^3b_2 \partial_1^k b_2 \partial_1^{4-k} u_1\;dx+ 2 \int_{\mathbb{R}^2} b_2 \partial_1^3b_2 \partial_1^3(b_1 \partial_1 u_2)\;dx \\
&- 2\sum_{k = 2}^3 \mathcal{C}_3^k\int_{\mathbb{R}^2} b_2 \partial_1^k u \cdot \nabla \partial_1^{3-k} b_2 \partial_1^3 b_2\;dx - 6\int_{\mathbb{R}^2} b_2 \partial_1 u_2 \partial_1^2 \partial_2 b_2 \partial_1^3 b_2\;dx \\
\triangleq& \; \mathcal{S}_{1,2,1} + \mathcal{S}_{1,2,2} + \mathcal{S}_{1,2,3}+ \mathcal{S}_{1,2,4}.
\end{split}
\end{equation}
By virtue of Lemma \ref{lemma-p13u1}, we establish for $\mathcal{S}_{1,2,1}$ that
\begin{equation}\nonumber
\begin{split}
&\Big|\mathcal{S}_{1,2,1} + \frac{d}{dt}\int_{\mathbb{R}^2}\partial_1^2 \Omega \, \partial_1^3 b_2 \, b_2 (1 + b_2) \;dx \Big|\\
\lesssim&\|(\rho, u, b)\|_{H^3} \big(\|(\partial_2 \rho, \partial_2 b, \Omega)\|_{H^2}^2 + \|\nabla u\|_{H^3}^2 + \|(\partial_t\rho, \partial_t b_2)\|_{L^2}^2\big).
\end{split}
\end{equation}
For $\mathcal{S}_{1,2,2}$, Lemma \ref{lemma-p1u1-rho-r2} yields
\begin{equation}\nonumber
\begin{split}
\Big|\mathcal{S}_{1,2,2} -4\frac{d}{dt}\int_{\mathbb{R}^2}|b_2|^2 r^2 \;dx \Big| \lesssim \|(b_2, \rho)\|_{H^3}^2\|\partial_2(b_2, \rho)\|_{H^1} \| \nabla u\|_{H^3}.
\end{split}
\end{equation}
As for the residual terms $\mathcal{S}_{1,2,3} + \mathcal{S}_{1,2,4}$, direct computation reveals that
\begin{equation}\nonumber
\begin{split}
\mathcal{S}_{1,2,3} + \mathcal{S}_{1,2,4} \lesssim& \|b_2\|_{H^3}^2 \|\partial_2 b_2\|_{H^2} \|\nabla u\|_{H^3}.
\end{split}
\end{equation}

\paragraph{\bf Estimate of $\mathcal{S}_{1,3}$:}
Invoking Lemma \ref{lemma-divu-b2-r2}, we readily obtain
\begin{equation}\nonumber
\begin{split}
&\Big|\mathcal{S}_{1,3} -\frac{1}{2}\frac{d}{dt} \int_{\mathbb{R}^2}\rho^2 r^2 \;dx \Big|\\
\lesssim& \int_{\mathbb{R}^2}\Gamma^2 r^2\;dx +\|(\rho, b_2)\|_{H^3}^2\big(\|(\partial_2 \rho, \partial_2b_2)\|_{H^1}^2 + \|\nabla u\|_{H^3}^2\big).
\end{split}
\end{equation}

\paragraph{\bf Estimate of $\mathcal{S}_2$:}
By virtue of Young's inequality and the Sobolev embedding theorem, there holds
\begin{equation}\nonumber
\begin{split}
\mathcal{S}_2 \lesssim& \int_{\mathbb{R}^2} \big(b_1^2 + (\partial_1 u_2)^2\big) r^2 \; dx \\
\lesssim& \int_{\mathbb{R}^2} |b_1|^2 |\partial_1^3(\rho, b_2)|^2\;dx + \|\partial_1^3(\rho, b_2)\|_{L^2}^2 \|\partial_1 u_2\|_{H^2}^2.
\end{split}
\end{equation}

\paragraph{\bf Estimate of $\mathcal{S}_3$:}
According to Lemma \ref{lemma-p1u1-rho-r2}, we have
\begin{equation}\nonumber
\begin{split}
\Big|\mathcal{S}_3-\frac{1}{2}\frac{d}{dt}\int_{\mathbb{R}^2}|b_2|^2 r^2 \;dx\Big| \lesssim \|(b_2, \rho)\|_{H^3}^2\|\partial_2(b_2, \rho)\|_{H^1} \| \nabla u\|_{H^3}.
\end{split}
\end{equation}
Consolidating the above bounds for both cases and invoking the \textit{a priori} assumption \eqref{ansatz}, we conclude the proof of this lemma.
\end{proof}
\vskip .3in
\section{Upper bound for the basic energy functional}
\label{basic}

This section establishes the desired \textit{a priori} upper bound for the basic energy $\mathcal{E}_0(t)$ associated with the setting \eqref{ansatz}.

\begin{lemma}\label{lemmaBasic}
For the energies defined in \eqref{energy-set}, the following inequality holds for all positive time $t$:
\begin{equation}\nonumber
   \mathcal{E}_0(t)  \lesssim   \mathcal{E}_4(t) + \mathcal{E}_5(t)+\mathcal{E}_{\mathrm{total}}^\frac{3}{2}(t).
\end{equation}
\end{lemma}

\begin{proof}
In light of the equivalence $\|f\|_{H^3} \sim \|f\|_{L^2} + \|f\|_{\dot H^3} \sim \|f\|_{L^2} + \|\p_1^3 f\|_{L^2} + \|\p_2^3 f\|_{L^2}$, we partition the proof into two distinct parts.
\medskip

\noindent{\bf Part 1: $L^2$ estimate}
\medskip
We commence with the $L^2$ estimate. Taking the $L^2$-inner product (on $\mathbb{R}^2$) of $(\rho, \tilde \rho u, b)$ with the corresponding equations in \eqref{mhd1}, we derive
\begin{equation}\nonumber
\begin{split}
  \frac{1}{2} \frac{d}{dt} \Big( \|\sqrt{\tilde \rho} u\|_{L^2}^2 + \|\rho\|_{L^2}^2 + \|b\|_{L^2}^2\Big) + \|\nabla u\|_{L^2}^2  = \sum_{i = 1}^3 I_i,
\end{split}
\end{equation}
where
\begin{equation}\nonumber
  \begin{split}
    I_1 = & -\int_{\mathbb{R}^2} (\nabla \cdot u ) \rho \, dx - \int_{\mathbb{R}^2} \nabla \rho \cdot u \, dx + \int_{\mathbb{R}^2} (\nabla^{\bot} \cdot b) u_1 \, dx \\
    & + \int_{\mathbb{R}^2} (\nabla^{\bot} u_1) \cdot b \, dx + \int_{\mathbb{R}^2} (b \cdot \nabla b) \cdot u \, dx + \int_{\mathbb{R}^2} (b \cdot \nabla u) \cdot b \, dx \\
    & - \int_{\mathbb{R}^2} b (\nabla \cdot u) b \, dx - \frac{1}{2}\int_{\mathbb{R}^2} (\nabla |b|^2)\cdot u \, dx - \int_{\mathbb{R}^2} (u \cdot \nabla b) \cdot b \, dx, \\
    I_2 = & - \int_{\mathbb{R}^2} (u \cdot \nabla \rho) \rho \, dx - \int_{\mathbb{R}^2} \rho (\nabla \cdot u) \rho \, dx, \qquad
    I_3 = \int_{\mathbb{R}^2} (\nabla \rho - \nabla P) \cdot u \, dx.
  \end{split}
\end{equation}

{\bf Part 1.1: Estimate of $I_1$ via cancellation}
\medskip
By virtue of integration by parts and the $\nabla \cdot b = 0$ condition, it is straightforward to verify that
\begin{equation}\label{I1}
 I_1 = 0.
\end{equation}

{\bf Part 1.2: Estimate of terms containing $\nabla \cdot u$}
\medskip

\quad $\bullet$ \textbf{Estimate of $I_2$:}
Integration by parts and H\"{o}lder's inequality reveal that
$$
 I_2 = -\frac{1}{2}\int_{\mathbb{R}^2} (\nabla \cdot u) \rho^2 \, dx.
$$
Decomposing $\nabla \cdot u$ into $\p_1 u_1 + \p_2 u_2$, we shall first estimate the term involving $\p_2 u_2$. Invoking Lemma \ref{lemma-p2u2}, this term can be controlled as follows:
\begin{equation}\label{I21}
  \begin{split}
    \int_0^t -\frac{1}{2}\int_{\mathbb{R}^2} \p_2 u_2 \rho^2 \, dx \, d\tau
    \lesssim& \sup_{0\leq \tau \leq t} \|(\rho, u_2)\|_{L^2} \int_0^t \|\p_1 u\|_{L^2}^\frac{1}{2} \|\p_2 \rho\|_{L^2}^\frac{3}{2} \, d\tau\\
    \lesssim& \;\mathcal{E}_{\mathrm{total}}^\frac{3}{2}(t).
  \end{split}
\end{equation}
Turning our attention to the term containing $\p_1 u_1$, Lemma \ref{lemma-p1u1} furnishes
\begin{equation}\label{I22}
  \begin{split}
    -\frac{1}{2}\int_{\mathbb{R}^2} \p_1 u_1 \rho^2 \, dx
    \lesssim& \sup_{0\leq \tau \leq t} \|(b_2, \rho)\|_{H^3}^3 \\
    &+ \sup_{0\leq \tau \leq t} \|(\rho, u, b)\|_{H^2}^2 \int_0^t \|(\p_2\rho, \nabla u, \p_2 b)\|_{H^2}^2 + \|(\p_t\rho, \p_t b_2)\|_{L^2}\, d\tau \\
    \lesssim& \;\mathcal{E}_{\mathrm{total}}^\frac{3}{2}(t).
  \end{split}
\end{equation}

\quad $\bullet$ \textbf{Estimate of $I_3$:}
To bound $I_3$, we introduce the auxiliary function
\begin{equation*}
	q(\rho)\triangleq (P(\rho+1) -P(1))-\rho = \int_0^\rho \big(P'(r + 1) - 1 \big) \, dr,
\end{equation*}
thereby reformulating $I_3$ as
$$
I_3 = -\int_{\mathbb R^2} q(\rho) \nabla\cdot u\,dx.
$$
Since $q(\rho)$ is a smooth function of $\rho$ with $q(0) = q'(0) = 0$ (due to $P'(1)=1$), we have the equivalence $q(\rho) \sim \rho^2$ for small $\rho$, provided that $\|\rho\|_{L^\infty}<1$. Consequently, following a parallel argument as for $I_2$ and invoking Lemmas \ref{lemma-p2u2} and \ref{prop-p1u1-rhorho}, we readily deduce
\begin{equation}\nonumber
  \begin{split}
    \int_0^t I_{3} \, d\tau
    \lesssim \mathcal{E}_{\mathrm{total}}^\frac{3}{2}(t).
  \end{split}
\end{equation}

Together with the previous estimates \eqref{I1}, \eqref{I21} and \eqref{I22}, we then arrive at the $L^2$ estimate of $\mathcal{E}_0(t)$.
\begin{equation}\label{I1I2I3}
  \begin{split}
    \int_0^t (I_{1}+ I_{2}+I_{3}) \, d\tau
    \lesssim \mathcal{E}_{\mathrm{total}}^\frac{3}{2}(t).
  \end{split}
\end{equation}

\noindent{\textbf{Part 2: $\dot H^3$ estimate}}

Applying the differential operators $(\p_1^3, \p_2^3)$ to the equations in \eqref{mhd1} and taking the $L^2$ inner product with $(\p_1^3(u, \rho, b), \p_2^3(u, \rho, b))$ respectively, we arrive at
\begin{equation}\nonumber
\begin{split}
  \frac{1}{2} \frac{d}{dt} \Big( \|\p_1^3 (u, \rho, b)\|_{L^2}^2 + \|\p_2^3 (u, \rho, b)\|_{L^2}^2\Big) + \|\p_1^3 u\|_{L^2}^2 + \|\p_2^3 u\|_{L^2}^2  = \sum_{i = 4}^{18} I_i,
\end{split}
\end{equation}
where
\begin{align*}
I_4 = & \;  \sum_{i = 1}^2\Big(- \int_{\mathbb{R}^2}  \partial_i^{3} \nabla \cdot u \partial_i^{3} \rho\; dx - \int_{\mathbb{R}^2} \partial_i^{3} \nabla \rho \cdot
\partial_i^{3} u\; dx \Big)\\
& +  \sum_{i = 1}^2\Big( \int_{\mathbb{R}^2} \partial_i^{3} \nabla^{\bot} \cdot b \partial_i^{3} u_1\; dx + \int_{\mathbb{R}^2} \partial_i^{3} \nabla^{\bot}  u_1\cdot \partial_i^{3} b \; dx\Big),\\
I_5 = & - \int_{\mathbb{R}^2} \partial_1^{3} (u \cdot \nabla \rho) \partial_1^{3} \rho \; dx , \quad I_6 = - \int_{\mathbb{R}^2} \partial_1^{3}(\rho \nabla \cdot u) \partial_1^{3} \rho \; dx , \\
I_7 = & - \int_{\mathbb{R}^2} \partial_1^{3} (u \cdot \nabla u) \cdot\partial_1^{3} u \; dx , \quad I_8 =   - \int_{\mathbb{R}^2} \partial_1^{3} \Big(\frac{\rho}{\rho + 1} \Delta u \Big)\cdot \partial_1^{3} u \; dx , \\
I_9 = &  \int_{\mathbb{R}^2} \partial_1^{3} \Big(\frac{1}{\rho + 1} b \cdot \nabla b\Big) \cdot\partial_1^{3} u \; dx , \quad I_{10} = - \int_{\mathbb{R}^2} \partial_1^{3} \Big(\frac{\rho}{\rho + 1} \nabla^{\bot} \cdot b \Big) \partial_1^{3} u_1 \; dx , \\
I_{11} = & \; \int_{\mathbb{R}^2} \partial_1^{3} \Big(\nabla \rho - \frac{1}{\rho + 1} \nabla P \Big) \cdot \partial_1^{3} u \; dx , \quad I_{12} =  - \int_{\mathbb{R}^2} \partial_1^{3} (u \cdot \nabla b)\cdot \partial_1^{3} b\; dx , \\
I_{13} = & \; \int_{\mathbb{R}^2} \partial_1^{3}(b \cdot \nabla u) \cdot\partial_1^{3} b \; dx , \quad I_{14} =  - \int_{\mathbb{R}^2} \partial_1^{3}(b \nabla \cdot u) \cdot\partial_1^{3} b \; dx ,  \\
I_{15} = & - \int_{\mathbb{R}^2} \partial_1^{3} \Big(\frac{1}{2(\rho+1)} \nabla |b|^2 \Big) \cdot\partial_1^{3} u \; dx, \\
I_{16} = & - \int_{\mathbb{R}^2} \partial_2^{3} (u \cdot \nabla \rho) \partial_2^{3} \rho \; dx - \int_{\mathbb{R}^2} \partial_2^{3} (u \cdot \nabla u) \cdot\partial_2^{3} u \; dx \\
& - \int_{\mathbb{R}^2} \partial_2^{3} \Big(\frac{\rho}{\rho + 1} \Delta u \Big)\cdot \partial_2^{3} u \; dx - \int_{\mathbb{R}^2} \partial_2^{3} (u \cdot \nabla b)\cdot \partial_2^{3} b\; dx, \\
I_{17} = & - \int_{\mathbb{R}^2} \partial_2^{3}(\rho \nabla \cdot u) \partial_2^{3} \rho \; dx + \int_{\mathbb{R}^2} \partial_2^{3} \Big(\frac{1}{\rho + 1} b \cdot \nabla b\Big) \cdot\partial_2^{3} u \; dx\\
 &- \int_{\mathbb{R}^2} \partial_2^{3} \Big(\frac{\rho}{\rho + 1} \nabla^{\bot} \cdot b \Big) \partial_2^{3} u_1 \; dx \\
 &+ \int_{\mathbb{R}^2} \partial_2^{3}(b \cdot \nabla u) \cdot\partial_2^{3} b \; dx - \int_{\mathbb{R}^2} \partial_2^{3}(b \nabla \cdot u) \cdot\partial_2^{3} b \; dx, \\
I_{18} = &  \int_{\mathbb{R}^2} \partial_2^{3} \Big(\nabla \rho - \frac{1}{\rho + 1} \nabla P \Big) \cdot \partial_2^{3} u \; dx - \int_{\mathbb{R}^2} \partial_2^{3} \Big(\frac{1}{2(\rho+1)} \nabla |b|^2 \Big) \cdot\partial_2^{3} u \; dx
\end{align*}

\textbf{Part 2.1: Estimate of linear terms}
\medskip
Executing integration by parts, we trivially observe that
\begin{equation}\label{I4}
  I_4 = 0.
\end{equation}

\textbf{Part 2.2: Estimate of favorable terms arising from the $\p_2^3$ derivatives}

$\bullet$ \textbf{Estimate of $I_{16}$:}
Utilizing the Leibniz rule and integration by parts, we can derive
\begin{equation}\nonumber
\begin{split}
I_{16} =& - \int_{\mathbb{R}^2} \sum_{k = 1}^3 \mathcal{C}_3^k \p_2^k u \cdot \nabla \p_2^{3-k} (\rho, u, b) \p_2^3(\rho, u, b) \; dx + \frac{1}{2}\int_{\mathbb{R}^2} \nabla \cdot u |\p_2^3(\rho, u, b)|^2 \; dx \\
& + \int_{\mathbb{R}^2} \p_2^2( \frac{\rho}{\rho + 1} \Delta u) \p_2^4 u \; dx
\end{split}
\end{equation}
Consequently, by virtue of H\"{o}lder's inequality and the Sobolev embedding theorem, together with Proposition \ref{prop-nonlinear-func}, there holds
\begin{equation}\label{I16}
\begin{split}
I_{16} \lesssim& \|(\rho, u, b)\|_{H^3} \|\p_2(\rho, u, b)\|_{H^2}^2 + \|\rho\|_{H^2} \|\nabla u\|_{H^3}^2.
\end{split}
\end{equation}

$\bullet$ \textbf{Estimate of $I_{17}$:}
In a similar vein as $I_{16}$, it is straightforward to obtain
\begin{equation}\label{I17}
\begin{split}
I_{17} \lesssim&  \|(\rho, u, b)\|_{H^3} \|\p_2(\rho, u, b)\|_{H^2}^2.
\end{split}
\end{equation}

$\bullet$ \textbf{Estimate of $I_{18}$:}
Let us define
\begin{equation}\label{def-w}
w(\rho)= \int_0^\rho \big(\frac{P'(r + 1)}{r + 1} - P'(1)\big) \; dr.
\end{equation}
It is readily verified that $w(0) = w'(0) = 0$. Employing integration by parts, we obtain
\begin{equation}\nonumber
\begin{split}
I_{18} =& \int_{\mathbb{R}^2} \partial_2^{2} \Big((\frac{P'(\rho + 1)}{\rho + 1} - 1) \nabla \rho \Big) \cdot \partial_2^{4} u \; dx + \int_{\mathbb{R}^2} \partial_2^{2} \Big(\frac{1}{2(\rho+1)} \nabla |b|^2 \Big) \cdot\partial_2^{4} u \; dx \\
=& \int_{\mathbb{R}^2} \partial_2^{2} \nabla w(\rho) \cdot \partial_2^{4} u \; dx + \int_{\mathbb{R}^2} \partial_2^{2} \Big(\frac{1}{2(\rho+1)} \nabla |b|^2 \Big) \cdot\partial_2^{4} u \; dx \\
\end{split}
\end{equation}
Appealing to H\"{o}lder's inequality and the Sobolev embedding theorem, we easily deduce
\begin{equation}\label{I18}
\begin{split}
I_{18} \lesssim \|(\rho, b)\|_{H^3}\|\p_2 (\rho, b)\|_{H^2} \|\nabla u\|_{H^3}.
\end{split}
\end{equation}

In what follows, we focus exclusively on the terms containing $\p_1^3$. By invoking the auxiliary lemmas established in Section \ref{lemma-pre}, we can effectively control these terms through the energies defined in \eqref{energy-set}.

\textbf{Part 2.3: Estimate of challenging terms generated by the $\p_1^3$ derivatives}

$\bullet$ \textbf{Estimate of $I_{5}$:}
Expanding the derivatives via the Leibniz rule, we decompose $I_5$ as
\begin{equation}\nonumber
\begin{split}
I_{5} =& -\int_{\mathbb{R}^2} \p_1^3 u \cdot \nabla \rho \p_1^3 \rho \; dx - 3\int_{\mathbb{R}^2} \p_1 u \cdot \nabla \p_1^2\rho \p_1^3 \rho \; dx\\
 &- 3\int_{\mathbb{R}^2} \p_1^2 u \cdot \nabla \p_1\rho \p_1^3 \rho \; dx - \int_{\mathbb{R}^2}  u \cdot \nabla \p_1^3\rho \p_1^3 \rho \; dx\\
=& \underbrace{-\int_{\mathbb{R}^2} \p_1^3 u_1 \p_1 \rho \p_1^3 \rho \; dx +\frac{3}{2}\int_{\mathbb{R}^2} \p_1^3 u_1 |\p_1^2 \rho|^2 \; dx}_{I_{5,1}} \\
&\underbrace{- 3\int_{\mathbb{R}^2} \p_1 u_1 \p_1^3\rho \p_1^3 \rho \; dx}_{I_{5,2}} + \underbrace{\frac{1}{2}\int_{\mathbb{R}^2}  \nabla \cdot u  \p_1^3\rho \p_1^3 \rho \; dx}_{I_{5,3}} \\
& \underbrace{-\int_{\mathbb{R}^2} \p_1^3 u_2 \cdot \p_2 \rho \p_1^3 \rho \; dx  - 3\int_{\mathbb{R}^2} \p_1 u_2 \p_2 \p_1^2\rho \p_1^3 \rho \; dx - 3\int_{\mathbb{R}^2} \p_1^2 u_2 \p_2 \p_1\rho \p_1^3 \rho \; dx}_{I_{5,4}}.
\end{split}
\end{equation}

 \quad $\diamond$ \textbf{Estimate of $I_{5,1}$:}
By virtue of Lemma \ref{lemma-p13u1}, we establish that
  \begin{equation}\nonumber
\begin{split}
&\int_0^t I_{5,1}  \; d\tau  \\
\lesssim& \sup_{0 \leq \tau \leq t} \int_{\mathbb{R}^2}\big|\p_1 \Omega  \cdot \p_1 \rho \cdot \p_1^3 \rho \big| + \big|\p_1 \Omega  \cdot \p_1^2 \rho \cdot \p_1^2 \rho \big|\;dx \\
 & + \sup_{0 \leq \tau \leq t}\|(\rho, u, b)\|_{H^3} \int_0^t (\|(\p_2 \rho, \p_2 b, \Omega)\|_{H^2}^2 + \|\nabla u\|_{H^3}^2 + \|(\p_t \rho, \p_t b_2)\|_{L^2}^2)
 \; d\tau \\
 \lesssim& \;\mathcal{E}_{\mathrm{total}}^\frac{3}{2}(t).
\end{split}
\end{equation}

 \quad $\diamond$ \textbf{Estimate of $I_{5,2}$:}
Invoking Lemma \ref{lemma-p1u1-r2}, we infer that
  \begin{equation}\nonumber
  \begin{split}
  \int_0^t I_{5,2}\; d\tau \lesssim& \sup_{0 \leq \tau \leq t} \int_{\mathbb{R}^2}\big| b_2 \cdot |\p_1^3 \rho|^2 \big| + \big|\p_1^2 \Omega  \cdot \p_1^3 \rho \cdot b_2 \big| +  (\rho^2+b_2^2) \cdot |\p_1^3\rho|^2\;dx \\
 & + \int_0^t \int_{\mathbb{R}^2} ( \Gamma^2+|b_1|^2) |\p_1^3(\rho, b_2)|^2\;dx\; d\tau \\
 & + \sup_{0 \leq \tau \leq t}\|(\rho, u, b)\|_{H^3} \cdot \int_0^t (\|(\p_2 \rho, \p_2 b, \Omega)\|_{H^2}^2 + \|\nabla u\|_{H^3}^2)\; d\tau \\
 \lesssim& \; \mathcal{E}_4(t) + \mathcal{E}_5(t)+\mathcal{E}_{\mathrm{total}}^\frac{3}{2}(t).
  \end{split}
  \end{equation}

 \quad $\diamond$ \textbf{Estimate of $I_{5,3}$:}
According to Lemma \ref{lemma-divu-r2}, it follows that
\begin{equation}\nonumber
  \begin{split}
    \int_0^t I_{5,3} \; d\tau \lesssim& \sup_{0\leq \tau \leq t} (\|\Omega\|_{H^2} \|\rho\|_{H^3}^2 + \|\rho\|_{H^3}^3) + \sup_{0\leq \tau \leq t} \|(\rho, u, b)\|_{H^3} \\
    & \cdot \int_0^t (\|(\p_2 \rho, \p_2 b, \Omega)\|_{H^2}^2 + \|\nabla u\|_{H^3}^2 + \|(\p_t\rho, \p_t b_2)\|_{L^2}^2) \; d\tau \\
    & + \int_0^t \int_{\mathbb{R}^2} \Gamma^2 |\p_1^3\rho|^2 \;dx d\tau \\
    \lesssim& \;\mathcal{E}_{\mathrm{total}}^\frac{3}{2}(t).
  \end{split}
\end{equation}

\quad $\diamond$ \textbf{Estimate of $I_{5,4}$:}
A direct application of H\"{o}lder's inequality ensures that
\begin{equation}\nonumber
  \begin{split}
    \int_0^t I_{5,4} \; d\tau \lesssim& \int_0^t \|\p_1 u_2\|_{H^3} \|\p_2 \rho\|_{H^2} \|\rho\|_{H^3} \; d\tau \\
    \lesssim& \;\mathcal{E}_{\mathrm{total}}^\frac{3}{2}(t).
  \end{split}
\end{equation}

Synthesizing the above bounds, we ultimately arrive at
\begin{equation}\label{I5}
  \begin{split}
    \int_0^t I_5 \; d\tau \lesssim \mathcal{E}_4(t) + \mathcal{E}_5(t)+\mathcal{E}_{\mathrm{total}}^\frac{3}{2}(t).
  \end{split}
\end{equation}

$\bullet$ \textbf{Estimate of $I_6$:}
We expand $I_6$ and regroup the terms as follows
\begin{equation}\nonumber
  \begin{split}
    I_6 =& -\int_{\mathbb{R}^2} \nabla \cdot u |\p_1^3 \rho|^2\;dx -3 \int_{\mathbb{R}^2} \p_1^2 \rho \p_1^2 u_1 \p_1^3 \rho\;dx-3 \int_{\mathbb{R}^2} \p_1 \rho \p_1^3 u_1 \p_1^3 \rho\;dx \\
    &-\int_{\mathbb{R}^2} \rho \p_1^4 u_1 \p_1^3 \rho\;dx -\sum_{l = 0}^2 \mathcal{C}_{3}^l \int_{\mathbb{R}^2} \p_1^l \rho \p_1^{3-l} \p_2 u_2 \p_1^3 \rho\;dx \\
    =& \underbrace{-\int_{\mathbb{R}^2} \nabla \cdot u |\p_1^3 \rho|^2\;dx}_{I_{6,1}} \\
    & \underbrace{+ \frac{3}{2} \int_{\mathbb{R}^2}  \p_1^3 u_1 |\p_1^2 \rho|^2\;dx-3 \int_{\mathbb{R}^2} \p_1 \rho \p_1^3 u_1 \p_1^3 \rho\;dx-\int_{\mathbb{R}^2} \rho \p_1^4 u_1 \p_1^3 \rho\;dx}_{I_{6,2}} \\
    &+\underbrace{\sum_{l = 0}^2 \mathcal{C}_{3}^l \int_{\mathbb{R}^2} \p_1^l \p_2 \rho \p_1^{3-l}  u_2 \p_1^3 \rho- \p_1^l  \rho \p_1^{4-l}  u_2 \p_1^2 \p_2 \rho- \p_1^{l+1}  \rho \p_1^{3-l}  u_2 \p_1^2 \p_2 \rho\;dx}_{I_{6,3}}.
  \end{split}
\end{equation}
Following an identical rationale as for $I_5$, we can readily deduce
\begin{equation}\label{I6}
  \begin{split}
    \int_0^t I_6 \; d\tau \lesssim \mathcal{E}_{\mathrm{total}}^\frac{3}{2}(t).
  \end{split}
\end{equation}

$\bullet$ \textbf{Estimate of $I_{12}$:}
We naturally split $I_{12}$ into two components:
\begin{equation}\nonumber
  \begin{split}
    I_{12} =& -\int_{\mathbb{R}^2} \p_1^3 (u \cdot \nabla b_1) \p_1^3 b_1 \;dx -\int_{\mathbb{R}^2} \p_1^3 (u \cdot \nabla b_2) \p_1^3 b_2 \;dx \\
    \triangleq& \, I_{12,1} + I_{12,2}.
  \end{split}
\end{equation}
Since $\p_1 b_1 = -\p_2 b_2$, the term $I_{12,1}$ can be straightforwardly estimated by
\begin{equation}\nonumber
  \begin{split}
    \int_0^t I_{12,1} \; d\tau \lesssim& \sup_{0 \leq \tau \leq t} \|b\|_{H^3} \int_0^t \|\p_1 u\|_{H^3} \|\p_2 b_2\|_{H^2} \; d\tau \\
    \lesssim& \;\mathcal{E}_{\mathrm{total}}^\frac{3}{2}(t).
  \end{split}
\end{equation}
For the term $I_{12,2}$, employing the same expansion technique as for $I_5$, we find
\begin{equation}\nonumber
\begin{split}
I_{12,2} =& -\int_{\mathbb{R}^2} \p_1^3 u \cdot \nabla b_2 \p_1^3 b_2 \; dx - 3\int_{\mathbb{R}^2} \p_1 u \cdot \nabla \p_1^2b_2 \p_1^3 b_2 \; dx - 3\int_{\mathbb{R}^2} \p_1^2 u \cdot \nabla \p_1b_2 \p_1^3 b_2 \; dx \\
&- \int_{\mathbb{R}^2}  u \cdot \nabla \p_1^3b_2 \p_1^3 b_2 \; dx\\
=& \underbrace{-\int_{\mathbb{R}^2} \p_1^3 u_1 \p_1 b_2 \p_1^3 b_2 \; dx +\frac{3}{2}\int_{\mathbb{R}^2} \p_1^3 u_1 |\p_1^2 b_2|^2 \; dx}_{I_{12,2,1}} \\
 &\underbrace{- 3\int_{\mathbb{R}^2} \p_1 u_1 \p_1^3b_2 \p_1^3 b_2 \; dx}_{I_{12,2,2}} + \underbrace{\frac{1}{2}\int_{\mathbb{R}^2}  \nabla \cdot u  \p_1^3b_2 \p_1^3 b_2 \; dx}_{I_{12,2,3}} \\
& \underbrace{-\int_{\mathbb{R}^2} \p_1^3 u_2 \cdot \p_2 b_2 \p_1^3 b_2 \; dx  - 3\int_{\mathbb{R}^2} \p_1 u_2 \p_2 \p_1^2b_2 \p_1^3 b_2 \; dx - 3\int_{\mathbb{R}^2} \p_1^2 u_2 \p_2 \p_1b_2 \p_1^3 b_2 \; dx}_{I_{12,2,4}}.
\end{split}
\end{equation}

\quad $\diamond$ \textbf{Estimate of $I_{12,2,1}$:}
Appealing to Lemma \ref{lemma-p13u1}, we establish that
  \begin{equation}\nonumber
\begin{split}
&\int_0^t I_{12,2,1} \; d\tau  \\
\lesssim& \sup_{0 \leq \tau \leq t} \int_{\mathbb{R}^2}\big|\p_1 \Omega  \cdot \p_1 b_2 \cdot \p_1^3 b_2 \big| + \big|\p_1 \Omega  \cdot \p_1^2 b_2 \cdot \p_1^2 b_2 \big|\;dx \\
 & + \sup_{0 \leq \tau \leq t}\|(\rho, u, b)\|_{H^3} \int_0^t (\|(\p_2 b_2, \p_2 b, \Omega)\|_{H^2}^2 + \|\nabla u\|_{H^3}^2 + \|(\p_t\rho, \p_t b_2)\|_{L^2}^2)
 \; d\tau \\
 \lesssim& \;\mathcal{E}_{\mathrm{total}}^\frac{3}{2}(t).
\end{split}
\end{equation}

\quad $\diamond$ \textbf{Estimate of $I_{12,2,2}$:}
By virtue of Lemma \ref{lemma-p1u1-r2}, there holds
    \begin{equation}\nonumber
  \begin{split}
  \int_0^t I_{12,2,2}\;d\tau \lesssim& \sup_{0 \leq \tau \leq t} \int_{\mathbb{R}^2} \big|b_2 \cdot |\p_1^3 b_2|^2\big|  + \big|\p_1^2 \Omega  \cdot \p_1^3 b_2 \cdot b_2 \cdot (1 + b_2) \big|+(\rho^2+|b_2|^2) \cdot |\p_1^3b_2|^2 \;dx \\
& +\int_0^t \int_{\mathbb{R}^2} (|b_1|^2 + \Gamma^2)|\p_1^3(\rho, b_2)|^2\;dx \; d\tau \\
 & + \sup_{0\leq \tau \leq t}\|(\rho, u, b)\|_{H^3} \int_0^t (\|(\p_2 \rho, \p_2 b, \Omega)\|_{H^2}^2 + \|\nabla u\|_{H^3}^2)\;d\tau \\
 \lesssim& \;  \mathcal{E}_4(t)+ \mathcal{E}_5(t)+\mathcal{E}_{\mathrm{total}}^\frac{3}{2}(t).
  \end{split}
  \end{equation}

\quad $\diamond$ \textbf{Estimate of $I_{12,2,3}$:}
Invoking Lemma \ref{lemma-divu-r2}, we readily obtain
  \begin{equation}\nonumber
\begin{split}
\int_0^t I_{12,2,3}\;d\tau \lesssim& \sup_{0\leq \tau \leq t}\int_{\mathbb{R}^2}\big| \rho |\p_1^3 b_2|^2\big| + \big|\p_1^2 \Omega  \cdot \p_1^3 b_2 \cdot \rho \big| +
2 \big|b_2 \cdot |\p_1^3b_2|^2 \cdot \rho \big| + |\p_1^3 b_2|^2 |b_2|^2\; dx \\
& +\int_0^t \int_{\mathbb{R}^2} \Gamma^2 |\p_1^3 b_2|^2 \;dx \;d\tau \\
& + \sup_{0\leq \tau \leq t}\|(\rho, u, b)\|_{H^3} \int_0^t (\|(\p_2 \rho, \p_2 b, \Omega)\|_{H^2}^2 + \|\nabla u\|_{H^3}^2 + \|(\p_t \rho, \p_t b_2)\|_{L^2}^2) \; d\tau \\
\lesssim& \;\mathcal{E}_5(t) +\mathcal{E}_{\mathrm{total}}^\frac{3}{2}(t) .
\end{split}
\end{equation}

\quad $\diamond$ \textbf{Estimate of $I_{12,2,4}$:}
A standard application of H\"{o}lder's inequality yields
\begin{equation}\nonumber
  \begin{split}
    \int_0^t I_{12,2,4} \; d\tau \lesssim& \int_0^t \|\p_1 u_2\|_{H^3} \|\p_2 b_2\|_{H^2} \|b_2\|_{H^3} \; d\tau \\
    \lesssim& \;\mathcal{E}_{\mathrm{total}}^\frac{3}{2}(t).
  \end{split}
\end{equation}

Aggregating these individual estimates, we conclude that
\begin{equation}\label{I12}
  \begin{split}
    \int_0^t I_{12} \; d\tau \lesssim  \mathcal{E}_4(t)+ \mathcal{E}_5(t)+\mathcal{E}_{\mathrm{total}}^\frac{3}{2}(t).
  \end{split}
\end{equation}

$\bullet$ \textbf{Estimate of $I_{14}$:}
By virtue of Leibniz's rule and integration by parts, we recast $I_{14}$ as
\begin{equation}\nonumber
\begin{split}
  I_{14} =& -\int_{\mathbb{R}^2} \nabla \cdot u  |\p_1^3b_2|^2\;dx -\int_{\mathbb{R}^2} \p_1^4 u_1  b_2 \p_1^3b_2 \;dx-\int_{\mathbb{R}^2} \p_1^3 \p_2 u_2  b_2 \p_1^3b_2 \;dx \\
  &-3 \int_{\mathbb{R}^2} \p_1^3 u_1 \p_1 b_2 \p_1^3b_2 \;dx-3\int_{\mathbb{R}^2} \p_1^2 \p_2 u_2  \p_1 b_2 \p_1^3b_2 \;dx \\
   & -3 \int_{\mathbb{R}^2} \p_1^2 u_1 \p_1^2 b_2 \p_1^3b_2 \;dx-3\int_{\mathbb{R}^2} \p_1 \p_2 u_2  \p_1^2 b_2 \p_1^3b_2 \;dx \\
   =& \underbrace{-\int_{\mathbb{R}^2} \nabla \cdot u |\p_1^3b_2|^2\;dx}_{I_{14,1}} \\
   &\underbrace{-\int_{\mathbb{R}^2} \p_1^4 u_1  b_2 \p_1^3b_2 \;dx-3 \int_{\mathbb{R}^2} \p_1^3 u_1 \p_1 b_2 \p_1^3b_2 \;dx + \frac{3}{2} \int_{\mathbb{R}^2} \p_1^3 u_1 |\p_1^2b_2|^2 \;dx}_{I_{14,2}} \\
   &\underbrace{-\int_{\mathbb{R}^2} \p_1^3 \p_2 u_2  b_2 \p_1^3b_2 \;dx-3\int_{\mathbb{R}^2} \p_1^2 \p_2 u_2  \p_1 b_2 \p_1^3b_2 \;dx-3\int_{\mathbb{R}^2} \p_1 \p_2 u_2  \p_1^2 b_2 \p_1^3b_2 \;dx}_{I_{14,3}} \\
   &\underbrace{-\int_{\mathbb{R}^2} \p_1^3(b_1 \nabla \cdot u) \p_1^3 b_1 \;dx}_{I_{14,4}}.
\end{split}
\end{equation}

 \quad $\diamond$ \textbf{Estimate of $I_{14,1}$:}
Invoking Lemma \ref{lemma-divu-r2} and replicating the arguments used for $I_{12,2,3}$, we find
\begin{equation}\nonumber
\begin{split}
\int_0^t I_{14,1}\;d\tau
\lesssim& \;\mathcal{E}_{\mathrm{total}}^\frac{3}{2}(t).
\end{split}
\end{equation}

 \quad $\diamond$ \textbf{Estimate of $I_{14,2}$:}
By virtue of Lemma \ref{lemma-p13u1}, there holds
\begin{equation}\nonumber
\begin{split}
\int_0^t I_{14,2}\;d\tau \lesssim& \sup_{0\leq \tau \leq t} \int_{\mathbb{R}^2}\big|\p_1^2 \Omega  \cdot b_2 \cdot \p_1^3 b_2 \big| + \big|\p_1 \Omega  \cdot \p_1 b_2 \cdot \p_1^3 b_2 \big| + \big|\p_1 \Omega  \cdot |\p_1^2 b_2|^2  \big|\;dx \\
 & + \sup_{0\leq \tau \leq t}\|(\rho, u, b)\|_{H^3} \int_0^t (\|(\p_2 \rho, \p_2 b, \Omega)\|_{H^2}^2 + \|\nabla u\|_{H^3}^2 + \|(\p_t \rho, \p_t b_2)\|_{L^2}^2)\;d\tau \\
 \lesssim& \;\mathcal{E}_{\mathrm{total}}^\frac{3}{2}(t).
\end{split}
\end{equation}

 \quad $\diamond$ \textbf{Estimate of $I_{14,3}$:}
Shifting the derivatives via integration by parts, we observe that
\begin{equation}\nonumber
  \begin{split}
    I_{14,3} =& \int_{\mathbb{R}^2} \p_1^3  u_2  \p_2 b_2 \p_1^3b_2 \;dx-\int_{\mathbb{R}^2} \p_1^4  u_2  b_2 \p_1^2\p_2 b_2 \;dx-\int_{\mathbb{R}^2} \p_1^3  u_2  \p_1 b_2 \p_1^2\p_2 b_2 \;dx\\
    &+3\int_{\mathbb{R}^2} \p_1^2 u_2  \p_1 \p_2 b_2 \p_1^3b_2 \;dx-3\int_{\mathbb{R}^2} \p_1^3 u_2  \p_1 b_2 \p_1^2\p_2 b_2 \;dx-6\int_{\mathbb{R}^2} \p_1^2 u_2  \p_1^2 b_2 \p_1^2\p_2 b_2 \;dx.
  \end{split}
\end{equation}
Consequently, applying H\"{o}lder's inequality and the Sobolev embedding theorem, we deduce
\begin{equation}\nonumber
  \begin{split}
    \int_0^t I_{14,3} \;d\tau \lesssim& \sup_{0\leq \tau \leq t} \|b_2\|_{H^3} \int_0^t \|\p_1 u_2\|_{H^3}^2 + \|\p_2 b_2\|_{H^2}^2 \;d\tau \\
    \lesssim& \;\mathcal{E}_{\mathrm{total}}^\frac{3}{2}(t).
  \end{split}
\end{equation}

\quad $\diamond$ \textbf{Estimate of $I_{14,4}$:}
Recalling the divergence-free condition $\p_1 b_1 = -\p_2 b_2$, it is straightforward to bound this term by
\begin{equation}\nonumber
  \begin{split}
    \int_0^t I_{14,4} \;d\tau \lesssim& \sup_{0\leq \tau \leq t} \|b_1\|_{H^3} \int_0^t \|\nabla \cdot u\|_{H^3}^2 + \|\p_2 b_2\|_{H^2}^2 \;d\tau \\
    \lesssim& \;\mathcal{E}_{\mathrm{total}}^\frac{3}{2}(t).
  \end{split}
\end{equation}

Gathering the estimates for $I_{14,1}$ through $I_{14,4}$, we secure the bound for $I_{14}$:
\begin{equation}\label{I14}
  \begin{split}
    \int_0^t I_{14} \; d\tau \lesssim \mathcal{E}_{\mathrm{total}}^\frac{3}{2}(t).
  \end{split}
\end{equation}

$\bullet$ \textbf{Estimate of $I_{15}$:}
Utilizing Leibniz's rule and integration by parts, we decompose $I_{15}$ as
\begin{equation}\nonumber
  \begin{split}
    I_{15} =& - \sum_{k = 1}^3 \mathcal{C}_{3}^k \int_{\mathbb{R}^2} \partial_1^k \frac{1}{2(\rho+1)} \nabla \partial_1^{3-k} |b|^2 \partial_1^{3} u \; dx \\
    & - \int_{\mathbb{R}^2} \frac{1}{2(\rho+1)} \partial_1^4 |b|^2 \partial_1^{3} u_1 \; dx  - \int_{\mathbb{R}^2} \frac{1}{2(\rho+1)} \partial_1^3\p_2 |b|^2 \partial_1^{3} u_2 \; dx \\
=& - \sum_{k = 1}^3 \mathcal{C}_{3}^k \int_{\mathbb{R}^2} \partial_1^k \frac{1}{2(\rho+1)} \nabla \partial_1^{3-k} |b|^2 \partial_1^{3} u \; dx +\int_{\mathbb{R}^2} \p_1 \frac{1}{2(\rho+1)} \partial_1^3 |b|^2 \partial_1^{3} u_1 \; dx \\
    & + \int_{\mathbb{R}^2} \frac{1}{2(\rho+1)} \partial_1^3 |b|^2 \partial_1^{4} u_1 \; dx \\
    & + \int_{\mathbb{R}^2} \p_1 \frac{1}{2(\rho+1)} \partial_1^2\p_2 |b|^2 \partial_1^{3} u_2 \; dx + \int_{\mathbb{R}^2} \frac{1}{2(\rho+1)} \partial_1^2\p_2 |b|^2 \partial_1^{4} u_2 \; dx \\
    \triangleq& \, I_{15,1} + I_{15,2} + I_{15,3}.
  \end{split}
\end{equation}

\quad $\diamond$ \textbf{Estimate of $I_{15,1}$:}
Invoking Propositions \ref{prop-anisotropic-est} and \ref{prop-nonlinear-func}, together with the \textit{a priori} assumption \eqref{ansatz}, we infer that
\begin{equation}\nonumber
\begin{split}
  I_{15,1} \lesssim& \|\p_1^3 \frac{1}{\rho +1}\|_{L^2} \|\nabla |b|^2\|_{L^\infty} \|\p_1^3 u\|_{L^2} + \|\p_1^2 \rho\|_{L_{x_1}^2L_{x_2}^\infty} \|\nabla \p_1 |b|^2\|_{L^2} \|\p_1^3 u\|_{L_{x_1}^\infty L_{x_2}^2}\\
  & + \|\p_1 \rho\|_{L^\infty} \|\nabla \p_1^2 |b|^2\|_{L^2} \|\p_1^3 u\|_{L^2} + \|\p_1\rho\|_{L^\infty} \|\p_1^3 |b|^2\|_{L^2} \|\p_1^3 u_1\|_{L^2} \\
  \lesssim&  \|(\rho, b)\|_{H^3} \|\p_2 (\rho, b)\|_{H^2} \|\p_1 u\|_{H^3}.
\end{split}
\end{equation}
Integrating this bound over time leads to
\begin{equation}\nonumber
\begin{split}
  \int_0^t I_{15,1}\; d\tau \lesssim& \sup_{0\leq \tau \leq t} \|(\rho, b)\|_{H^3} \int_0^t \|\p_2 (\rho, b)\|_{H^2} \|\p_1 u\|_{H^3} \; d\tau \\
  \lesssim& \;\mathcal{E}_{\mathrm{total}}^\frac{3}{2}(t).
\end{split}
\end{equation}

\quad $\diamond$ \textbf{Estimate of $I_{15,2}$:}
By virtue of Lemma \ref{lemma-p13u1}, it follows that
\begin{equation}\nonumber
\begin{split}
\int_0^t I_{15,2}\;d\tau \lesssim& \sup_{0\leq \tau \leq t} \sum_{k =0}^3 \int_{\mathbb{R}^2}\big|\p_1^2 \Omega  \cdot \p_1^k b_2 \cdot \p_1^{3-k} b_2 \cdot \frac{1}{\rho+1} \big| \;dx \\
 & + \sup_{0\leq \tau \leq t}\|(\rho, u, b)\|_{H^3} \int_0^t (\|(\p_2 \rho, \p_2 b, \Omega)\|_{H^2}^2 + \|\nabla u\|_{H^3}^2 + \|(\p_t \rho, \p_t b_2)\|_{L^2}^2)\;d\tau \\
 \lesssim& \;\mathcal{E}_{\mathrm{total}}^\frac{3}{2}(t).
\end{split}
\end{equation}

\quad $\diamond$ \textbf{Estimate of $I_{15,3}$:}
Following a parallel argument as for $I_{15,1}$, we readily obtain
\begin{equation}\nonumber
  \begin{split}
    \int_0^t I_{15,3}\;d\tau \lesssim& \sup_{0\leq \tau \leq t} \|b\|_{H^3} \int_0^t \|\p_2 b\|_{H^2} \|\p_1u_2\|_{H^3} \; d\tau \\
    \lesssim& \;\mathcal{E}_{\mathrm{total}}^\frac{3}{2}(t).
  \end{split}
\end{equation}

Summing up the bounds for $I_{15,1}$, $I_{15,2}$, and $I_{15,3}$ yields
\begin{equation}\label{I15}
  \begin{split}
    \int_0^t I_{15} \; d\tau \lesssim \mathcal{E}_{\mathrm{total}}^\frac{3}{2}(t).
  \end{split}
\end{equation}

$\bullet$ \textbf{Estimates of $I_7$ and $I_8$:}
Utilizing integration by parts and Proposition \ref{prop-anisotropic-est}, we directly arrive at
\begin{equation}\label{I7I8}
\begin{split}
\int_0^t( I_7 + I_8 )\; d\tau
\lesssim& \sup_{0 \leq \tau \leq t} \|\rho\|_{H^3} \int_0^t \|\nabla u\|_{H^3}^2 \; d\tau \\
\lesssim& \;\mathcal{E}_{\mathrm{total}}^\frac{3}{2}(t).
\end{split}
\end{equation}

$\bullet$ \textbf{Estimates of $I_9$ and $I_{13}$:}
By virtue of Leibniz's rule and integration by parts, and exploiting the condition $\nabla \cdot b = 0$, we can cancel the highest-order terms as follows:
\begin{equation}\nonumber
\begin{split}
I_9 + I_{13} =& \int_{\mathbb{R}^2} \sum_{k = 1}^3 \mathcal{C}_{3}^k\p_1^k (\frac{1}{\rho + 1} b) \cdot \nabla \p_1^{3-k} b \cdot \p_1^3 u \;dx \\
& + \int_{\mathbb{R}^2} \sum_{k = 1}^3 \mathcal{C}_{3}^k\p_1^k b \cdot \nabla \p_1^{3-k} u \cdot \p_1^3 b \;dx
- \int_{\mathbb{R}^2} \frac{\rho}{\rho + 1} b \cdot \nabla \p_1^3 b \cdot \p_1^3 u \;dx \\
& + \int_{\mathbb{R}^2}  b \cdot \nabla \p_1^3 b \cdot \p_1^3 u \;dx + \int_{\mathbb{R}^2} b \cdot \nabla \p_1^3 u \cdot \p_1^3 b \;dx \\
=& \int_{\mathbb{R}^2} \sum_{k = 1}^3 \mathcal{C}_{3}^k\p_1^k (\frac{1}{\rho + 1} b) \cdot \nabla \p_1^{3-k} b \cdot \p_1^3 u \;dx \\
& + \int_{\mathbb{R}^2} \sum_{k = 1}^3 \mathcal{C}_{3}^k\p_1^k b \cdot \nabla \p_1^{3-k} u \cdot \p_1^3 b \;dx
 - \int_{\mathbb{R}^2} \frac{\rho}{\rho + 1} b \cdot \nabla \p_1^3 b \cdot \p_1^3 u \;dx.
\end{split}
\end{equation}
Subsequently, we reformulate the expression in the following manner
\begin{equation}\label{I9I13}
\begin{split}
I_9 + I_{13}
=& \int_{\mathbb{R}^2} \sum_{k = 1}^3 \mathcal{C}_{3}^k\p_1^k \frac{1}{\rho + 1} b \cdot \nabla \p_1^{3-k} b \cdot \p_1^3 u \;dx \\
& + \int_{\mathbb{R}^2} \sum_{k = 1}^3 \mathcal{C}_{3}^k \sum_{l = 0}^{k-1}\mathcal{C}_{k}^{l} \p_1^l \frac{1}{\rho + 1}\p_1^{k-l} b \cdot \nabla \p_1^{3-k} b \cdot \p_1^3 u \;dx \\
& + \int_{\mathbb{R}^2} \sum_{k = 1}^3 \mathcal{C}_{3}^k\p_1^k b_1 \p_1^{4-k} u \cdot \p_1^3 b \;dx + \int_{\mathbb{R}^2} \sum_{k = 1}^2 \mathcal{C}_{3}^k\p_1^k b_2 \p_2\p_1^{3-k} u \cdot \p_1^3 b \;dx \\
& - \int_{\mathbb{R}^2} \frac{\rho}{\rho + 1} b \cdot \nabla \p_1^3 b \cdot \p_1^3 u \;dx + \int_{\mathbb{R}^2} \p_1^3 b_2 \p_2  u_1 \p_1^3 b_1 \;dx \\
& + \int_{\mathbb{R}^2} \p_1^3 b_2 \p_2  u_2 \p_1^3 b_2 \;dx
\end{split}
\end{equation}
Hence, appealing to H\"{o}lder's inequality and the Sobolev embedding theorem, all items above except the last one can be controlled by
\begin{equation}\nonumber
\begin{split}
\|(\rho, b)\|_{H^3} \|\p_2 (\rho, b)\|_{H^2} \|\nabla u\|_{H^3}.
\end{split}
\end{equation}
Here, we have used the \textit{a priori} assumption \eqref{ansatz}. For the last term on the right-hand side of \eqref{I9I13}, noting that $\p_2 u_2 = \nabla \cdot u - \p_1 u_1$, we observe that
\begin{equation}\nonumber
\begin{split}
\int_{\mathbb{R}^2} \p_1^3 b_2 \p_2  u_2 \p_1^3 b_2 \;dx =&
\int_{\mathbb{R}^2} \nabla \cdot u |\p_1^3 b_2|^2 \;dx - \int_{\mathbb{R}^2} \p_1 u_1 |\p_1^3 b_2|^2 \;dx
\end{split}
\end{equation}
In a similar vein as the estimates of $I_{14,1}$ and $I_{12,2,2}$, we establish that
\begin{equation}\nonumber
\begin{split}
\int_0^t \int_{\mathbb{R}^2} \p_1^3 b_2 \p_2  u_2 \p_1^3 b_2 \;dx \; d\tau \lesssim  \mathcal{E}_4(t) + \mathcal{E}_5(t)+\mathcal{E}_{\mathrm{total}}^\frac{3}{2}(t).
\end{split}
\end{equation}
Consolidating these bounds, we arrive at
\begin{equation}\label{I9I13}
\begin{split}
\int_0^t (I_9 + I_{13}) \;d\tau \lesssim  \mathcal{E}_4(t)+ \mathcal{E}_5(t)+\mathcal{E}_{\mathrm{total}}^\frac{3}{2}(t).
\end{split}
\end{equation}

$\bullet$ \textbf{Estimate of $I_{10}$:}
Recalling the identity $\nabla^\bot \cdot b = \p_2 b_1 - \p_1 b_2$, we rewrite $I_{10}$ as
\begin{equation}\nonumber
\begin{split}
I_{10} =& \int_{\mathbb{R}^2}  \p_1^2(\frac{\rho}{\rho + 1} \p_2 b_1) \p_1^4 u_1\;dx - \int_{\mathbb{R}^2}  \p_1^2(\frac{\rho}{\rho + 1} \p_1 b_2) \p_1^4 u_1\;dx
\end{split}
\end{equation}
For the first integral, it can be easily bounded by
\begin{equation}\nonumber
\|\rho\|_{H^3}  \|\p_2 b_1\|_{H^2} \|\p_1^4 u_1\|_{L^2}.
\end{equation}
For the second integral, invoking Lemma \ref{lemma-p13u1}, we finally deduce
\begin{equation}\label{I10}
\int_0^t I_{10}\; d\tau
\lesssim \mathcal{E}_{\mathrm{total}}^\frac{3}{2}(t).
\end{equation}

$\bullet$ \textbf{Estimate of $I_{11}$:}
Recalling the definition of $w$ in \eqref{def-w}, namely
    \begin{equation}\nonumber
w(\rho)= \int_0^\rho \big(\frac{P'(r + 1)}{r + 1} - P'(1)\big) \; dr,
\end{equation}
we can elegantly express $I_{11}$ as
\begin{equation}\nonumber
\begin{split}
I_{11} =& \int_{\mathbb{R}^2} \p_1^2(\frac{1}{\rho + 1} \nabla P - \nabla \rho) \p_1^4 u \; dx\\
=& \int_{\mathbb{R}^2} \p_1^2 \nabla w(\rho) \p_1^4 u \; dx\\
=& \int_{\mathbb{R}^2} \p_1^2 \p_2 w(\rho) \p_1^4 u_2 \; dx + \int_{\mathbb{R}^2} \p_1^3 w(\rho) \p_1^4 u_1 \; dx \\
\triangleq& \, I_{11,1} + I_{11,2}.
\end{split}
\end{equation}
Noting that $w(0) = w'(0) = 0$ and using Proposition \ref{prop-nonlinear-func}, we easily infer that
$$ I_{11,1} \lesssim \|\rho\|_{H^3} \|\p_2 \rho\|_{H^2} \|\p_1^4 u_2\|_{L^2}.$$
For the term $I_{11,2}$, utilizing Lemma \ref{lemma-p13u1} and the mean-value theorem $w'(\rho) = w''(r(\rho))\rho$, we find that
\begin{equation}\nonumber
\begin{split}
\int_0^t I_{11,2}\; d\tau \lesssim& \sup_{0\leq \tau \leq t} \int_{\mathbb{R}^2}\big|\p_1^2 \Omega   \p_1 \rho  \p_1 \rho  (w{'''}(\rho)\p_1\rho)\big|\; dx \\
& + \sup_{0\leq \tau \leq t} \int_{\mathbb{R}^2}3\big|\p_1^2 \Omega   \p_1 \rho \p_1^2 \rho  w{''}(\rho)\big|+ \big|\p_1^2 \Omega  \rho  \p_1^3 \rho w{''}(r(\rho))\big|\;dx \\
 & + \sup_{0\leq \tau \leq t} \|(\rho, u, b)\|_{H^3} \int_0^t \big(\|(\p_2 \rho, \p_2 b, \Omega)\|_{H^2}^2 + \|\nabla u\|_{H^3}^2 + \|(\p_t \rho, \p_t b_2)\|_{L^2}^2\big) \; d\tau \\
 \lesssim& \;\mathcal{E}_{\mathrm{total}}^\frac{3}{2}(t).
\end{split}
\end{equation}
Taking the estimates for both $I_{11,1}$ and $I_{11,2}$ into consideration we shall derive

\begin{equation}\label{I11}
  \int_0^t I_{11}\; d\tau \lesssim \;\mathcal{E}_{\mathrm{total}}^\frac{3}{2}(t).
\end{equation}

Collecting all the estimates derived in Part 1, namely \eqref{I1I2I3} and all the estimates derived in Part 2 for $\dot H^3$ norm, namely \eqref{I4}, \eqref{I5}, \eqref{I6}, \eqref{I7I8}, \eqref{I9I13}, \eqref{I10}, \eqref{I11}, \eqref{I12}, \eqref{I14}, \eqref{I15}, \eqref{I16}, \eqref{I17}, \eqref{I18}, we then complete the proof of this lemma.
\end{proof}

\section{Dissipation of density and magnetic field via vertical derivatives}

In this section, we shall fully exploit the dissipative structure of the system in the vicinity of a non-zero equilibrium state, as elaborated in Subsection \ref{subsection-dissipative-structure}. To this end, we establish the following crucial lemmas.

\begin{lemma}
For the energies defined in \eqref{energy-set}, the following estimate holds for all $t > 0$:
\begin{equation}\nonumber
  \begin{split}
    \mathcal{E}_1(t) \triangleq  \int_0^t \|\partial_2 \rho\|_{H^2}^2 \, d\tau \lesssim \mathcal{E}_0(t) + \mathcal{E}_4(t) + \mathcal{E}_{\mathrm{total}}^\frac{3}{2}(t).
  \end{split}
\end{equation}
\end{lemma}

\begin{proof}
We commence by rearranging the momentum equation for $u_2$ from \eqref{mhd1} to isolate the vertical derivative of the density:
\begin{equation}\nonumber
\begin{split}
\partial_2 \rho = & \; - \partial_t u_2+ \Big( 1- \frac{1}{\rho + 1} P'(\rho + 1)\Big) \partial_2 \rho + \frac{1}{\rho+1} \Delta u_2 -
u \cdot \nabla u_2  \\
&\quad + \frac{1}{\rho + 1} b \cdot \nabla b_2 - \frac{1}{2(\rho + 1)} \partial_2 |b|^2.
\end{split}
\end{equation}
Taking the $H^2$ inner product of this identity with $\partial_2 \rho$ yields the following energy equality:

\begin{equation}\label{p2_rho}
\begin{split}
\big\|\partial_2 \rho \big\|_{H^2}^2 = \sum_{i = 1}^{6} N_i,
\end{split}
\end{equation}
where
\begin{align*}
    N_1 =& -\sum_{k = 0}^2\int_{\mathbb{R}^2}  \nabla^k \partial_t u_2\, \nabla^k\partial_2 \rho\, dx,\\
    N_2 =& \sum_{k = 0}^2\int_{\mathbb{R}^2}    \nabla^k \Big( \big( 1- \frac{1}{\rho + 1} P'(\rho + 1) \big) \partial_2 \rho \Big)\,  \nabla^k\partial_2 \rho\,dx,\\
    N_3 =& \sum_{k = 0}^2\int_{\mathbb{R}^2}    \nabla^k \Big(\frac{1}{\rho+1} \Delta u_2 \Big)\,  \nabla^k\partial_2 \rho\,dx,\\
    N_4 =& -\sum_{k = 0}^2\int_{\mathbb{R}^2}   \nabla^k(u \cdot \nabla u_2)\, \nabla^k\partial_2 \rho\,dx,\\
    N_5 =& \sum_{k = 0}^2\int_{\mathbb{R}^2}   \nabla^k\Big(\frac{1}{\rho + 1} b \cdot \nabla b_2\Big)\, \nabla^k\partial_2 \rho\,dx,\\
    N_6 =& -\sum_{k = 0}^2\int_{\mathbb{R}^2}  \nabla^k \Big(\frac{1}{2(\rho + 1)} \partial_2 |b|^2 \Big)\, \nabla^k\partial_2 \rho\,dx.
 \end{align*}

$\bullet$ \textbf{Estimate of $N_1$:}
To control the time-derivative term $N_1$, we shift the temporal derivative via integration by parts and invoke the continuity equation for $\rho$, which yields
\begin{equation}\nonumber
\begin{split}
N_1 =& -\sum_{k = 0}^2 \frac{d}{dt} \int_{\mathbb{R}^2}  \nabla^k   u_2\, \nabla^k\partial_2 \rho\; dx
+ \sum_{k = 0}^2 \int_{\mathbb{R}^2}  \nabla^k  u_2\, \nabla^k\partial_2 \partial_t \rho\; dx\\
= & -\sum_{k = 0}^2 \frac{d}{dt} \int_{\mathbb{R}^2}  \nabla^k   u_2\, \nabla^k\partial_2 \rho\; dx
+ \sum_{k = 0}^2 \int_{\mathbb{R}^2}  \nabla^k   u_2\, \nabla^k\partial_2 (-\nabla \cdot u - \nabla \cdot (\rho u))\; dx.
\end{split}
\end{equation}
Consequently, applying standard Sobolev inequalities, we deduce that
\begin{equation}\label{N1}
\begin{split}
  &\Big|N_1 + \sum_{k = 0}^2 \frac{d}{dt} \int_{\mathbb{R}^2}  \nabla^k  u_2\, \nabla^k\partial_2 \rho\; dx \Big|\\
   \lesssim& \, \|\nabla u\|_{H^3}^2 + \|\rho\|_{H^3} \|\nabla u\|_{H^3}^2 + \|\partial_2 \rho\|_{H^2} \|\nabla u\|_{H^3}\|u\|_{H^3}.
\end{split}
\end{equation}

$\bullet$  \textbf{Estimates of $N_2$, $N_3$, $N_4$, and $N_6$:}
For these standard nonlinear terms, a routine application of H\"{o}lder's inequality and the Sobolev embedding theorem furnishes
\begin{equation}\label{N2346}
\begin{split}
N_2 + N_3 + N_4 + N_6 \lesssim& \;\|\rho\|_{H^3} \|\partial_2 \rho\|_{H^2}^2 + (1 + \|u\|_{H^3})\|\nabla u\|_{H^3} \|\partial_2 \rho\|_{H^2}\\
 &+ \|b\|_{H^3} \|\partial_2 b\|_{H^2} \|\partial_2 \rho\|_{H^2}.
\end{split}
\end{equation}

$\bullet$  \textbf{Estimate of $N_5$:}
Expanding the derivatives via the Leibniz rule first, we find that
\begin{equation}\nonumber
\begin{split}
N_5 =& \sum_{k = 0}^2\int_{\mathbb{R}^2}   \nabla^k\Big(\frac{1}{\rho + 1} b_1 \partial_1 b_2\Big)\, \nabla^k\partial_2 \rho\,dx
+ \sum_{k = 0}^2\int_{\mathbb{R}^2}   \nabla^k\Big(\frac{1}{\rho + 1} b_2 \partial_2 b_2\Big)\, \nabla^k\partial_2 \rho\,dx\\
=& \sum_{k = 0}^2\int_{\mathbb{R}^2}   \sum_{l = 1}^k \mathcal{C}_{k}^l \nabla^l b_1 \nabla^{k-l} \Big(\frac{1}{\rho + 1} \partial_1 b_2\Big)\, \nabla^k\partial_2 \rho\,dx\\
&+ \sum_{k = 0}^2\int_{\mathbb{R}^2}  b_1 \nabla^k \Big(\frac{1}{\rho + 1} \partial_1 b_2\Big)\, \nabla^k\partial_2 \rho\,dx + \sum_{k = 0}^2\int_{\mathbb{R}^2}   \nabla^k\Big(\frac{1}{\rho + 1} b_2 \partial_2 b_2\Big)\, \nabla^k\partial_2 \rho\,dx.
\end{split}
\end{equation}
To handle the highest-order derivatives carefully, we decompose $N_5$ into the following components:
\begin{equation}\nonumber
\begin{split}
N_5
=& \sum_{k = 0}^2\int_{\mathbb{R}^2}   \sum_{l = 1}^k \mathcal{C}_{k}^l \nabla^l b_1 \nabla^{k-l} \Big(\frac{1}{\rho + 1} \partial_1 b_2\Big)\, \nabla^k\partial_2 \rho\,dx\\
&+ \sum_{k = 0}^1\int_{\mathbb{R}^2}  b_1 \nabla^k \Big(\frac{1}{\rho + 1} \partial_1 b_2\Big)\, \nabla^k\partial_2 \rho\,dx + \int_{\mathbb{R}^2}  b_1  \sum_{l = 1}^2 \mathcal{C}_{2}^{l} \nabla^l\Big(\frac{1}{\rho + 1}\Big) \nabla^{2-l} \partial_1 b_2\, \nabla^2\partial_2 \rho\,dx\\
& + \int_{\mathbb{R}^2}  b_1  \Big(\frac{1}{\rho + 1} \nabla^2 \partial_1 b_2\Big)\, \nabla^2\partial_2 \rho\,dx \\
&+ \sum_{k = 0}^2\int_{\mathbb{R}^2}   \nabla^k\Big(\frac{1}{\rho + 1} b_2 \partial_2 b_2\Big)\, \nabla^k\partial_2 \rho\,dx\\
\triangleq& \, N_{5,1} + N_{5,2} + N_{5,3} + N_{5,4}.
\end{split}
\end{equation}
By virtue of standard Sobolev embeddings, the terms $N_{5,1}$, $N_{5,2}$, and $N_{5,4}$ satisfy
\begin{equation}\label{N51}
\begin{split}
N_{5,1} + N_{5,2} + N_{5,4} \lesssim& \|\nabla b_1\|_{H^2} (1 + \|\rho\|_{H^3}) \|\partial_1 b_2\|_{H^2} \|\partial_2 \rho\|_{H^2} \\
& + \|b_1\|_{L_{x_1}^\infty L_{x_2}^2} (1 + \|\rho\|_{H^3}) \|(\partial_1 b_2, \nabla \partial_1 b_2)\|_{L_{x_1}^2 L_{x_2}^\infty} \|\partial_2\rho\|_{H^2} \\
& + (1 + \|\rho\|_{H^3})\|b_2\|_{H^3} \|\partial_2 b_2\|_{H^2} \|\partial_2 \rho\|_{H^2}.
\end{split}
\end{equation}
It remains to estimate $N_{5,3}$. Applying Young's inequality, we obtain
\begin{equation}\label{N52}
\begin{split}
N_{5,3} \lesssim& \int_{\mathbb{R}^2} |b_1 \nabla^2 \partial_1 b_2 \nabla^2 \partial_2 \rho|\;dx \\
\lesssim & \, \frac{1}{c} \|b_1 \partial_1^3 b_2\|_{L^2}^2 + c\|\partial_2 \rho\|_{H^2}^2 + \|b\|_{H^3}\|(\partial_2 \rho, \partial_2 b_2)\|_{H^2}^2.
\end{split}
\end{equation}
Here, $c > 0$ should be chosen sufficiently small so that the highest-order contribution $c\|\partial_2 \rho\|_{H^2}^2$ can be absorbed by the left-hand side of \eqref{p2_rho}.

Integrating the identity \eqref{p2_rho} in time and synthesizing all the estimates derived above, namely \eqref{N1}, \eqref{N2346}, \eqref{N51} and \eqref{N52},  we ultimately deduce
\begin{equation}\nonumber
  \begin{split}
    \mathcal{E}_1(t) \triangleq&  \int_0^t \|\partial_2 \rho\|_{H^2}^2 \, d\tau \\
    \lesssim& \sup_{0 \leq \tau \leq t} \|u\|_{H^3} \|\rho\|_{H^3} + \int_0^t \|\nabla u\|_{H^3}^2 \; d\tau \\
&+ \sup_{0 \leq \tau \leq t} \|(\rho, b, u)\|_{H^3}^2 \cdot \int_0^t (\|\nabla u\|_{H^3}^2 + \|\partial_2(\rho, b)\|_{H^2}^2) \; d\tau \\
& + \int_0^t \int_{\mathbb{R}^2} |b_1|^2 |\partial_1^3 b_2|^2 \;dx \, d\tau \\
\lesssim& \;\mathcal{E}_0(t) + \mathcal{E}_4(t) + \mathcal{E}_{\mathrm{total}}^\frac{3}{2}(t).
  \end{split}
\end{equation}
This completes the proof of Lemma 5.1.
\end{proof}

\begin{lemma}
For the energies defined in \eqref{energy-set}, the following estimate holds for all $t > 0$:
\begin{equation}\nonumber
\begin{split}
  \mathcal{E}_2(t) \triangleq \int_0^t \|\partial_2 b\|_{H^2}^2 \; d\tau
  \lesssim \mathcal{E}_0(t) + \mathcal{E}_1(t) + \mathcal{E}_{\mathrm{total}}^\frac{3}{2}(t).
\end{split}
\end{equation}
\end{lemma}

\begin{proof}
In a parallel vein, we first rearrange the momentum equation for $u_1$ to extract the magnetic curl term:
\begin{equation}\nonumber
 - \nabla^\bot \cdot b =  -\partial_t u_1-\frac{\rho}{\rho + 1} \nabla^\bot \cdot b - u \cdot \nabla u_1 + \frac{1}{\rho+1}\Big(\Delta u_1 - \partial_1 P  + b \cdot \nabla b_1 - \frac{1}{2} \partial_1 |b|^2\Big).
\end{equation}
Applying $\partial_2$ to the above equation and taking the $H^2$ inner product with $b_1$, and noting the crucial structural identity
\begin{equation}\nonumber
\langle \partial_2 \nabla^\bot \cdot b, b_1 \rangle_{H^2} = \langle \Delta b_1, b_1 \rangle_{H^2} = -\|\nabla b_1\|_{H^2}^2 =- \|\partial_2 b\|_{H^2}^2,
\end{equation}
we arrive at the following relation:
\begin{equation}\nonumber
\|\partial_2 b\|_{H^2}^2 = \sum_{i = 1}^5 K_i,
\end{equation}
where
\begin{equation}\nonumber
\begin{split}
K_1 =& \langle \partial_t u_1, \partial_2 b_1 \rangle_{H^2}, \\
K_2 =& -\langle \partial_2 \Big(\frac{\rho}{\rho+1} \nabla^\bot \cdot b\Big), b_1 \rangle_{H^2}, \\
K_3 =& \langle u \cdot \nabla u_1 - \frac{1}{\rho + 1} \Delta u_1, \partial_2 b_1 \rangle_{H^2}, \\
K_4 =& - \langle \partial_2 \Big(\frac{1}{\rho+1} \partial_1 P\Big), b_1 \rangle_{H^2}, \\
K_5 =& \langle \partial_2 \Big(\frac{1}{\rho+1}\Big(b\cdot \nabla b_1 - \frac{1}{2} \partial_1  |b|^2\Big)\Big), b_1 \rangle_{H^2}.
\end{split}
\end{equation}

$\bullet$  \textbf{Estimate of $K_1$:}
For the time-derivative term $K_1$, we perform integration by parts in time and invoke the induction equation for $b_1$, which leads to
\begin{equation}\nonumber
\begin{split}
K_1 =& \frac{d}{dt}\langle u_1, \partial_2 b_1 \rangle_{H^2} + \langle \partial_2 u_1, \partial_t b_1 \rangle_{H^2} \\
    =& \frac{d}{dt}\langle u_1, \partial_2 b_1 \rangle_{H^2} + \langle \partial_2 u_1, \partial_2 u_1 - u \cdot \nabla b_1 + b \cdot \nabla u_1 - b_1 \nabla \cdot u \rangle_{H^2}.
\end{split}
\end{equation}
Consequently, standard Sobolev inequalities ensure that
\begin{equation}\label{K1}
\Big|K_1 -\frac{d}{dt}\langle u_1, \partial_2 b_1 \rangle_{H^2}\Big| \lesssim \|\partial_2 u_1\|_{H^2}(\|\partial_2 u_1\|_{H^2} + \|u\|_{H^3} \|\nabla b_1\|_{H^2} + \|b\|_{H^3} \|\nabla u\|_{H^2}).
\end{equation}

$\bullet$ \textbf{Estimate of $K_2$:}
Invoking the anisotropic bounds from Propositions \ref{prop-anisotropic-est} and \ref{prop-nonlinear-func} yields
\begin{equation}\label{K2}
\begin{split}
K_2 \lesssim& (\|\partial_2\rho\|_{L^2} \|\nabla^\bot \cdot b\|_{L_{x_1}^2L_{x_2}^\infty} +  \|\rho\|_{L_{x_1}^2L_{x_2}^\infty} \|\partial_2 \nabla^\bot \cdot b\|_{L^2}) \|b_1\|_{L_{x_1}^\infty L_{x_2}^2} \\
&  + (\|\partial_2\rho\|_{H^2} \|b\|_{H^3} + \|\rho\|_{H^3} \|\partial_2 b\|_{H^2}) \|\nabla b_1\|_{H^2} \\
\lesssim& \|(\rho, b)\|_{H^3} \|\partial_2(\rho, b)\|_{H^2}^2.
\end{split}
\end{equation}

$\bullet$  \textbf{Estimate of $K_3$:}
A direct application of H\"{o}lder's inequality gives
\begin{equation}\label{K3}
K_3 \lesssim \|(\rho, u)\|_{H^3} \|\nabla u\|_{H^3} \|\partial_2 b_1\|_{H^2}.
\end{equation}

$\bullet$  \textbf{Estimate of $K_4$:}
Let us introduce the auxiliary function $w(\rho) = \int_{0}^\rho \frac{P'(r+1)}{r + 1} \; dr$. Then, shifting the spatial derivatives via integration by parts, we observe that
\begin{equation}\label{K4}
\begin{split}
K_4 =& - \langle \partial_1 \partial_2 w(\rho), b_1 \rangle_{H^2}\\
 =& -\langle \partial_2 w(\rho), \partial_2 b_2 \rangle_{H^2} \\
 \lesssim& \, \|\partial_2 \rho\|_{H^2} \|\partial_2 b_2 \|_{H^2}.
\end{split}
\end{equation}

$\bullet$  \textbf{Estimate of $K_5$:}
Appealing once again to Propositions \ref{prop-anisotropic-est} and \ref{prop-nonlinear-func}, we infer that
\begin{equation}\label{K5}
\begin{split}
K_5 =& \langle \partial_2 \Big(\frac{1}{\rho+1}\Big(b\cdot \nabla b_1 - \frac{1}{2} \partial_1  |b|^2\Big)\Big), b_1 \rangle_{L^2} \\
 & + \langle \partial_2 \Big(\frac{1}{\rho+1}\Big(b\cdot \nabla b_1 - \frac{1}{2} \partial_1  |b|^2\Big)\Big), b_1 \rangle_{\dot H^1 \cap \dot H^2} \\
\lesssim& \, (\|\partial_2 \rho\|_{L^2}\|b\|_{L_{x_1}^2L_{x_2}^\infty}\|\nabla b\|_{L^\infty} + \|\partial_2 (b, \nabla b)\|_{L^2}\|(b, \nabla b)\|_{L_{x_1}^2 L_{x_2}^\infty})\|b_1\|_{L_{x_1}^\infty L_{x_2}^2}\\
 &+ (\|\partial_2 b\|_{H^2} + \|\partial_2 \rho\|_{H^2})\|b\|_{H^3} \|\nabla b_1\|_{H^2} \\
\lesssim& \, (\|\partial_2 b\|_{H^2} + \|\partial_2 \rho\|_{H^2})\|b\|_{H^3} \|\partial_2 b\|_{H^2}.
\end{split}
\end{equation}
Aggregating the bounds for $K_1$ through $K_5$, namely \eqref{K1}, \eqref{K2}, \eqref{K3}, \eqref{K4} and \eqref{K5}, we obtain the desired integral estimate:
\begin{equation}\nonumber
\begin{split}
  \int_0^t \|\partial_2 b\|_{H^2}^2 \; d\tau \lesssim& \sup_{0\leq \tau \leq t} \big|\langle u_1, \partial_2 b_1 \rangle_{H^2}\big| + \int_0^t \|\partial_2 u_1\|_{H^2}(\|\partial_2 u_1\|_{H^2} + \|u\|_{H^3} \|\nabla b_1\|_{H^2}\\
   &+ \|b\|_{H^3} \|\nabla u\|_{H^2}) \; d\tau \\
  & + \sup_{0 \leq \tau \leq t} \|(\rho, b, u)\|_{H^3} \int_0^t \|\partial_2(\rho, b)\|_{H^2}^2 + \|\nabla u\|_{H^3}^2 \; d\tau \\
  & + \int_0^t \|\partial_2 \rho\|_{H^2} \|\partial_2 b_2 \|_{H^2} \; d\tau \\
  \lesssim& \;\mathcal{E}_0(t) + \mathcal{E}_1(t) + \mathcal{E}_{\mathrm{total}}^\frac{3}{2}(t).
\end{split}
\end{equation}
This concludes the proof of Lemma 5.2 also finishes this section.
\end{proof}

\vskip .3in
\section{Dissipation of modified effective viscous flux and critical time-decay terms}
\label{estimate-omega}

This section is devoted to establishing the estimates for the modified effective viscous flux $(\Gamma, \Omega)$ and the critical time-decay terms, the structural motivations of which have been elaborated in Subsection \ref{subsection-dissipative-structure}.

\begin{lemma}
For the energies defined in \eqref{energy-set}, the following estimate holds for all $t > 0$:
\begin{equation}\nonumber
\begin{split}
 \mathcal{E}_3(t) \triangleq \int_0^t \big(\|\Omega\|_{H^2}^2 + \|\nabla \Gamma\|_{L^2}^2\big) \, d\tau
 \lesssim \mathcal{E}_0(t) + \mathcal{E}_1(t) + \mathcal{E}_2(t)+ \mathcal{E}_{\mathrm{total}}^\frac{3}{2}(t).
\end{split}
\end{equation}
\end{lemma}
\begin{proof}
The proof is naturally partitioned into two parts.

{\bf Part 1: Estimate of $\Omega$}

We commence by recalling the structural identity
\begin{equation}\nonumber
  \begin{split}
    \partial_t u_1 + u \cdot \nabla u_1 - \Delta u_1 - \Omega = -\frac{\rho}{\rho+1}\big(\Delta u_1 + \Omega \big).
  \end{split}
\end{equation}
Taking the $H^2$ inner product of the above equation with $\Omega$, we obtain
\begin{equation}\nonumber
\|\Omega\|_{H^2}^2 = \sum_{i = 1}^3 M_i,
\end{equation}
where
\begin{equation}\nonumber
\begin{split}
M_1 &= \langle \partial_t u_1, \Omega \rangle_{H^2}, \\
M_2 &= \langle u \cdot \nabla u_1-\Delta u_1, \Omega \rangle_{H^2}, \\
M_3 &= \Big\langle \frac{\rho}{\rho + 1}(\Delta u_1 + \Omega), \Omega \Big\rangle_{H^2}.
\end{split}
\end{equation}

$\bullet$  \textbf{Estimate of $M_1$:}
Shifting the temporal derivative via integration by parts yields
\begin{equation}\nonumber
\begin{split}
M_1 &= \frac{d}{dt}\langle u_1, \Omega \rangle_{H^2} - \langle u_1, \partial_t \Omega \rangle_{H^2} \triangleq M_{1,1} + M_{1,2}.
\end{split}
\end{equation}
Substituting the evolution equation for $\Omega$ from \eqref{eqomega}, the term $M_{1,2}$ can be recast as
\begin{equation}\nonumber
\begin{split}
M_{1,2} &= -\langle u_1, 2 \partial_1^2 u_1 + \partial_2^2 u_1 + \partial_1 \partial_2 u_2 \rangle_{H^2} -\langle u_1, - u \cdot \nabla \Omega \rangle_{H^2} -\langle u_1, Q \rangle_{H^2} \\
&\triangleq M_{1,2,1} + M_{1,2,2} + M_{1,2,3}.
\end{split}
\end{equation}
Here, $Q$ is defined in \eqref{Q}, namely,
\begin{equation}\nonumber
\begin{split}
Q &= \partial_1 u\cdot \nabla b_2 - \frac{1}{2} \partial_1 u \cdot \nabla (|b|^2 )+ \partial_1 u \cdot \nabla P + \nabla^\perp \cdot (b \cdot \nabla u) \\
&\quad - \partial_2 u \cdot \nabla b_1 - \partial_1 (b \cdot \nabla^\perp u_1) + b \cdot \nabla (\partial_2 u_1) + \nabla^\perp u_1 \cdot \nabla b_1 \\
&\quad + b \cdot \nabla(b \cdot \nabla u_1)- \partial_1 \big( b \cdot (b \cdot \nabla u) \big)- \nabla^\perp \cdot (b \nabla \cdot u) + \partial_1 (|b|^2 \nabla \cdot u) \\
&\quad + \partial_1 \Big\{ \big(P'(\tilde \rho) \tilde \rho - 1 \big) \nabla \cdot u \Big\} - \nabla \cdot u (b \cdot \nabla b_1) - b \cdot \nabla (b_1 \nabla \cdot u).
\end{split}
\end{equation}
For $M_{1,2,1}$ and $M_{1,2,2}$, executing integration by parts gives
\begin{equation}\nonumber
\begin{split}
M_{1,2,1} + M_{1,2,2} &= \langle \p_1 u_1, 2 \partial_1 u_1 + \partial_2 u_2 \rangle_{H^2} + \langle \p_2 u_1, \partial_2 u_1 \rangle_{H^2}\\
& =-\langle u_1, \nabla \cdot u \, \Omega \rangle_{H^2} - \langle \nabla u_1, u \, \Omega \rangle_{H^2} \\
&\lesssim \|\nabla u\|_{H^2}^2 + \|u\|_{H^2} \|\nabla u\|_{H^2} \|\Omega\|_{H^2}.
\end{split}
\end{equation}
For $M_{1,2,3}$, we observe that each term in $Q$ contains at least one velocity derivative $\partial_i u_j$. Hence,
\begin{equation}\nonumber
\begin{split}
M_{1,2,3} &= -\langle u_1, Q \rangle_{L^2} - \langle u_1, Q \rangle_{\dot H^1 \cap \dot H^2} \\
&\lesssim \|u_1\|_{L_{x_1}^\infty L_{x_2}^2} \|(\rho, b, \nabla \rho, \nabla b)\|_{L_{x_1}^2 L_{x_2}^\infty} \|\nabla u\|_{L^2} + \|\nabla u_1\|_{H^3} \|Q\|_{L^2} \\
&\lesssim \|(\rho, u, b)\|_{H^2} \|\partial_1 u_1\|_{L^2}^\frac{1}{2} \|\partial_2 (\rho, b)\|_{L^2}^\frac{1}{2} \|\nabla u\|_{L^2} + \|(\rho, u, b)\|_{H^3} \|\nabla u\|_{H^3}^2.
\end{split}
\end{equation}
Consequently, integrating these bounds in time yields
\begin{equation}\label{M1}
\begin{split}
  \int_0^t M_1 \, d\tau &\lesssim \sup_{0\leq \tau \leq t} \|u_1\|_{H^2} \|\Omega\|_{H^2} + \int_0^t \|\nabla u\|_{H^2}^2 \, d\tau  + \sup_{0 \leq \tau \leq t} \|u\|_{H^2} \int_0^t \|\nabla u\|_{H^2} \|\Omega\|_{H^2} \, d\tau \\
   &\quad + \sup_{0 \leq \tau \leq t} \|(\rho, u, b)\|_{H^3} \int_0^t \|\nabla u\|_{H^3}^2 + \|(\p_2 \rho, \p_2 b)\|_{L^2}^2 \, d\tau \\
   &\lesssim \mathcal{E}_0(t) + \mathcal{E}_{\mathrm{total}}^\frac{3}{2}(t).
\end{split}
\end{equation}

$\bullet$ \textbf{Estimates of $M_2$ and $M_3$:}
We bound $M_2$ and $M_3$ simultaneously:
\begin{equation}\nonumber
|M_2 + M_3| \lesssim (1 + \|u\|_{H^2})\|\nabla u\|_{H^3} \|\Omega\|_{H^2} + \|\rho\|_{H^2}(\|\nabla u\|_{H^3} + \|\Omega\|_{H^2})\|\Omega\|_{H^2}.
\end{equation}
Applying Young's inequality and integrating in time, we obtain
\begin{equation}\label{M23}
\int_0^t (M_2 + M_3) \, d\tau \leq C \sup_{0 \leq \tau \leq t} (1 + \|(\rho, u)\|_{H^2}) \int_0^t \|\nabla u\|_{H^3}^2 \, d\tau + \frac{1}{2}\int_0^t \|\Omega\|_{H^2}^2 \, d\tau.
\end{equation}
Synthesizing the above estimates \eqref{M1}, \eqref{M23} and invoking the \textit{a priori} assumption \eqref{ansatz}, we complete the proof for $\Omega$.

{\bf Part 2: Estimate of $\Gamma$}

  $\bullet$ \textbf{Estimate of $\partial_1 \Gamma$:}
  Utilizing the identity
   \begin{equation}\nonumber
\begin{split}
\Omega - \partial_2 b_1 = - \partial_1 \Gamma + b_2 \partial_2 b_1,
\end{split}
\end{equation}
   a direct computation yields
\begin{equation}\nonumber
\begin{split}
  \int_0^t \|\partial_1 \Gamma\|_{L^2}^2 \, d\tau &\lesssim \int_0^t \|\Omega\|_{L^2}^2 \, d\tau + \int_0^t \|\partial_2 b\|_{H^2}^2 \, d\tau+ \sup_{0 \leq \tau \leq t} \|b_2\|_{H^2}^2 \int_0^t \|\p_2 b \|_{L^2}^2 \, d\tau.
\end{split}
\end{equation}

$\bullet$ \textbf{Estimate of $\partial_2 \Gamma$:}
Recalling the definition of $\Gamma$, we have
$\partial_2 \Gamma = \partial_2 (b_2 + P(\rho + 1) - P(1)) + b_2 \partial_2 b_2.$
Appealing to H\"older's inequality and Proposition \ref{prop-nonlinear-func}, we readily obtain
\begin{equation}\nonumber
\begin{split}
  \int_0^t \|\partial_2 \Gamma\|_{L^2}^2 \, d\tau &\lesssim \sup_{0 \leq \tau \leq t} (1 + \|(\rho, b_2)\|_{H^2}) \int_0^t \|\partial_2(\rho, b_2)\|_{L^2}^2 \, d\tau.
\end{split}
\end{equation}
Aggregating all the estimates derived in Part 1 and Part 2 successfully completes the proof of this lemma.
\end{proof}

\begin{lemma}
For the energies defined in \eqref{energy-set}, the following estimate holds for all $t > 0$:
\begin{equation}\nonumber
\begin{split}
\mathcal{E}_4(t) \triangleq \int_0^t \int_{\mathbb{R}^2} (|b_1 |^2 + |u_1|^2)|\partial_1^3 b_2|^2 \,dx\,d\tau
    \lesssim \mathcal{E}_{\mathrm{total}}^2(t).
  \end{split}
\end{equation}
\end{lemma}
\begin{proof}
We begin by recalling the equation for $u_1$ from \eqref{equ1}, which furnishes the following identity:
\begin{equation}\nonumber
\partial_2 b_1 = \partial_t u_1 + u\cdot\nabla u_1 - \Delta u_1 - (\Omega - \p_2 b_1) + \frac{\rho}{\rho+1}\big(\Delta u_1 + \Omega \big).
\end{equation}
Multiplying the above identity by $b_1$, integrating with respect to $y_2$ over $(-\infty, x_2)$, and then taking the $L^2$ inner product with $|\partial_1^3 b_2|^2$, we obtain
\begin{equation}\nonumber
\begin{split}
\frac{1}{2}\int_{\mathbb{R}^2} |b_1 |^2 |\partial_1^3 b_2|^2 \,dx = \sum_{i=1}^5 J_i,
\end{split}
\end{equation}
where
\begin{equation}\nonumber
\begin{split}
J_1 &= \int_{\mathbb{R}^2} \int_{-\infty}^{x_2} b_1 \partial_t u_1 \, dy_2 \, |\partial_1^3 b_2|^2 \,dx,\\
J_2 &= -\int_{\mathbb{R}^2} \int_{-\infty}^{x_2} b_1 \Delta u_1 \, dy_2 \, |\partial_1^3 b_2|^2 \,dx,\\
J_3 &= -\int_{\mathbb{R}^2} \int_{-\infty}^{x_2} b_1 (\Omega - \partial_2 b_1) \, dy_2 \, |\partial_1^3 b_2|^2 \,dx,\\
J_4 &= \int_{\mathbb{R}^2} \int_{-\infty}^{x_2} b_1 u \cdot \nabla u_1 \, dy_2 \, |\partial_1^3 b_2|^2 \,dx,\\
J_5 &= \int_{\mathbb{R}^2} \int_{-\infty}^{x_2} b_1 \frac{\rho}{\rho + 1} (\Delta u_1 + \Omega) \, dy_2 \, |\partial_1^3 b_2|^2 \,dx.
\end{split}
\end{equation}

$\bullet$ \textbf{Estimate of $J_1$:}
Performing integration by parts in time yields
\begin{equation}\nonumber
\begin{split}
J_1 &= \frac{d}{dt}\int_{\mathbb{R}^2} \int_{-\infty}^{x_2} b_1  u_1 \, dy_2 \, |\partial_1^3 b_2|^2 \,dx - \int_{\mathbb{R}^2} \int_{-\infty}^{x_2} \partial_t b_1  u_1 \, dy_2 \, |\partial_1^3 b_2|^2 \,dx \\
&\quad - 2 \int_{\mathbb{R}^2} \int_{-\infty}^{x_2}  b_1  u_1 \, dy_2 \, \partial_t \partial_1^3 b_2 \, \partial_1^3 b_2 \,dx \\
&= \frac{d}{dt}\int_{\mathbb{R}^2} \int_{-\infty}^{x_2} b_1  u_1 \, dy_2 \, |\partial_1^3 b_2|^2 \,dx \\
&\quad - \int_{\mathbb{R}^2} \int_{-\infty}^{x_2} (\partial_2 u_1 - u \cdot \nabla b_1 + b \cdot \nabla u_1 - b_1 \nabla \cdot u)   u_1 \, dy_2 \, |\partial_1^3 b_2|^2 \,dx \\
&\quad - 2 \int_{\mathbb{R}^2} \int_{-\infty}^{x_2}  b_1  u_1 \, dy_2  \, \partial_1^3 (-\partial_1 u_1 - u \cdot \nabla b_2 + b \cdot \nabla u_2 - b_2 \nabla \cdot u)\partial_1^3 b_2 \,dx \\
&\triangleq J_{1,1} + J_{1,2} + J_{1,3}.
\end{split}
\end{equation}
For $J_{1,2}$, we deduce that
\begin{equation}\nonumber
\begin{split}
  &J_{1,2} + \frac{1}{2} \int_{\mathbb{R}^2}  |u_1|^2 |\partial_1^3 b_2|^2 \,dx \\
  &\lesssim \|(- u \cdot \nabla b_1 + b \cdot \nabla u_1 - b_1 \nabla \cdot u)   u_1\|_{L_{x_1}^\infty L_{x_2}^1} \|\partial_1^3 b_2\|_{L^2}^2 \\
  &\lesssim \big(\|u\|_{L^\infty}\|\nabla b_1\|_{L_{x_1}^\infty L_{x_2}^2} + \|b\|_{L^\infty}\|\nabla u_1\|_{L_{x_1}^\infty L_{x_2}^2} + \|b_1\|_{L^\infty}\|\nabla \cdot u \|_{L_{x_1}^\infty L_{x_2}^2}\big) \|u_1\|_{L_{x_1}^\infty L_{x_2}^2}\|\partial_1^3 b_2\|_{L^2}^2 \\
  &\lesssim \big(\|u\|_{H^1} \|\nabla u\|_{H^1} \|\partial_2 b\|_{H^1} + \|b\|_{H^1}^\frac{1}{2}\|\partial_2 b\|_{H^1}^\frac{1}{2} \|\nabla u\|_{H^1} \|u_1\|_{L^2}^\frac{1}{2}\|\partial_1 u_1\|_{L^2}^\frac{1}{2}\big) \|\partial_1^3 b_2\|_{L^2}^2.
\end{split}
\end{equation}
For $J_{1,3}$, we decompose it as
\begin{equation}\nonumber
\begin{split}
J_{1,3}  &= - 2 \int_{\mathbb{R}^2} \int_{-\infty}^{x_2}  b_1  u_1 \, dy_2  \, \partial_1^3 (-\partial_1 u_1 + b \cdot \nabla  u_2 - b_2 \nabla \cdot u)\partial_1^3 b_2 \,dx \\
&\quad +2 \int_{\mathbb{R}^2} \int_{-\infty}^{x_2}  b_1  u_1 \, dy_2  \, \sum_{k = 1}^3 \mathcal{C}_{3}^k \partial_1^k u \cdot \nabla  \partial_1^{3-k} b_2 \, \partial_1^3 b_2 \,dx \\
&\quad +  2 \int_{\mathbb{R}^2} \int_{-\infty}^{x_2}  b_1  u_1 \, dy_2  \, u \cdot \nabla \partial_1^3 b_2 \, \partial_1^3 b_2 \,dx\\
\triangleq& \, J_{1,3,1} + J_{1,3,2} + J_{1,3,3}.
\end{split}
\end{equation}
For the first two terms $J_{1,3,1}$ and $J_{1,3,2}$, we establish that
\begin{equation}\nonumber
\begin{split}
J_{1,3,1} + J_{1,3,2} &\lesssim \|b_1 u_1\|_{L_{x_1}^\infty L_{x_2}^1}
\|\partial_1^3(\partial_1 u_1 - b \cdot \nabla u_2 + b_2 \nabla \cdot u)\|_{L^2} \|\partial_1^3 b_2\|_{L^2} \\
&\quad + \|b_1 u_1\|_{L_{x_1}^\infty L_{x_2}^1}
\Big\|\sum_{k = 1}^3 \mathcal{C}_{3}^k \partial_1^k u \cdot \nabla  \partial_1^{3-k} b_2\Big\|_{L^2} \|\partial_1^3 b_2\|_{L^2} \\
&\lesssim \|b_1\|_{L^2}^\frac{1}{2}\|\partial_1 b_1\|_{L^2}^\frac{1}{2}\|u_1\|_{L^2}^\frac{1}{2}\|\partial_1 u_1\|_{L^2}^\frac{1}{2} \big(\|\partial_1^4 u_1\|_{L^2} + \|b\|_{H^3} \|\nabla u\|_{H^3}\big) \|\partial_1^3 b_2\|_{L^2}.
\end{split}
\end{equation}
For the last term $J_{1,3,3}$, integration by parts gives
\begin{equation}\nonumber
\begin{split}
J_{1,3,3} =
& - \int_{\mathbb{R}^2} \int_{-\infty}^{x_2}  b_1  u_1 \, dy_2  \, \nabla \cdot u |\partial_1^3 b_2|^2 \,dx  - \int_{\mathbb{R}^2} ( b_1  u_1)  u_2 |\partial_1^3 b_2|^2 \,dx \\
& - \int_{\mathbb{R}^2} \int_{-\infty}^{x_2} \partial_1( b_1  u_1) \, dy_2  \, u_1 |\partial_1^3 b_2|^2 \,dx \\
\lesssim& \|b_1 u_1\|_{L_{x_1}^\infty L_{x_2}^1} \|\nabla \cdot u\|_{L^\infty} \|\partial_1^3 b_2\|_{L^2}^2  + \|b_1\|_{L^\infty} \|u_1\|_{L^\infty} \|u_2\|_{L^\infty} \|\partial_1^3 b_2\|_{L^2}^2 \\
& + \big(\|\partial_1 b_1\|_{L_{x_1}^\infty L_{x_2}^2} \|u_1\|_{L_{x_1}^\infty L_{x_2}^2} +
\|\partial_1 u_1\|_{L_{x_1}^\infty L_{x_2}^2} \|b_1\|_{L_{x_1}^\infty L_{x_2}^2}\big)\|u_1\|_{L^\infty} \|\partial_1^3 b_2\|_{L^2}^2 \\
\lesssim& \|b_1\|_{L^2}^\frac{1}{2} \|\partial_1 b_1\|_{L^2}^\frac{1}{2}\|u_1\|_{L^2}^\frac{1}{2} \|\partial_1 u_1\|_{L^2}^\frac{1}{2} \|\nabla \cdot u\|_{H^2} \|\partial_1^3 b_2\|_{L^2}^2\\
& + \|(b_1, u)\|_{L^2}^\frac{3}{4} \|\nabla (b_1, u)\|_{H^1}^\frac{9}{4}\|\partial_1^3 b_2\|_{L^2}^2\\
& + \big(\|\partial_1 b_1 \|_{H^1} \|u_1\|_{L^2}^\frac{1}{2} \|\partial_1 u_1\|_{L^2}^\frac{1}{2} + \|\partial_1 u_1 \|_{H^1} \|b_1\|_{L^2}^\frac{1}{2} \|\partial_1 b_1\|_{L^2}^\frac{1}{2}\big) \\
& \cdot \|u_1\|_{L^2}^\frac{1}{2} \|\nabla u_1\|_{H^1}^\frac{1}{2} \|\partial_1^3 b_2\|_{L^2}^2.
\end{split}
\end{equation}
Therefore,
\begin{equation}\label{J1}
  \begin{split}
    &\Big|\int_0^t J_1 \, d\tau + \frac{1}{2}\int_0^t\int_{\mathbb{R}^2}|u_1|^2|\partial_1^3 b_2|^2 \,dx d\tau \Big| \\
    \lesssim& \sup_{0 \leq \tau \leq t} \|(b, u)\|_{H^3}^4 +
    \sup_{0 \leq \tau \leq t} \|(b, u)\|_{H^3}^2 \int_0^t \big(\|\partial_2 b\|_{H^1}^2 + \|\nabla u\|_{H^2}^2\big) \,d\tau.
  \end{split}
\end{equation}

$\bullet$ \textbf{Estimate of $J_2$:}
Invoking the equation for $\Omega$ from \eqref{eqomega}, we reformulate $J_2$ as
\begin{equation}\nonumber
\begin{split}
J_2 &= -\int_{\mathbb{R}^2} \int_{-\infty}^{x_2} b_1 \partial_1^2 u_1 \, dy_2 \, |\partial_1^3 b_2|^2 \,dx -\int_{\mathbb{R}^2} \int_{-\infty}^{x_2} b_1 \partial_2^2 u_1 \, dy_2 \, |\partial_1^3 b_2|^2 \,dx \\
&= -\int_{\mathbb{R}^2} \int_{-\infty}^{x_2} b_1  \frac{1}{2}\big(\partial_t \Omega + u\cdot \nabla \Omega - \partial_2^2 u_1 - \partial_1 \partial_2 u_2 + Q\big) \, dy_2 \, |\partial_1^3 b_2|^2 \,dx \\
 &\quad -\int_{\mathbb{R}^2} \int_{-\infty}^{x_2} b_1 \partial_2^2 u_1 \, dy_2 \, |\partial_1^3 b_2|^2 \,dx \\
 &= -\frac{1}{2}\int_{\mathbb{R}^2} \int_{-\infty}^{x_2} b_1 (\partial_2^2 u_1 - \partial_1 \partial_2 u_2) \, dy_2 \, |\partial_1^3 b_2|^2 \,dx \\
 &\quad -\frac{1}{2}\int_{\mathbb{R}^2} \int_{-\infty}^{x_2} b_1 Q \, dy_2 \, |\partial_1^3 b_2|^2 \,dx \\
 &\quad -\frac{1}{2}\int_{\mathbb{R}^2} \int_{-\infty}^{x_2} b_1(\partial_t \Omega + u\cdot \nabla \Omega) \, dy_2 \, |\partial_1^3 b_2|^2 \,dx \\
 &\triangleq J_{2,1} + J_{2,2} + J_{2,3}.
\end{split}
\end{equation}
Here, $Q$ is defined in \eqref{Q}, which represents a nonlinear term containing at least one velocity derivative $\partial_i u_j$.

\quad $\diamond$ \textbf{Estimate of $J_{2,1}$:}
Integrating by parts and applying Young's inequality yield
\begin{equation}\nonumber
\begin{split}
J_{2,1} &=- \frac{1}{2}\int_{\mathbb{R}^2}  b_1 (\partial_2 u_1 - \partial_1 u_2) |\partial_1^3 b_2|^2 \,dx + \frac{1}{2}\int_{\mathbb{R}^2} \int_{-\infty}^{x_2} \partial_2 b_1 (\partial_2 u_1 - \partial_1 u_2) \, dy_2 \, |\partial_1^3 b_2|^2 \,dx \\
&\leq \frac{1}{16} \int_{\mathbb{R}^2}  |b_1|^2 |\partial_1^3 b_2|^2 \,dx + C \int_{\mathbb{R}^2}  |\nabla u|^2  |\partial_1^3 b_2|^2 \,dx + C \|\partial_2 b_1\|_{L_{x_1}^\infty L_{x_2}^2} \|\nabla u\|_{L_{x_1}^\infty L_{x_2}^2} \|\partial_1^3 b_2\|_{L^2}^2 \\
&\leq \frac{1}{16} \int_{\mathbb{R}^2}  |b_1|^2 |\partial_1^3 b_2|^2 \,dx + C \|\nabla u\|_{H^2}^2 \|\partial_1^3 b_2\|_{L^2}^2 + C \|\partial_2 b_1\|_{H^1} \|\nabla u\|_{H^1} \|\partial_1^3 b_2\|_{L^2}^2.
\end{split}
\end{equation}

\quad $\diamond$ \textbf{Estimate of $J_{2,2}$:}
Standard anisotropic Sobolev embeddings ensure that
\begin{equation}\nonumber
\begin{split}
J_{2,2} &\lesssim \|b_1\|_{L_{x_1}^\infty L_{x_2}^2} \|Q\|_{L_{x_1}^\infty L_{x_2}^2} \|\partial_1^3 b_2\|_{L^2}^2 \\
&\lesssim \|b_1\|_{L^2}^\frac{1}{2} \|\partial_1 b_1\|_{L^2}^\frac{1}{2} \|(\rho, b)\|_{H^3}^\frac{1}{2} \|\partial_2(\rho, b)\|_{H^2}^\frac{1}{2} \|\nabla u\|_{H^3} \|\partial_1^3 b_2\|_{L^2}^2.
\end{split}
\end{equation}

\quad $\diamond$ \textbf{Estimate of $J_{2,3}$:}
We now address the following challenging term by decomposing it into four components:
\begin{equation}\nonumber
\begin{split}
J_{2,3} &= - \frac{1}{2}\frac{d}{dt} \int_{\mathbb{R}^2} \int_{-\infty}^{x_2} b_1  \Omega \, dy_2 \, |\partial_1^3 b_2|^2 \,dx \\
&\quad + \frac{1}{2}\int_{\mathbb{R}^2} \int_{-\infty}^{x_2} \partial_t b_1  \Omega \, dy_2 \, |\partial_1^3 b_2|^2 \,dx \\
&\quad - \frac{1}{2}\int_{\mathbb{R}^2} \int_{-\infty}^{x_2} b_1 u\cdot \nabla \Omega \, dy_2 \, |\partial_1^3 b_2|^2 \,dx \\
&\quad + \frac{1}{2}\int_{\mathbb{R}^2} \int_{-\infty}^{x_2} b_1  \Omega \, dy_2 \, \frac{d}{dt}|\partial_1^3 b_2|^2 \,dx \\
&\triangleq J_{2,3,1} + J_{2,3,2} + J_{2,3,3} + J_{2,3,4}.
\end{split}
\end{equation}
For $J_{2,3,2}$, substituting the equation for $b_1$ gives
\begin{equation}\nonumber
\begin{split}
J_{2,3,2} &= \frac{1}{2}\int_{\mathbb{R}^2} \int_{-\infty}^{x_2} (\partial_2 u_1 - u \cdot \nabla b_1 + b \cdot \nabla u_1 - b_1 \nabla\cdot u) \Omega \, dy_2 \, |\partial_1^3 b_2|^2 \,dx \\
&\lesssim \|(\partial_2 u_1 - u \cdot \nabla b_1 + b \cdot \nabla u_1 - b_1 \nabla\cdot u)\|_{L_{x_1}^\infty L_{x_2}^2} \|\Omega\|_{L_{x_1}^\infty L_{x_2}^2} \|\partial_1^3 b_2\|_{L^2}^2 \\
&\lesssim \big(\|\partial_2 u_1\|_{H^1} + \|u\|_{H^2} \|\nabla b_1\|_{H^1} + \|b\|_{H^2} \|\nabla u\|_{H^1}\big) \|\Omega\|_{H^1} \|\partial_1^3 b_2\|_{L^2}^2.
\end{split}
\end{equation}
For $J_{2,3,3}$, we infer that
\begin{equation}\nonumber
\begin{split}
J_{2,3,3} &\lesssim \|b_1 u \cdot \nabla \Omega\|_{L_{x_1}^\infty L_{x_2}^1} \|\partial_1^3 b_2\|_{L^2}^2 \\
&\lesssim \|b_1\|_{L^\infty} \|u\|_{L_{x_1}^\infty L_{x_2}^2} \|\nabla \Omega\|_{L_{x_1}^\infty L_{x_2}^2} \|\partial_1^3 b_2\|_{L^2}^2 \\
&\lesssim \|b_1\|_{L^2}^\frac{1}{2} \|\nabla b_1\|_{H^1}^\frac{1}{2} \|u\|_{L^2}^\frac{1}{2}\|\partial_1 u\|_{L^2}^\frac{1}{2} \|\nabla \Omega\|_{H^1} \|\partial_1^3 b_2\|_{L^2}^2.
\end{split}
\end{equation}
For $J_{2,3,4}$, exploiting the following identity
\begin{equation}\nonumber
\begin{split}
 \frac{1}{2} \frac{d}{dt} |\partial_1^3 b_2|^2 = - \partial_1^4 u_1 \partial_1^3 b_2  - \partial_1^3 (u \cdot \nabla b_2) \partial_1^3 b_2 + \partial_1^3 (b \cdot \nabla u_2) \partial_1^3 b_2 - \partial_1^3 (b_2 \nabla \cdot u) \partial_1^3 b_2,
\end{split}
\end{equation}
we can decompose $J_{2,3,4}$ into
\begin{equation}\nonumber
\begin{split}
J_{2,3,4} =& \int_{\mathbb{R}^2} \int_{-\infty}^{x_2} b_1  \Omega \, dy_2 \Big(- \partial_1^4 u_1 - \sum_{k = 1}^3 \mathcal{C}_{3}^k \partial_1^k u \cdot \nabla \partial_1^{3-k} b_2\\
 &+ \partial_1^3 (b \cdot \nabla u_2) - \partial_1^3 (b_2 \nabla \cdot u)\Big) \partial_1^3 b_2 \,dx \\
 & - \frac{1}{2}\int_{\mathbb{R}^2} \int_{-\infty}^{x_2} b_1  \Omega \, dy_2 \, u \cdot \nabla |\partial_1^3 b_2|^2 \,dx \\
\triangleq& \, J_{2,3,4,1} + J_{2,3,4,2}.
\end{split}
\end{equation}
For $J_{2,3,4,1}$, a direct application of H\"older's inequality furnishes
\begin{equation}\nonumber
\begin{split}
J_{2,3,4,1} &\lesssim \|b_1  \Omega\|_{L_{x_1}^\infty L_{x_2}^1} \|\partial_1^3 b_2\|_{L^2} \\
 &\quad \cdot \big\|- \partial_1^4 u_1 + \sum_{k = 1}^3 \mathcal{C}_{3}^k \partial_1^k u \cdot \nabla \partial_1^{3-k} b_2 + \partial_1^3 (b \cdot \nabla u_2) - \partial_1^3 (b_2 \nabla \cdot u)\big\|_{L^2} \\
&\lesssim \|b_1\|_{L^2}^\frac{1}{2}\|\partial_1 b_1\|_{L^2}^\frac{1}{2}\|\Omega\|_{L^2}^\frac{1}{2}\|\partial_1 \Omega\|_{L^2}^\frac{1}{2} \\
&\quad \cdot \big(\|\partial_1^4 u_1\|_{L^2} + \|\partial_1 u\|_{H^3} \|\nabla b_2\|_{H^2} + \|b\|_{H^3} \|\nabla u\|_{H^3}\big) \|\partial_1^3 b_2\|_{L^2}.
\end{split}
\end{equation}
For $J_{2,3,4,2}$, integration by parts yields
\begin{equation}\nonumber
\begin{split}
J_{2,3,4,2} &=  \frac{1}{2}\int_{\mathbb{R}^2} \int_{-\infty}^{x_2} b_1  \Omega \, dy_2 \, \nabla \cdot u |\partial_1^3 b_2|^2 \,dx \\
&\quad + \frac{1}{2}\int_{\mathbb{R}^2} \int_{-\infty}^{x_2} \partial_1(b_1  \Omega) \, dy_2 \, u_1 |\partial_1^3 b_2|^2 \,dx + \frac{1}{2}\int_{\mathbb{R}^2} (b_1  \Omega)  u_2 |\partial_1^3 b_2|^2 \,dx.
\end{split}
\end{equation}
Appealing to H\"older's inequality and the Sobolev embedding theorem, we find
\begin{equation}\nonumber
\begin{split}
J_{2,3,4,2} &\lesssim \|b_1 \Omega\|_{L_{x_1}^\infty L_{x_2}^1} \|\nabla \cdot u\|_{L^\infty} \|\partial_1^3 b_2\|_{L^2}^2 \\
&\quad + \|\partial_1(b_1 \Omega)\|_{L_{x_1}^\infty L_{x_2}^1} \|u_1\|_{L^\infty} \|\partial_1^3 b_2\|_{L^2}^2  + \|b_1 \Omega u_2\|_{L^\infty} \|\partial_1^3 b_2\|_{L^2}^2 \\
&\lesssim \|b_1\|_{H^1} \|\Omega\|_{H^1} \|\nabla \cdot u\|_{H^2} \|\partial_1^3 b_2\|_{L^2}^2 \\
&\quad + \|b_1\|_{H^1}^\frac{1}{2} \|\partial_1 b_1\|_{H^1}^\frac{1}{2} \|\Omega\|_{H^2} \|u_1\|_{L^2}^\frac{1}{2} \|\nabla^2 u_1\|_{L^2}^\frac{1}{2} \|\partial_1^3 b_2\|_{L^2}^2 \\
&\quad + \|b_1\|_{L^2}^\frac{1}{2} \|\nabla^2 b_1\|_{L^2}^\frac{1}{2} \|\Omega\|_{H^2} \|u_2\|_{L^2}^\frac{1}{2} \|\nabla^2 u_2\|_{L^2}^\frac{1}{2} \|\partial_1^3 b_2\|_{L^2}^2.
\end{split}
\end{equation}
Consequently, the time integral of $J_2$ satisfies
\begin{equation}\label{J2}
  \begin{split}
    \int_0^t J_2 \, d\tau &\leq \frac{1}{16} \int_0^t \int_{\mathbb{R}^2}|b_1|^2 |\partial_1^3 b_2|^2 \,dx \,d\tau \\
    &+ C \sup_{0 \leq \tau \leq t} \|(\rho, b, u)\|_{H^3}^2 \int_0^t \big(\|(\partial_2 b, \Omega)\|_{H^2}^2 + \|\nabla u\|_{H^3}^2\big) \, d\tau \\
    &\quad + C \sup_{0 \leq \tau \leq t} \Big| \int_{\mathbb{R}^2} \int_{-\infty}^{x_2} b_1  \Omega \, dy_2 \, |\partial_1^3 b_2|^2 \,dx \Big| \\
    &\leq \frac{1}{16} \int_0^t \int_{\mathbb{R}^2}|b_1|^2 |\partial_1^3 b_2|^2 \,dx \,d\tau \\
    &+ C \sup_{0 \leq \tau \leq t} \|(\rho, b, u)\|_{H^3}^2 \int_0^t \big(\|(\partial_2 b, \Omega)\|_{H^2}^2 + \|\nabla u\|_{H^3}^2\big) \, d\tau \\
    &\quad + C \sup_{0 \leq \tau \leq t} \|(\rho, b)\|_{H^3}^4.
  \end{split}
\end{equation}

$\bullet$ \textbf{Estimate of $J_3$:}
We now turn to the most challenging term $J_3$.
Utilizing the identity
\begin{equation}\nonumber
\begin{split}
\Omega - \partial_2 b_1 = - \partial_1 \Gamma + b_2 \partial_2 b_1,
\end{split}
\end{equation}
we decompose $J_3$ as
\begin{equation}\nonumber
\begin{split}
J_{3} &= \int_{\mathbb{R}^2} \int_{-\infty}^{x_2} b_1 \partial_1 \Gamma \, dy_2 \, |\partial_1^3 b_2|^2 \,dx \\
&\quad + \int_{\mathbb{R}^2} \int_{-\infty}^{x_2} b_1 ( b_2 \partial_2 b_1) \, dy_2 \, |\partial_1^3 b_2|^2 \,dx \\
&\triangleq J_{3,1} + J_{3,2}.
\end{split}
\end{equation}

\quad $\diamond$ \textbf{Estimate of $J_{3,1}$:}
According to \eqref{eq-div-rho+b2}, $\Gamma$ obeys the following elliptic equation:
\begin{equation}\nonumber
\begin{split}
 \Delta \Gamma
=& -\nabla \cdot u_t + \Delta \nabla \cdot u + \mathcal{H}.
\end{split}
\end{equation}
Here,
\begin{equation}\nonumber
\begin{split}
\mathcal{H} &= - \nabla \cdot (u \cdot \nabla u) + \nabla \cdot \big( b\cdot \nabla b - \tfrac{1}{2} \nabla |b_1|^2\big) \\
&\quad + \nabla \cdot \Big[\frac{\rho}{\rho + 1}\left(
		\begin{array}{c}
			-\Omega - \Delta u_1 \\
			\partial_2 P - \Delta u_2 - b \cdot \nabla b_2 + \tfrac{1}{2} \partial_2 |b|^2 \\
		\end{array}
		\right) \Big].
\end{split}
\end{equation}
Thus, $J_{3,1}$ can be expressed as
\begin{equation}\nonumber
\begin{split}
J_{3,1} &= \int_{\mathbb{R}^2} \int_{-\infty}^{x_2} b_1 (-\Delta)^{-1}\partial_1\nabla \cdot u_t \, dy_2 \, |\partial_1^3 b_2|^2 \,dx \\
&\quad + \int_{\mathbb{R}^2} \int_{-\infty}^{x_2} b_1 \partial_1 \nabla \cdot u \, dy_2 \, |\partial_1^3 b_2|^2 \,dx \\
&\quad - \int_{\mathbb{R}^2} \int_{-\infty}^{x_2} b_1 (-\Delta)^{-1}\partial_1 \mathcal{H} \, dy_2 \, |\partial_1^3 b_2|^2 \,dx \\
&\triangleq J_{3,1,1} + J_{3,1,2} + J_{3,1,3}.
\end{split}
\end{equation}
For $J_{3,1,1}$, we first rewrite it as follows to circumvent dealing with $|\nabla|^{-1} u_t$:
\begin{equation}\nonumber
\begin{split}
J_{3,1,1} &= \int_{\mathbb{R}^2} \int_{-\infty}^{x_2} \int_{-\infty}^{x_1} \partial_1 b_1 (-\Delta)^{-1}\partial_1\nabla \cdot u_t \, dy_1 dy_2 \, |\partial_1^3 b_2|^2 \,dx \\
&\quad + \int_{\mathbb{R}^2} \int_{-\infty}^{x_2} \int_{-\infty}^{x_1} b_1 (-\Delta)^{-1}\partial_1^2 \nabla \cdot u_t \, dy_1 dy_2 \, |\partial_1^3 b_2|^2 \,dx \\
&\triangleq J_{3,1,1,1} + J_{3,1,1,2}.
\end{split}
\end{equation}
By virtue of the Calder\'on--Zygmund theorem and $\nabla \cdot b = 0$, we can derive
\begin{equation}\nonumber
\begin{split}
J_{3,1,1,1} &\lesssim \|\partial_1 b_1\|_{L^2}\| (-\Delta)^{-1}\partial_1\nabla \cdot u_t\|_{L^2} \|\partial_1^3 b_2\|_{L^2}^2 \\
&\lesssim \|\partial_2 b_2\|_{L^2}\|u_t\|_{L^2} \|\partial_1^3 b_2\|_{L^2}^2.
\end{split}
\end{equation}
For $J_{3,1,1,2}$, integration by parts gives
\begin{equation}\nonumber
\begin{split}
J_{3,1,1,2} &= \frac{d}{dt}\int_{\mathbb{R}^2} \int_{-\infty}^{x_2} \int_{-\infty}^{x_1} b_1 (-\Delta)^{-1}\partial_1^2 \nabla \cdot u \, dy_1 dy_2 \, |\partial_1^3 b_2|^2 \,dx \\
&\quad - \int_{\mathbb{R}^2} \int_{-\infty}^{x_2} \int_{-\infty}^{x_1} \partial_t b_1 (-\Delta)^{-1}\partial_1^2 \nabla \cdot u \, dy_1 dy_2 \, |\partial_1^3 b_2|^2 \,dx \\
&\quad - 2\int_{\mathbb{R}^2} \int_{-\infty}^{x_2} \int_{-\infty}^{x_1} b_1 (-\Delta)^{-1}\partial_1^2 \nabla \cdot u \, dy_1 dy_2 \, \partial_t \partial_1^3 b_2 \, \partial_1^3 b_2 \,dx \\
&\triangleq J_{3,1,1,2,1} + J_{3,1,1,2,2} + J_{3,1,1,2,3}.
\end{split}
\end{equation}
For $J_{3,1,1,2,2}$, arguing in a similar vein as for $J_{3,1,1,1}$, we have
\begin{equation}\nonumber
\begin{split}
J_{3,1,1,2,2} &\lesssim \| \partial_t b_1\|_{L^2}\| (-\Delta)^{-1}\partial_1^2 \nabla \cdot u\|_{L^2} \|\partial_1^3 b_2\|_{L^2}^2 \\
&\lesssim \| \partial_t b_1\|_{L^2}\| \nabla \cdot u\|_{L^2} \|\partial_1^3 b_2\|_{L^2}^2.
\end{split}
\end{equation}
For $J_{3,1,1,2,3}$, we observe that
\begin{equation}\nonumber
  \partial_t \partial_1^3 b_2 \, \partial_1^3 b_2 = - u \cdot \nabla \partial_1^3 b_2 \, \partial_1^3 b_2 + \mathcal{U},
\end{equation}
where
$$\mathcal{U} \triangleq - \partial_1^3 b_2 \, \partial_1^4 u_1 - \sum_{k = 1}^3 \mathcal{C}_{3}^k \partial_1^k u \cdot \nabla \partial_1^{3-k} b_2 \, \partial_1^3 b_2 + \partial_1^3 b_2 \, \partial_1^3(b \cdot \nabla u_2 - b_2 \nabla \cdot u).$$
Integrating by parts, we can rewrite $J_{3,1,1,2,3}$ as
\begin{equation}\nonumber
\begin{split}
J_{3,1,1,2,3} =& 2\int_{\mathbb{R}^2} \int_{-\infty}^{x_2} \int_{-\infty}^{x_1} b_1 (-\Delta)^{-1}\partial_1^2 \nabla \cdot u \, dy_1 dy_2 \, u\cdot \nabla \partial_1^3 b_2 \, \partial_1^3 b_2 \,dx \\
&\quad - 2\int_{\mathbb{R}^2} \int_{-\infty}^{x_2} \int_{-\infty}^{x_1} b_1 (-\Delta)^{-1}\partial_1^2 \nabla \cdot u \, dy_1 dy_2  \mathcal{U} \,dx \\
=& - \int_{\mathbb{R}^2} \int_{-\infty}^{x_2} \int_{-\infty}^{x_1} b_1 (-\Delta)^{-1}\partial_1^2 \nabla \cdot u \, dy_1 dy_2 \, \nabla \cdot u \partial_1^3 b_2 \, \partial_1^3 b_2 \,dx  \\
&- \int_{\mathbb{R}^2} \int_{-\infty}^{x_2} b_1 (-\Delta)^{-1}\partial_1^2 \nabla \cdot u \,  dy_2 \, u_1 |\partial_1^3 b_2|^2 \,dx \\
& - \int_{\mathbb{R}^2} \int_{-\infty}^{x_1} b_1 (-\Delta)^{-1}\partial_1^2 \nabla \cdot u \, dy_1  \, u_2 |\partial_1^3 b_2|^2 \,dx \\
& - 2\int_{\mathbb{R}^2} \int_{-\infty}^{x_2} \int_{-\infty}^{x_1} b_1 (-\Delta)^{-1}\partial_1^2 \nabla \cdot u \, dy_1 dy_2  \mathcal{U} \,dx.
\end{split}
\end{equation}
Applying H\"older's inequality and the Sobolev embedding theorem, we bound $J_{3,1,1,2,3}$ as
\begin{equation}\nonumber
\begin{split}
J_{3,1,1,2,3} &\lesssim \|b_1\|_{L^2}\|\nabla \cdot u\|_{L^2}\|\nabla \cdot u\|_{L^\infty}\|\p_1^3 b_2\|_{L^2}^2\\
&\quad +\|b_1\|_{L_{x_1}^\infty L_{x_2}^2} \| (-\Delta)^{-1}\partial_1^2 \nabla \cdot u \|_{L_{x_1}^\infty L_{x_2}^2} \|u_1\|_{L^\infty} \|\partial_1^3 b_2\|_{L^2}^2 \\
&\quad + \|b_1\|_{L_{x_2}^\infty L_{x_1}^2} \| (-\Delta)^{-1}\partial_1^2 \nabla \cdot u \|_{L_{x_2}^\infty L_{x_1}^2} \|u_2\|_{L^\infty} \|\partial_1^3 b_2\|_{L^2}^2 \\
&\quad + \|b_1\|_{L^2} \|\nabla \cdot u\|_{L^2} \|\mathcal{U}\|_{L^1} \\
&\lesssim \|b_1\|_{L^2}^\frac{1}{2} \|\partial_1 b_1\|_{L^2}^\frac{1}{2} \| \nabla \cdot u \|_{L^2}^\frac{1}{2} \| \partial_1 \nabla \cdot u \|_{L^2}^\frac{1}{2} \|u_1\|_{L^2}^\frac{1}{2}\|\nabla u_1\|_{H^1}^\frac{1}{2} \|\partial_1^3 b_2\|_{L^2}^2 \\
&\quad + \|b_1\|_{L^2}^\frac{1}{2} \|\partial_2 b_1\|_{L^2}^\frac{1}{2} \| \nabla \cdot u \|_{L^2}^\frac{1}{2} \| \partial_2 \nabla \cdot u \|_{L^2}^\frac{1}{2} \|u_2\|_{L^2}^\frac{1}{2}\|\nabla u_2\|_{H^1}^\frac{1}{2} \|\partial_1^3 b_2\|_{L^2}^2 \\
&\quad + \|b_1\|_{L^2}\|\nabla \cdot u\|_{L^2} \big(\|b\|_{H^3} + \|b\|_{H^3}^2\big)\|\nabla u\|_{H^3}.
\end{split}
\end{equation}
For $J_{3,1,2}$, we expand it as
\begin{equation}\nonumber
\begin{split}
J_{3,1,2} &= \int_{\mathbb{R}^2} \int_{-\infty}^{x_2} b_1 \partial_1^2 u_1 \, dy_2 \, |\partial_1^3 b_2|^2 \,dx + \int_{\mathbb{R}^2} \int_{-\infty}^{x_2} b_1 \partial_1 \partial_2 u_2 \, dy_2 \, |\partial_1^3 b_2|^2 \,dx \\
&= \int_{\mathbb{R}^2} \int_{-\infty}^{x_2} b_1 \partial_1^2 u_1 \, dy_2 \, |\partial_1^3 b_2|^2 \,dx \\
&\quad + \int_{\mathbb{R}^2} b_1 \partial_1 u_2 |\partial_1^3 b_2|^2 \,dx  - \int_{\mathbb{R}^2} \int_{-\infty}^{x_2} \partial_2 b_1 \partial_1 u_2 \, dy_2 \, |\partial_1^3 b_2|^2 \,dx \\
&\triangleq J_{3,1,2,1} + J_{3,1,2,2} + J_{3,1,2,3}.
\end{split}
\end{equation}
The estimate for $J_{3,1,2,1}$ follows from an identical rationale as that for $J_2$, we have
\begin{equation}\nonumber
\begin{split}
\int_0^t J_{3,1,2,1} \, d\tau &\leq \frac{1}{16} \int_0^t \int_{\mathbb{R}^2}|b_1|^2 |\partial_1^3 b_2|^2 \,dx \,d\tau \\
&+ C \sup_{0 \leq \tau \leq t} \|(\rho, b, u)\|_{H^3}^2 \int_0^t \big(\|(\partial_2 b, \Omega)\|_{H^2}^2 + \|\nabla u\|_{H^3}^2\big) \, d\tau \\
    &\quad + C \sup_{0 \leq \tau \leq t} \|(\rho, b)\|_{H^3}^4.
\end{split}
\end{equation}
For $J_{3,1,2,2}$, Young's inequality gives
\begin{equation}\nonumber
\begin{split}
J_{3,1,2,2} \leq \frac{1}{16}\int_{\mathbb{R}^2} |b_1|^2  |\partial_1^3 b_2|^2 \,dx + \|\partial_1 u_2\|_{H^2}^2\|\partial_1^3 b_2\|_{L^2}^2.
\end{split}
\end{equation}
For $J_{3,1,2,3}$, direct computation yields
\begin{equation}\nonumber
\begin{split}
J_{3,1,2,3} &\lesssim \|\partial_2 b_1 \partial_1 u_2\|_{L_{x_1}^\infty L_{x_2}^1} \|\partial_1^3 b\|_{L^2}^2 \\
&\lesssim \|\partial_2 b_1\|_{H^1} \|\partial_1 u_2\|_{H^1} \|\partial_1^3 b\|_{L^2}^2.
\end{split}
\end{equation}
For $J_{3,1,3}$, by virtue of anisotropic-type estimates, we can derive
\begin{equation}\nonumber
\begin{split}
J_{3,1,3} &\lesssim \|b_1\|_{L_{x_1}^\infty L_{x_2}^2} \|(-\Delta)^{-1} \partial_1 \mathcal{H}\|_{L_{x_1}^\infty L_{x_2}^2} \|\partial_1^3 b_2\|_{L^2}^2 \\
&\lesssim \|b_1\|_{L^2}^\frac{1}{2}\|\p_2 b_2\|_{L^2}^\frac{1}{2}\|(-\Delta)^{-1} \partial_1 \mathcal{H}\|_{L^2}^\frac{1}{2}\|(-\Delta)^{-1} \partial_1^2 \mathcal{H}\|_{L^2}^\frac{1}{2}\|\partial_1^3 b_2\|_{L^2}^2.
\end{split}
\end{equation}
Recalling the definition of $\mathcal{H}$ and appealing to the Calder\'on--Zygmund theorem, we deduce that
\begin{equation}\label{est-H}
\begin{split}
\|(-\Delta)^{-1} \partial_1 \mathcal{H}\|_{L^2} \lesssim& \|u \cdot \nabla u\|_{L^2} + \sum_{i = 1}^2 \|\p_2 (b_2 b_i)\|_{L^2} + \|\p_1 (b_1 b_1)\|_{L^2} + \|\nabla |b_1|^2\|_{L^2}\\
 &+ \|\frac{\rho}{\rho+1}(\Delta u_1 + \Omega)\|_{L^2} \\ &+\|\frac{\rho}{\rho+1}(\p_2 P - \Delta u_2 - b \cdot \nabla b_2 + \frac{1}{2} \p_2 |b|^2)\|_{L^2} \\
 \lesssim& \|(\rho, b, u)\|_{L^2}^\frac{1}{2} \|(\p_2 \rho, \p_2 b, \Delta u, \Omega)\|_{H^2}^\frac{3}{2}.
\end{split}
\end{equation}
Applying the identical methodology to bound $\|(-\Delta)^{-1} \partial_1^2 \mathcal{H}\|_{L^2}$, it follows that
\begin{equation}\nonumber
\begin{split}
J_{3,1,3}
&\lesssim \|(\rho, b, u)\|_{H^3}^2\big( \|(\partial_2\rho, \partial_2 b, \Omega)\|_{H^2}^2 + \|\nabla u\|_{H^3}^2\big).
\end{split}
\end{equation}
Here, we have used the \textit{a priori} assumption \eqref{ansatz}.

\quad $\diamond$ \textbf{Estimate of $J_{3,2}$:}
For $J_{3,2}$, anisotropic-type estimates yield
\begin{equation}\nonumber
\begin{split}
J_{3,2} &\lesssim \|b_1(b_2 \partial_2 b_1)\|_{L_{x_1}^\infty L_{x_2}^1} \|\partial_1^3 b_2\|_{L^2}^2 \\
&\lesssim \big(\|b_1\|_{L_{x_1}^\infty L_{x_2}^2}\| b_2\|_{L^\infty} \|\partial_2 b_1\|_{L_{x_1}^\infty L_{x_2}^2} \big) \|\partial_1^3 b_2\|_{L^2}^2 \\
&\lesssim \|b_1\|_{L^2}^\frac{1}{2} \|\partial_1 b_1\|_{L^2}^\frac{1}{2} \| b_2\|_{H^2}^\frac{1}{2} \|\partial_2 b_2\|_{H^2}^\frac{1}{2} \|\partial_2 b_1\|_{H^2} \|\partial_1^3 b_2\|_{L^2}^2.
\end{split}
\end{equation}
Consolidating the above estimates for $J_3$, we obtain
\begin{equation}\label{J3}
\begin{split}
\Big|\int_0^t J_{3} \,d\tau\Big| &\leq \frac{1}{16} \int_0^t \int_{\mathbb{R}^2}|b_1|^2 |\partial_1^3 b_2|^2 \,dx \,d\tau + C \sup_{0 \leq \tau \leq t} \|(\rho, b, u)\|_{H^3}^4 \\
    &\quad + C \sup_{0 \leq \tau \leq t} \|(\rho, b, u)\|_{H^3}^2 \int_0^t \big(\|(\p_2 \rho, \partial_2 b, \Omega)\|_{H^2}^2 + \|\nabla u\|_{H^3}^2 + \|(u_t, b_t)\|_{L^2}^2\big) \, d\tau.
\end{split}
\end{equation}

$\bullet$ \textbf{Estimate of $J_4$:}
For $J_4$, invoking anisotropic-type estimates once more gives
\begin{equation}\label{J4}
\begin{split}
J_4 &\lesssim \|b_1 u\cdot \nabla u\|_{L_{x_1}^\infty L_{x_2}^1} \|\partial_1^3 b_2\|_{L^2}^2 \\
&\lesssim \|b_1\|_{L_{x_1}^\infty L_{x_2}^2} \|u\|_{L^\infty} \|\nabla u\|_{L_{x_1}^\infty L_{x_2}^2} \|\partial_1^3 b_2\|_{L^2}^2 \\
&\lesssim \|b_1\|_{L^2}^\frac{1}{2} \|\partial_1 b_1\|_{L^2}^\frac{1}{2} \|u\|_{L^2}^\frac{1}{2} \|\nabla^2 u\|_{L^2}^\frac{1}{2} \|\nabla u\|_{H^1} \|\partial_1^3 b_2\|_{L^2}^2.
\end{split}
\end{equation}

$\bullet$ \textbf{Estimate of $J_5$:}
In a parallel vein, for $J_5$, by virtue of Propositions \ref{prop-anisotropic-est} and \ref{prop-nonlinear-func}, we have
\begin{equation}\label{J5}
\begin{split}
J_5 &\lesssim \Big\|b_1 \frac{\rho}{\rho + 1} (\Omega + \Delta u_1)\Big\|_{L_{x_1}^\infty L_{x_2}^1} \|\partial_1^3 b_2\|_{L^2}^2 \\
&\lesssim \|b_1\|_{L_{x_1}^\infty L_{x_2}^2} \Big\|\frac{\rho}{\rho + 1}\Big\|_{L^\infty} \|(\Omega + \Delta u_1)\|_{L_{x_1}^\infty L_{x_2}^2} \|\partial_1^3 b_2\|_{L^2}^2 \\
&\lesssim \|b_1\|_{L^2}^\frac{1}{2} \|\partial_1 b_1\|_{L^2}^\frac{1}{2} \|\rho\|_{H^1}^\frac{1}{2} \|\partial_2 \rho\|_{H^1}^\frac{1}{2} \|(\Omega, \Delta u_1)\|_{H^1} \|\partial_1^3 b_2\|_{L^2}^2.
\end{split}
\end{equation}

Finally, collecting all the above estimates namely \eqref{J1}, \eqref{J2}, \eqref{J3}, \eqref{J4} and \eqref{J5}, invoking the \textit{a priori } assumption \eqref{ansatz}, we arrive at
\begin{equation}\nonumber
\begin{split}
\frac{1}{2}\int_{\mathbb{R}^2} (|b_1 |^2 + |u_1|^2)|\partial_1^3 b_2|^2 \,dx
    \leq& \frac{1}{8} \int_0^t \int_{\mathbb{R}^2}|b_1|^2 |\partial_1^3 b_2|^2 \,dx \,d\tau + C \sup_{0 \leq \tau \leq t} \|(\rho, b, u)\|_{H^3}^4 \\
    &+ C \sup_{0 \leq \tau \leq t} \|(\rho, b, u)\|_{H^3}^2 \\
    &\quad \cdot \int_0^t \big(\|(\partial_2 b, \partial_2\rho, \Omega)\|_{H^2}^2 + \|\nabla u\|_{H^3}^2 + \|(u_t, b_t)\|_{L^2}^2\big) \, d\tau.
\end{split}
\end{equation}
This successfully completes the proof of this lemma.
\end{proof}

\begin{lemma}
For the energies  defined in \eqref{energy-set}, the following estimate holds for all $t > 0$:
\begin{equation}\nonumber
\begin{split}
  \mathcal{E}_5(t) \triangleq \int_0^t \int_{\mathbb{R}^2} \Gamma^2 |\p_1^3 (\rho, b_2)|^2 \, dx d\tau
    \lesssim  \mathcal{E}_{\mathrm{total}}^\frac{3}{2}(t).
\end{split}
\end{equation}
\end{lemma}
\begin{proof}
Recalling the definition of $\Gamma$ given in \eqref{defgamma}, namely $\Gamma = b_2 + P(\rho + 1) - P(1) + \frac{1}{2} |b_2|^2$, and invoking \eqref{eq-div-rho+b2}, we obtain
  \begin{equation}\nonumber
\begin{split}
\Gamma =& (-\Delta)^{-1}\nabla \cdot u_t + \mathcal{V}.
\end{split}
\end{equation}
Here,
\begin{equation}\label{defv}
\begin{split}
\mathcal{V} \triangleq
& \nabla \cdot u + (-\Delta)^{-1}\big[\nabla \cdot (u \cdot \nabla u) - \nabla \cdot ( b\cdot \nabla b - \frac{1}{2} \nabla |b_1|^2)\big]\\
&+ (-\Delta)^{-1} \nabla \cdot \Big[\frac{\rho}{\rho + 1}\left(
	\begin{array}{c}
		\Omega + \Delta u_1 \\
		-\p_2 P + \Delta u_2 + b \cdot \nabla  b_2 - \frac{1}{2} \p_2 |b|^2 \\
	\end{array}
	\right)  \Big]. \\
\triangleq& \mathcal{V}_1 + \mathcal{V}_2 + \mathcal{V}_3.
\end{split}
\end{equation}
Let $r \in \{\p_1^3 \rho, \p_1^3 b_2\}$. Then
\begin{equation}\nonumber
  \begin{split}
    \int_{\mathbb{R}^2} \Gamma^2 r^2\,dx = 2 \int_{\mathbb{R}^2} \int_{-\infty}^{x_1} \p_1 \Gamma \, \Gamma \, dy_1 \, r^2\,dx \triangleq {W}_1 + {W}_2.
  \end{split}
\end{equation}

$\bullet$ \textbf{Estimate of ${W}_1$:}
We decompose ${W}_1$ as follows:
\begin{equation}\nonumber
  \begin{split}
    {W}_1 =&2 \int_{\mathbb{R}^2}\int_{-\infty}^{x_1} \p_1 (-\Delta)^{-1} \nabla \cdot u_t \Gamma \,dy_1 \, r^2\,dx \\
    =&2 \int_{\mathbb{R}^2}\int_{-\infty}^{x_2} \int_{-\infty}^{x_1}\p_1 \p_2 (-\Delta)^{-1} \nabla \cdot u_t  \Gamma \,dy \, r^2\,dx \\
    & + 2\int_{\mathbb{R}^2}\int_{-\infty}^{x_2} \int_{-\infty}^{x_1}\p_1 (-\Delta)^{-1} \nabla \cdot u_t  \p_2\Gamma \,dy \, r^2\,dx \\
        \triangleq& \; {W}_{1,1} + {W}_{1,2}.
  \end{split}
\end{equation}
For the term ${W}_{1,1}$, integrating by parts in time yields
\begin{equation}\nonumber
  \begin{split}
    {W}_{1,1} =& 2 \frac{d}{dt}\int_{\mathbb{R}^2}\int_{-\infty}^{x_2} \int_{-\infty}^{x_1}\p_1 \p_2 (-\Delta)^{-1} \nabla \cdot u  \Gamma \,dy \, r^2\,dx \\
    & -2 \int_{\mathbb{R}^2}\int_{-\infty}^{x_2} \int_{-\infty}^{x_1}\p_1 \p_2 (-\Delta)^{-1} \nabla \cdot u  \Gamma_t \,dy \, r^2\,dx \\
    & -4 \int_{\mathbb{R}^2}\int_{-\infty}^{x_2} \int_{-\infty}^{x_1}\p_1 \p_2 (-\Delta)^{-1} \nabla \cdot u  \Gamma \,dy \, r \, r_t\,dx \\
    \triangleq& \, {W}_{1,1,1} + {W}_{1,1,2} + {W}_{1,1,3}.
  \end{split}
\end{equation}
By H\"older's inequality and the Calder\'on--Zygmund theorem, we have
\begin{equation}\nonumber
  \begin{split}
    {W}_{1,1,2} \lesssim \|\nabla \cdot u\|_{L^2} \|\Gamma_t\|_{L^2} \|r\|_{L^2}^2.
  \end{split}
\end{equation}
For the term ${W}_{1,1,3}$, we exploit the identity
\begin{equation}\nonumber
  \begin{split}
    \p_t r^2 = - u \cdot \nabla r^2 + \mathcal{Z},
  \end{split}
\end{equation}
where, for $r = \p_1^3 \rho$,
\begin{equation}\nonumber
\begin{split}
  \mathcal{Z} =&
    - 2\p_1^3 \rho \cdot\sum_{k = 1}^3 \mathcal{C}_3^k \p_1^k u\cdot \nabla \p_1^{3-k} \rho - 2\p_1^3\rho \p_1^3 \nabla \cdot u - 2\p_1^3\rho \p_1^3(\rho \nabla \cdot u),
\end{split}
\end{equation}
and for $r = \p_1^3 b_2$,
\begin{equation}\nonumber
\begin{split}
  \mathcal{Z} =&
    - 2\p_1^3 b_2 \cdot\sum_{k = 1}^3 \mathcal{C}_3^k \p_1^k u\cdot \nabla \p_1^{3-k} b_2 - 2\p_1^3b_2 \p_1^4 u_1 \\
    &+ 2 \p_1^3b_2 \p_1^3(b \cdot \nabla u_2) - 2\p_1^3b_2 \p_1^3(b_2 \nabla \cdot u).
\end{split}
\end{equation}
Hence, by integration by parts,
\begin{equation}\nonumber
  \begin{split}
    {W}_{1,1,3} =
    & \, 2 \int_{\mathbb{R}^2}\int_{-\infty}^{x_2} \int_{-\infty}^{x_1}\p_1 \p_2 (-\Delta)^{-1} \nabla \cdot u \cdot \Gamma \,dy \, u \cdot \nabla r^2 \,dx \\
    & -2 \int_{\mathbb{R}^2}\int_{-\infty}^{x_2} \int_{-\infty}^{x_1}\p_1 \p_2 (-\Delta)^{-1} \nabla \cdot u \cdot \Gamma \,dy \, \mathcal{Z} \,dx \\
    =
    & -2 \int_{\mathbb{R}^2}\int_{-\infty}^{x_2} \int_{-\infty}^{x_1}\p_1 \p_2 (-\Delta)^{-1} \nabla \cdot u \cdot \Gamma \,dy  \nabla \cdot u \, r^2 \,dx \\
    & - 2 \int_{\mathbb{R}^2} \int_{-\infty}^{x_1}\p_1 \p_2 (-\Delta)^{-1} \nabla \cdot u \cdot \Gamma \,dy_1  u_2 \, r^2 \,dx \\
    & - 2 \int_{\mathbb{R}^2} \int_{-\infty}^{x_2}\p_1 \p_2 (-\Delta)^{-1} \nabla \cdot u \cdot \Gamma \,dy_2  u_1 \, r^2 \,dx \\
    & -2 \int_{\mathbb{R}^2}\int_{-\infty}^{x_2} \int_{-\infty}^{x_1}\p_1 \p_2 (-\Delta)^{-1} \nabla \cdot u \cdot \Gamma \,dy \, \mathcal{Z} \,dx.
  \end{split}
\end{equation}
Now, by virtue of H\"older's inequality and Proposition \ref{prop-anisotropic-est}, we deduce that
\begin{equation}\nonumber
  \begin{split}
    {W}_{1,1,3} \lesssim& \|\nabla \cdot u\|_{L^2} \|\Gamma\|_{L^2} \|\nabla \cdot u\|_{L^\infty} \|r\|_{L^2}^2 \\
    & + \|\nabla \cdot u\|_{H^1} \|\Gamma\|_{L^2}^\frac{1}{2}\|\p_2 \Gamma\|_{L^2}^\frac{1}{2} \|u_2\|_{L^2}^\frac{1}{2}\|\nabla^2 u_2\|_{L^2}^\frac{1}{2} \|r\|_{L^2}^2 \\
    & + \|\nabla \cdot u\|_{H^1} \|\Gamma\|_{L^2}^\frac{1}{2}\|\p_1 \Gamma\|_{L^2}^\frac{1}{2} \|u_1\|_{L^2}^\frac{1}{2}\|\nabla^2 u_1\|_{L^2}^\frac{1}{2} \|r\|_{L^2}^2 \\
    & + \|\nabla \cdot u\|_{L^2} \|\Gamma\|_{L^2} \|\mathcal{Z}\|_{L^1} \\
    \lesssim& \, \|(\rho, b, u)\|_{H^3} (\|\nabla \Gamma\|_{L^2}^2 + \|\nabla u\|_{H^3}^2).
  \end{split}
\end{equation}
Here, we have used \eqref{ansatz} and the fact that each term in $\mathcal{Z}$ contains at least one factor of $\partial_i u_j$. For the remaining term $W_{1,2}$, similar to the estimate of $W_{1,1,2}$, we have
\begin{equation}\nonumber
W_{1,2} \lesssim \|u_t\|_{L^2} \|\p_2 \Gamma\|_{L^2} \|r\|_{L^2}^2.
\end{equation}
Synthesizing the above estimates for ${W}_1$, we obtain
\begin{equation}\label{W1}
  \begin{split}
{W}_1
   \lesssim& \sup_{0 \leq \tau \leq t} \Big |\int_{\mathbb{R}^2}\int_{-\infty}^{x_2} \int_{-\infty}^{x_1}\p_1 \p_2 (-\Delta)^{-1} \nabla \cdot u  \Gamma \,dy \, r^2\,dx \Big| \\
  & + \sup_{0 \leq \tau \leq t} \|(\rho, b, u)\|_{H^3} \cdot \int_0^t ( \|\p_2 b\|_{H^2}^2 + \|(\Gamma_t, \nabla \Gamma)\|_{L^2}^2 + \|\nabla u\|_{H^3}^2) \, d\tau \\
  \lesssim& \sup_{0 \leq \tau \leq t} \|\nabla \cdot u\|_{L^2}  \|\Gamma\|_{L^2} \|(\p_1^3 b_2, \p_1^3\rho)\|_{L^2}^2 \\
  & +  \sup_{0 \leq \tau \leq t} \|(\rho, b, u)\|_{H^3} \cdot \int_0^t ( \|\p_2 b\|_{H^2}^2 + \|(\Gamma_t, \nabla \Gamma)\|_{L^2}^2 + \|\nabla u\|_{H^3}^2) \, d\tau \\
  \lesssim& \sup_{0 \leq \tau \leq t} \|(\rho, b, u)\|_{H^3}^3\\
  & +  \sup_{0 \leq \tau \leq t} \|(\rho, b, u)\|_{H^3} \cdot \int_0^t ( \|\p_2 b\|_{H^2}^2 + \|(\Gamma_t, \nabla \Gamma)\|_{L^2}^2 + \|\nabla u\|_{H^3}^2) \, d\tau.
  \end{split}
\end{equation}
Here, we have used \eqref{ansatz}, which implies $\|\Gamma\|_{L^2} \lesssim 1$.

$\bullet$ \textbf{Estimate of ${W}_2$:}
Executing integration by parts,
\begin{equation}\nonumber
  \begin{split}
    {W}_2 =& \,2 \int_{\mathbb{R}^2}\int_{-\infty}^{x_1} \p_1 \mathcal{V} \Gamma \, d y_1 \, r^2 \,dx\\
    =& \,2\int_{\mathbb{R}^2}  \mathcal{V} \Gamma  r^2 \,dx  -2\int_{\mathbb{R}^2} \int_{-\infty}^{x_1} \mathcal{V} \p_1 \Gamma \, d y_1  r^2 \,dx.
  \end{split}
\end{equation}
By Young's inequality and H\"older's inequality, we have
\begin{equation}\nonumber
  \begin{split}
    {W}_2 \leq& \,\frac{1}{2}\int_{\mathbb{R}^2}  \Gamma^2 r^2 \,dx + 8 \int_{\mathbb{R}^2}  \mathcal{V}^2 r^2 \,dx + C\|\mathcal{V}\|_{L_{x_1}^2 L_{x_2}^\infty} \|\p_1 \Gamma\|_{L_{x_1}^2 L_{x_2}^\infty} \|r\|_{L^2}^2 \\
    \leq& \,\frac{1}{2}\int_{\mathbb{R}^2}  \Gamma^2 r^2 \,dx + C \|\mathcal{V}\|_{\dot H^{1-\delta} \cap \dot H^{1+\delta}}^2 \|r\|_{L^2}^2 \\
    & + C (\|\mathcal{V}_1\|_{L^2} \|\p_2 \mathcal{V}_1\|_{L^2}+ \|\mathcal{V}_2\|_{L^2} \|\p_2 \mathcal{V}_2\|_{L^2}+ \|\mathcal{V}_3\|_{\dot H^{\frac{1}{2}-\delta} \cap \dot H^{\frac{1}{2} + \delta}}+ \|\p_1\Gamma\|_{L^2}\|\p_1 \p_2\Gamma\|_{L^2})\|r\|_{L^2}^2.
  \end{split}
\end{equation}
Recalling the definitions of $\Gamma$ in \eqref{defgamma} and $\mathcal{V}$ in \eqref{defv}, and arguing similarly to \eqref{est-H}, we find that for sufficiently small $\delta$, there holds
\begin{equation}
  \begin{split}
    \|\mathcal{V}\|_{\dot H^{1-\delta} \cap \dot H^{1+\delta}} \lesssim& \|\nabla \cdot u\|_{H^2} + \|(u \cdot \nabla u, b \cdot \nabla b, \nabla|b_1|^2) \|_{\dot H^{-\delta} \cap \dot H^{\delta}} \\
    & + \|\frac{\rho}{\rho + 1} ( \Omega + \Delta u_1)\|_{\dot H^{-\delta} \cap \dot H^{\delta}} \\
    &+ \|\frac{\rho}{\rho + 1} (-\p_2 P + \Delta u_2 + b \cdot \nabla  b_2 - \frac{1}{2} \p_2 |b|^2)\|_{\dot H^{-\delta} \cap \dot H^{\delta}}, \\
    \lesssim& \|(\nabla u, \p_2 b, \p_2 \rho, \Omega)\|_{H^2},\\
    \|\mathcal{V}_1\|_{L^2}\|\p_2 \mathcal{V}_1\|_{L^2} \lesssim& \|\nabla \cdot u\|_{H^1}^2 ,\\
    \|\mathcal{V}_2\|_{L^2}\|\p_2 \mathcal{V}_2\|_{L^2} \lesssim& \|(u \otimes u, b \otimes b)\|_{L^2}\|\p_2 (u \otimes u, b \otimes b)\|_{L^2} \\
    \lesssim& \|(\p_2 u, \p_2 b)\|_{H^2}^2,\\
  \end{split}
\end{equation}
and
\begin{equation}
  \begin{split}
    \|\mathcal{V}_3\|_{\dot H^{\frac{1}{2}-\delta} \cap \dot H^{\frac{1}{2} + \delta}} \lesssim& \|\frac{\rho}{\rho + 1}(\Omega + \Delta u_1 )\|_{\dot H^{-\frac{1}{2}-\delta} \cap \dot H^{-\frac{1}{2} + \delta}}\\
     &+ \|\frac{\rho}{\rho + 1}(-\p_2 P + \Delta u_2 + b \cdot \nabla  b_2 - \frac{1}{2} \p_2 |b|^2 )\|_{\dot H^{-\frac{1}{2}-\delta} \cap \dot H^{-\frac{1}{2} + \delta}} \\
    \lesssim& \|(\Omega, \nabla u, \p_2 \rho, \p_2 b)\|_{H^2}.
  \end{split}
\end{equation}
Here, we have used \eqref{ansatz}. Hence,
\begin{equation}\label{W2}
  \begin{split}
   {W}_2
    \leq& \frac{1}{2}\int_{\mathbb{R}^2}  \Gamma^2 r^2 \,dx  + C\|(\rho, b, u)\|_{H^3}^2\|(\p_2\rho, \p_2 b, \nabla u, \Omega)\|_{H^2}^2.
  \end{split}
\end{equation}
Aggregating the estimates for ${W}_1$ and ${W}_2$, namely \eqref{W1} and \eqref{W2}, we successfully complete the proof of this lemma.
\end{proof}

\section{Dissipation estimates for time derivatives}
\label{estimate-time-derivative}

\begin{lemma}
For the energies defined in \eqref{energy-set}, the following estimate holds for all $t > 0$:
  \begin{equation}\nonumber
  \begin{split}
    \mathcal{E}_6(t) \triangleq \int_0^t \|\p_t(\rho, b, u, \Omega, \Gamma)\|_{L^2}^2 \,d\tau
     \lesssim \sum_{i = 0}^3 \mathcal{E}_i(t).
  \end{split}
  \end{equation}
\end{lemma}
\begin{proof}
To establish the desired bound for the temporal dissipation, we proceed by estimating each time derivative term individually.

$\bullet$ \textbf{Estimate of $\p_t\rho$:}
Invoking the mass conservation law (continuity equation), we have the structural relation
  \[
  \p_t \rho = - \nabla \cdot u - \nabla \cdot (\rho u).
  \]
A direct application of the anisotropic bounds from Proposition \ref{prop-anisotropic-est} furnishes
  \begin{equation}\nonumber
  \begin{split}
    \int_0^t \|\p_t \rho\|_{L^2}^2 \,d\tau \lesssim& \int_0^t \|\nabla u\|_{L^2}^2 \,d\tau + \sup_{0 \leq \tau \leq t} \|(\rho, u)\|_{H^1}^2 \int_0^t \|(\p_2 \rho, \p_1 u)\|_{H^1}^2 \,d\tau.
  \end{split}
  \end{equation}

$\bullet$ \textbf{Estimate of $\p_t b$:}
Recalling the evolution equation for the magnetic field (induction equation), we observe that
  \[
  b_t = \nabla^{\bot}u_1 - u\cdot\nabla b + b\cdot\nabla u - b\nabla\cdot u.
  \]
  Consequently, there is
  \begin{equation}\nonumber
  \begin{split}
    \int_0^t \|\p_t b\|_{L^2}^2 \,d\tau \lesssim& \int_0^t \|\nabla^{\bot}u_1\|_{L^2}^2 \,d\tau + \sup_{0 \leq \tau \leq t} \|(b, u)\|_{H^1}^2 \int_0^t \|(\p_2 b, \p_1 u)\|_{H^1}^2 \,d\tau.
  \end{split}
  \end{equation}

$\bullet$ \textbf{Estimate of $\p_t u$:}
Due to the distinct dissipative structures associated with the horizontal and vertical velocity components, we analyze $\p_t u_1$ and $\p_t u_2$ separately.

\quad $\diamond$ \textbf{Estimate of $\p_t u_1$:}
Turning to the momentum equation for the first velocity component, we recall that
  \begin{equation}\nonumber
  \begin{split}
    \partial_t u_1 + u \cdot \nabla u_1 - \Delta u_1 - \Omega = -\frac{\rho}{\rho+1}\big(\Delta u_1 + \Omega \big).
  \end{split}
  \end{equation}
Appealing to H\"older's inequality and Proposition \ref{prop-nonlinear-func}, we readily deduce
  \begin{equation}\nonumber
  \begin{split}
    \int_0^t \|\p_t u_1\|_{L^2}^2 \,d\tau \lesssim& \sup_{0 \leq \tau \leq t} (1+\|\rho\|_{L^\infty}^2) \cdot \int_0^t \|(\Delta u_1, \Omega)\|_{L^2}^2 \,d\tau + \sup_{0 \leq \tau \leq t} \|u\|_{L^\infty}^2 \int_0^t \|\nabla u_1\|_{L^2}^2 \,d\tau.
  \end{split}
  \end{equation}

\quad $\diamond$ \textbf{Estimate of $\p_t u_2$:}
For the second component, the momentum equation yields
  \[ \p_t u_2 = - u \cdot \nabla u_2 + \frac{1}{\rho + 1} \big(\Delta u_2 - \p_2 P + b \cdot \nabla b_2 - \tfrac{1}{2} \p_2 |b|^2\big). \]
  This structural relation directly implies that
  \begin{equation}\nonumber
  \begin{split}
    \int_0^t \|\p_t u_2\|_{L^2}^2 \,d\tau \lesssim& \sup_{0 \leq \tau \leq t} \|u\|_{L^\infty}^2 \int_0^t \|\nabla u\|_{L^2}^2 \, d\tau +  \int_0^t \|(\Delta u, \p_2 \rho)\|_{L^2}^2 \, d\tau\\
     &+ \sup_{0 \leq \tau \leq t} \|b\|_{L^\infty}^2\int_0^t \|\p_2 b\|_{L^2}^2 \, d\tau.
  \end{split}
  \end{equation}
Aggregating the estimates for $\p_t u_1$ and $\p_t u_2$, and invoking the \textit{a priori} bound \eqref{ansatz}, we obtain
  \begin{equation}\nonumber
  \begin{split}
    \int_0^t \|\p_t u\|_{L^2}^2 \,d\tau \lesssim& \int_0^t \big(\|(\p_2\rho, \p_2 b, \Omega)\|_{L^2}^2 + \|\nabla u\|_{H^1}^2\big)\, d\tau.
  \end{split}
  \end{equation}

$\bullet$ \textbf{Estimate of $\p_t \Omega$:}
Exploiting the evolution identity for the effective viscous flux $\Omega$ given in \eqref{eqomega},
    \begin{equation}\nonumber
    \begin{split}
     \Omega_t = - u \cdot \nabla \Omega + 2 \partial_1^2 u_1 + \partial_2^2 u_1 + \partial_1\partial_2 u_2 + Q,
    \end{split}
    \end{equation}
    it is straightforward to bound this term by
    \begin{equation}\nonumber
    \begin{split}
     \int_0^t \|\p_t \Omega\|_{L^2}^2 \,d\tau \lesssim& \sup_{0 \leq \tau \leq t} \|u\|_{L^\infty} \cdot \int_0^t \|\nabla \Omega\|_{L^2}^2 \, d\tau + \int_0^t \|(\nabla^2 u, Q)\|_{L^2}^2 \, d\tau.
    \end{split}
    \end{equation}
   Noting that every term in $Q$ contains at least one factor of $\partial_i u_j$, it follows from \eqref{ansatz} that
    \begin{equation}\nonumber
    \begin{split}
     \int_0^t \|\p_t \Omega\|_{L^2}^2 \,d\tau \lesssim& \int_0^t \|\nabla \Omega\|_{L^2}^2 \, d\tau + \int_0^t \|\nabla u\|_{H^3}^2 \, d\tau.
    \end{split}
    \end{equation}

$\bullet$ \textbf{Estimate of $\p_t \Gamma$:}
Differentiating the definition of the modified effective flux $\Gamma$ in \eqref{defgamma} with respect to time, we find
  \[ \p_t \Gamma = \p_t b_2 + P'(\rho + 1) \p_t \rho + b_2 \p_t b_2, \]
which immediately leads to the bound
   \begin{equation}\nonumber
    \begin{split}
     \int_0^t \|\p_t \Gamma\|_{L^2}^2 \,d\tau \lesssim& \sup_{0\leq \tau \leq t} (1 + \|(\rho, b_2)\|_{L^\infty}^2) \int_0^t \|(\p_t b_2, \p_t \rho)\|_{L^2}^2 \, d\tau.
    \end{split}
    \end{equation}

Substituting the previously established bounds into the inequality above, and synthesizing all the temporal estimates derived in this section under the \textit{a priori} assumption \eqref{ansatz}, we successfully conclude the proof of the lemma.
\end{proof}

\section{Proof of Theorem \ref{thm1}}

This section brings together the energy estimates derived in the previous lemmas to finalize the proof of Theorem \ref{thm1}.

\begin{proof}[Proof of Theorem \ref{thm1}]
The local-in-time well-posedness of \eqref{mhd1} in $H^3(\mathbb{R}^2)$ follows from standard theory (see, e.g., \cite{MaBe}). It remains to establish a uniform global bound for $(\rho, u, b)$ in $H^3(\mathbb{R}^2)$. To this end, we employ a standard bootstrapping argument (see, e.g., \cite{Tao}). Collecting the estimates from the preceding lemmas, we obtain
\begin{equation}\nonumber
\begin{split}
   \mathcal{E}_0(t) &\triangleq \sup_{0 \leq \tau \leq t} \big(\|u(\tau)\|_{H^3}^2 + \|b(\tau)\|_{H^3}^2 + \|\rho(\tau)\|_{H^3}^2\big) + \int_0^t \|\nabla u(\tau)\|_{H^3}^2 \,d\tau \\
   &\lesssim \mathcal{E}_4(t) + \mathcal{E}_5(t)+\mathcal{E}_{\mathrm{total}}^\frac{3}{2}(t),\\
   \mathcal{E}_1(t) &\triangleq  \int_0^t \|\partial_2 \rho\|_{H^2}^2 \, d\tau \lesssim \mathcal{E}_0(t) + \mathcal{E}_4(t) + \mathcal{E}_{\mathrm{total}}^\frac{3}{2}(t),\\
   \mathcal{E}_2(t) &\triangleq \int_0^t \|\partial_2 b\|_{H^2}^2 \, d\tau \lesssim \mathcal{E}_0(t) + \mathcal{E}_1(t) + \mathcal{E}_{\mathrm{total}}^\frac{3}{2}(t),\\
   \mathcal{E}_3(t) &\triangleq \int_0^t \big(\|\Omega\|_{H^2}^2 + \|\nabla \Gamma\|_{L^2}^2\big) \, d\tau \lesssim \mathcal{E}_0(t) + \mathcal{E}_1(t) + \mathcal{E}_2(t)+ \mathcal{E}_{\mathrm{total}}^\frac{3}{2}(t),\\
   \mathcal{E}_4(t) &\triangleq \int_0^t \int_{\mathbb{R}^2} (|b_1|^2 + |u_1|^2)|\partial_1^3 b_2|^2 \,dx \, d\tau \lesssim \mathcal{E}_{\mathrm{total}}^2(t),\\
   \mathcal{E}_5(t) &\triangleq \int_0^t \int_{\mathbb{R}^2} \Gamma^2 |\p_1^3 (\rho, b_2)|^2 \, dx d\tau \lesssim  \mathcal{E}_{\mathrm{total}}^\frac{3}{2}(t),\\
   \mathcal{E}_6(t) &\triangleq \int_0^t \|\p_t(\rho, b, u, \Omega, \Gamma)\|_{L^2}^2 \,d\tau \lesssim \sum_{i = 0}^3 \mathcal{E}_i(t).
\end{split}
\end{equation}
Multiplying each inequality by an appropriate coefficient and summing them, we deduce that for some absolute constant $C^* > 0$,
\begin{equation}\label{ee}
	\mathcal{E}_{\mathrm{total}}(t) \leq C^* \|(\rho_0, u_0, b_0)\|_{H^3}^2+ C^* \mathcal{E}_{\mathrm{total}}^\frac{3}{2}(t).
\end{equation}
\vskip .1in
We now apply a bootstrapping argument to \eqref{ee}.
Suppose that the following \textit{a priori} bound holds:
\begin{equation}\label{an}
\mathcal{E}_{\mathrm{total}}(t) \le 4 C^* \|(\rho_0, u_0, b_0)\|_{H^3}^2.
\end{equation}
Substituting this ansatz into \eqref{ee} yields
\[
\mathcal{E}_{\mathrm{total}}(t) \le C^*\|(\rho_0, u_0, b_0)\|_{H^3}^2 + 8(C^*)^\frac{5}{2}\|(\rho_0, u_0, b_0)\|_{H^3}^3.
\]
Assume that the initial data
$(\rho_0, u_0, b_0) \in H^3(\mathbb{R}^2)$ is sufficiently small, i.e.,
$$\|(\rho_0, u_0, b_0)\| \leq \epsilon,$$
which simplifies to
\begin{equation}\label{cc}
\mathcal{E}_{\mathrm{total}}(t) \le 2 C^* \epsilon^2.
\end{equation}
Since the bound in \eqref{cc} is strictly sharper than the ansatz \eqref{an}, the standard bootstrapping argument implies that \eqref{cc} holds globally for all $t > 0$. Consequently, we obtain the desired uniform $H^3$ bound for $(\rho, u, b)$.

Finally, by choosing $C^*$ sufficiently large, we ensure $\mathcal{E}_{\mathrm{total}}(t) < \frac{1}{2}$, which validates the \textit{a priori} assumption \eqref{ansatz}. And This completes the proof of Theorem \ref{thm1}.
\end{proof}

\vskip .2in
\section*{Acknowledgement}
Zhu was partially supported by the National Natural Science Foundation of China (NNSFC) under grant 12271160.

~\\
~\\

\noindent{\bf{Data availability:}} No data was used for the research described in the article.
~\\
~\\
\noindent{\bf{Conflict of interest:}} The authors declare no conflicts of interest.


\vskip .3in
\begin{thebibliography}{89}

	
\bibitem{in1} H. Abidi and P. Zhang, On the global solutions of 3D MHD system with initial data near equilibrium, {\it Commun. Pure. Appl. Math.}, {\bf 70}(8) (2017), 1509--1561.
	
\bibitem{Alex} A. Alexakis, Two-dimensional behavior of three-dimensional magnetohydrodynamic
flow with a strong guiding field, {\it Phys. Rev. E \bf 84} (2011), 056330.

\bibitem{Alf} H.  Alfv\'{e}n, Existence of electromagnetic-hydrodynamic waves, {\it  Nature \bf 150} (1942), 405--406.
	
\bibitem{in2} C. Bardos, C. Sulem and P. Sulem, Longtime dynamics of a conductive fluid in the presence of a strong magnetic field, {\it Trans. Am. Math. Soc.}, {\bf 305} (1988), 175--191.

\bibitem{Bis} D. Biskamp, {\it Nonlinear Magnetohydrodynamics}, Cambridge University Press, Cambridge, 1993.

\bibitem{Bur} P. Burattini, O. Zikanov and B. Knaepen, Decay of magnetohydrodynamic turbulence at
low magnetic Reynolds number, {\it J. Fluid Mech. \bf 657} (2010), 502--538.

\bibitem{in3} Y. Cai and Z. Lei, Global well-posedness of the incompressible magnetohydrodynamics, {\it Arch. Ration. Mech. Anal.}, {\bf 228} (2018), 969--993.

\bibitem{CW1} C. Cao, D. Regmi and J. Wu, The 2D MHD equations with honrizontal dissipation and horizontal magnetic diffusion, {\it J. Differential Equations} {\bf 254} (2013), 2661--2681.

\bibitem{chemin} J. Chemin, D. S. McCormick, J. C. Robinson and J. L. Rodrigo,
Local existence for the non-resistive MHD equations in Besov spaces, {\it Adv. Math. \bf 286} (2016), 1-31.

\bibitem{com1} G. Chen and D. Wang, Gloabl solutions of nonlinear magnetohydrodynamics with large initial data, {\it J. Differential Equations} {\bf 182} (2002), 344--376.


\bibitem{Davi0} P.A. Davidson,  Magnetic damping of jets and vortices, {\it J. Fluid Mech. \bf 299} (1995), 153--186.

\bibitem{Davi1} P.A. Davidson,  The role of angular momentum in the magnetic damping of turbulence,
{\it J. Fluid Mech. \bf 336} (1997), 123-150.

\bibitem{Davi} P.A. Davidson, {\it An Introduction to Magnetohydrodynamics},  Cambridge University Press, Cambridge, England, 2001.


\bibitem{DZ} W. Deng and P. Zhang, Large Time Behavior of Solutions to 3-D MHD System with Initial
Data Near Equilibrium, {\it Arch. Ration. Mech. Anal. \bf 230} (2018), 1017-1102.

\bibitem{dong} B. Dong, J. Wu and X. Zhai,  Global small solutions to a special 21
2-D compressible viscous non-resistive MHD system, {\it J. Nonlinear Sci. \bf 33} (2023) 21.


\bibitem{feffer} C. L. Fefferman, D. S. McCormick, J. C. Robinson and J. L. Rodrigo,
Higher order commutator estimates and local existence for the non-resistive MHD equations and
related models, {\it J. Funct. Anal. \bf 267} (2014) 1035-1056.

\bibitem{feffer2} C. L. Fefferman, D. S. McCormick, J. C. Robinson and J. L. Rodrigo,
Local existence for the non-resistive MHD equations in nearly optimal Sobolev spaces, {\it Arch. Ration. Mech. Anal. \bf 233} (2017) 677-691.


\bibitem{Gall} B. Gallet,  M. Berhanu and N. Mordant, Influence of an external
magnetic field on forced turbulence in a swirling flow of liquid metal,
{\it Phys. Fluids \bf 21} (2009), 085107.

\bibitem{Gall2}  B. Gallet and C.R. Doering, Exact two-dimensionalization of low-magnetic-Reynolds-number flows subject to a strong magnetic field, {\it J. Fluid Mech. \bf 773} (2015), 154--177.

\bibitem{in5} L. He, L. Xu and P. Yu, On global dynamics of three dimensional magnetohydrodynamics: nonlinear stability of Alfven waves, {\it Ann. PDE}, {\bf 4} (2018),  No.5, 105 pp.



\bibitem{hong} G. Hong, X. Hou, H. Peng and C. Zhu, Global existence for a class of large solutions to
three-dimensional compressible magnetohydrodynamic equations with vacuum, {\it SIAM J. Math. Anal. \bf 49} (2017) 2409-2441.



\bibitem{com2} X. Hu and D. Wang, Global existence and large-time behavior of solutions to the three-dimensional equations of compressible magnetohydrodynamic flows,  {\it Arch. Ration. Mech. Anal.}, {\bf 197} (2010), 203--238.

\bibitem{hu} X. Hu and F. Lin, Global existence for two dimensional compressible magnetohydrodynamic flows with zero magnetic diffusivity, arXiv:1405.0274v1 [math.AP], 1 May 2014.

\bibitem{com3} X. Huang and J. Li, Serrin-type blow up criterion for viscous, compressible, and heat conducting Navier-Stokes and magnetohydrodynamics, {\it Commun. Math. Phys.}, {\bf 324} (2013), 147--171.

\bibitem{HuangXinYan} X. Huang, Z. Xin and W. Yan, Finite time blowup of strong solutions to the two dimensional MHD equations, {\it Math. Ann.}, {\bf 392} (2025), no. 2, 2365--2394.

\bibitem{jiangzhang} S. Jiang and J. Zhang, On the non-resistive limit and the magnetic boundary-layer for one-dimensional compressible magnetohydrodynamics, {\it Nonlinearity}, {\bf  30}  (2017), 3587--3612.

\bibitem{JJ} F. Jiang and S. Jiang, Nonlinear stability and instability in the Rayleigh-Taylor problem of stratified compressible MHD fluids,
{\it Calc. Var. Partial Differ. Equ. \bf 58} (2019) 29.

\bibitem{com4} S. Kawashima, Smooth global solutions for two-dimensional equations of electro-magneto-fluid dynamics, {\it Japan J. Appl. Math.} {\bf 1} (1984), 207--222.


\bibitem{HXJ} H. Li, X. Xu and J. Zhang, Global classical solutions to 3D compressible magnetohydrodynamic equations with large oscillations and vacuum, {\it SIAM J. Math. Anal. \bf 45} (2013) 1356-1387.

\bibitem{lisun}Y. Li and Y. Sun, Global weak solutions and long time behavior for 1D compressible
MHD equations without resistivity, {\it J. Math. Phys. \bf 60} (2019) 071511.

\bibitem{lisun2}
Y. Li and Y. Sun, Global weak solutions to a two-dimensional compressible MHD
equations of viscous non-resistive fluids, {\it J. Differ. Equ. \bf 267} (2019) 3827-3851.



\bibitem{LXZ} F. Lin, L. Xu and P. Zhang,   Global small solutions to 2-D incompressible MHD system,  {\it   J. Differential Equations \bf 259} (2015), 5440-5485.


\bibitem{in8} F. Lin and T. Zhang, Global small solutions to a complex fluid model in three dimensiobal,  {\it Arch. Ration. Mech. Anal.} {\bf 216} (2015), 905--920.

\bibitem{liuzhang}
Y. Liu and T. Zhang, Global weak solutions to a 2D compressible non-resistivity MHD
system with non-monotone pressure law and nonconstant viscosity, {\it J. Math. Anal. Appl. \bf 502} (2021) 125244.


\bibitem{MaBe} A. Majda and A. Bertozzi,  {\it Vorticity and Incompressible Flow}, Cambridge University Press, 2002.


\bibitem{Pri} E. Priest and T. Forbes, {\it Magnetic Reconnection, MHD Theory and Applications}, Cambridge University Press, Cambridge, 2000.

\bibitem{in9} X. Ren, J. Wu, Z. Xiang and Z. Zhang, Global existence and decay of smooth solution for the 2-D MHD equations without magnetic diffusion, {\it J. Funct. Anal.} {\bf 267} (2014), 503--541.


\bibitem{tanwang} Z. Tan and Y. Wang, Global well-posedness of an initial-boundary value problem for viscous non-resistive MHD sys-tems, {\it SIAM J. Math. Anal.}  {\bf 50}  (2018), 1432--1470.


\bibitem{Tao} T. Tao, {\it Nonlinear Dispersive Equations: Local and Global Analysis}, CBMS Regional Conference Series in Mathematics, 106, Amercian Mathematical Society, Providence, RI: 2006.


\bibitem{in11} D. Wei and Z. Zhang, Global well-posedness of the MHD equations in a homogeneous magnetic field, {\it Anal. PDE}, {\bf 10} (2017), 1361--1406.

\bibitem{wuguochun}
G. Wu, Y. Zhang and W. Zou, Optimal time-decay rates for the 3D compressible
magnetohydrodynamic flows with discontinuous initial data and large oscillations,
{\it J. London Math. Soc. \bf 103} (2021) 817-845.

\bibitem{WuWu} J. Wu and Y. Wu, Global small solutions to the compressible 2D magnetohydrodynamic system without magnetic diffusion, {\it Adv. Math.}, {\bf 310} (2017), 759--888.

\bibitem{wuzhai}
J. Wu and X. Zhai, Global small solutions to the 3D compressible viscous non-resistive MHD system,
{\it Math. Models Methods Appl. Sci. \bf 33} (2023), no. 13, 2629-2656.

\bibitem{wuzhu}
J. Wu ans Y. Zhu,
Global well-posedness for 2D non-resistive compressible MHD system in periodic domain,
{\it J. Funct. Anal. \bf 283} (2022), no. 7, Paper No. 109602, 49 pp.


\bibitem{zhong}
 X. Zhong, On local strong solutions to the 2D Cauchy problem of the compressible non-resistive magnetohydrodynamic equations with vacuum,
 {\it J. Dyn. Differ. Equ. \bf 32} (2020) 505-526.




























\end{thebibliography}
\end{document}